\documentclass{amsart}
\usepackage{amssymb, graphics, color, enumitem, mathrsfs}
\usepackage[all]{xy}

\usepackage{tikz}

 \usepackage[applemac]{inputenc}
\usepackage[cyr]{aeguill}

\usepackage[margin = 1in]{geometry}

\newtheorem{lemma}{Lemma}[section]
\newtheorem{proposition}[lemma]{Proposition}
\newtheorem{corollary}[lemma]{Corollary}
\newtheorem{theorem}[lemma]{Theorem}

\newtheorem{example}[lemma]{Example}
\newtheorem{definition}[lemma]{Definition}
\newtheorem{remark}[lemma]{Remark}

\newtheorem*{Acknowledgement}{Acknowledgements}

\newtheorem{Theorem}{Theorem}
\newtheorem*{Remark}{Remark}
\newtheorem{assumption}[lemma]{Assumption}
\newtheorem{Corollary}[Theorem]{Corollary}

\newcommand\cf{cf\@. }

\newcommand\pa{ \partial}

\newcommand\bbC{\mathbb C}

\newcommand\bbN{\mathbb N}

\newcommand\bbR{\mathbb R}
\newcommand\bbS{\mathbb S}

\renewcommand\Re{\operatorname{Re}}

\usepackage{color}

\newcommand\CI{\mathcal{C}^{\infty}}

\newcommand\fc{\operatorname{fc}}
\newcommand\Diff{\operatorname{Diff}}

\newcommand\cC{\mathcal{C}}
\newcommand\cA{\mathcal{A}}
\newcommand\cE{\mathcal{E}}
\newcommand\cD{\mathcal{D}}
\newcommand\cF{\mathcal{F}}

\newcommand\cV{\mathcal{V}}
\newcommand\cU{\mathcal{U}}

\newcommand\End{\operatorname{End}}

\newcommand\pr{\operatorname{pr}}

\newcommand\phg{\operatorname{phg}}
\newcommand\Spec{\operatorname{Spec}}
\newcommand\Id{\operatorname{Id}}

\renewcommand\sc{\operatorname{sc}}

\newcommand\lb{\operatorname{lb}}
\newcommand\rb{\operatorname{rb}}
\newcommand\fb{\operatorname{bf}}
\newcommand\ff{\operatorname{ff}}

\newcommand\cN{\mathcal{N}}

\newcommand\FB{\operatorname{FB}}
\newcommand\AC{\operatorname{AC}}

\newcommand\SU{\operatorname{SU}}

\newcommand\QAC{\operatorname{QAC}}

\newcommand\QFB{\operatorname{QFB}}

\newcommand\ALE{\operatorname{ALE}}
\newcommand\QALE{\operatorname{QALE}}

\newcommand\cK{\mathcal{K}}
\newcommand\sus{\operatorname{sus}}

\newcommand\Hom{\operatorname{Hom}}

\newcommand\cS{\mathcal{S}}

\newcommand\cQ{\mathcal{Q}}
\newcommand\cR{\mathcal{R}}

\newcommand\cG{\mathcal{G}}

\newcommand{\cO}{\mathcal{O}}
\newcommand{\lf}{\operatorname{lf}}
\newcommand{\rf}{\operatorname{rf}}
\newcommand{\fbf}{\phi\!\operatorname{bf}}
\newcommand{\tf}{\operatorname{tf}}
\newcommand{\zf}{\operatorname{zf}}
\newcommand{\Crit}{\operatorname{Crit}}
\newcommand{\cl}{\operatorname{cl}}
\newcommand{\cW}{\mathcal{W}}
\newcommand{\ALF}{\operatorname{ALF}}
\newcommand{\ALC}{\operatorname{ALC}}

\begin{document}
\title[Low energy limit for $\FB$-operators]
{Low energy limit for the resolvent of some fibered boundary operators}

\author{Chris Kottke}
\address{New College of Florida}
\email{ckottke@ncf.edu}

\author{Fr\'ed\'eric Rochon}
\address{Département de Mathématiques, Universit\'e du Qu\'ebec \`a Montr\'eal}
\email{rochon.frederic@uqam.ca}

\maketitle

\begin{abstract}
For certain   Dirac operators $\eth_{\phi}$ associated to a fibered boundary metric $g_{\phi}$, we provide a pseudodifferential characterization of the limiting behavior of $(\eth_{\phi}+k\gamma)^{-1}$ as $k\searrow 0$, where $\gamma$ is a self-adjoint operator anti-commuting with $\eth_{\phi}$ and whose square is the identity.  This yields in particular a pseudodifferential characterization of the low energy limit of the resolvent of $\eth_{\phi}^2$, generalizing a result of Guillarmou and Sher about the low energy limit of the resolvent of the Hodge Laplacian of an asymptotically conical metric.  As an application, we use our result to give a pseudodifferential characterization of the inverse of some suspended version of the operator $\eth_{\phi}$.  One important ingredient in the proof of our main theorem is that the Dirac operator $\eth_{\phi}$ is Fredholm when acting on suitable weighted Sobolev spaces.  This result has been known to experts for some time and we take this as an occasion to provide a complete explicit proof. \end{abstract}

\tableofcontents

\section*{Introduction}

An important class of complete non-compact Riemannian metrics with bounded geometry is the one of asymptotically conical metrics ($\AC$-metrics).  Those consist in metrics asymptotically modelled on the infinite end of a Riemannian cone $(\cC,g_{\cC})$, where $\cC=(0,\infty)\times Y$ with $(Y,g_Y)$ a closed Riemannian manifold and 
$$
    g_{\cC}= dr^2+ r^2g_{Y}.
$$
When $Y=\bbS^n$ and $g_Y$ is the standard metric $g_{\bbS^n}$, we get the more restricted class of asymptotically Euclidean metrics ($\operatorname{AE}$-metrics).  If $Y=\bbS^n/\Gamma$ for $\Gamma\subset O(n+1,\bbR)$ a finite subgroup and $g_Y$ is the quotient of the standard metric $g_{\bbS^n}$, this corresponds to the slightly larger subclass of asymptotically locally Euclidean metrics ($\ALE$-metrics).  In general relativity, $\operatorname{AE}$-metrics play an important role in the formulation of the positive mass theorem, while many important examples of gravitational instantons are $\ALE$-metrics \cite{Kronheimer1989}.  More generally, there are many examples of asymptotically conical Calabi-Yau metrics \cite{Tian-Yau,Joyce,Goto, vanC,CH2013,CH2015}.  

In terms of scattering theory and spectral theory, $\AC$-metrics constitute a natural generalization of the Euclidean space.  In that respect and compared to other types of geometries like asymptotically hyperbolic metrics \cite{MM1987}, meromorphic continuations of the resolvent of the Laplacian are hard to obtain.  For instance, the meromorphic continuation of the resolvent obtained by Wunsch-Zworski \cite{Wunsch-Zworski} is in a conic neighborhood of the continuous spectrum, which as expected from \cite[\S~6.10]{MelroseGST}, does not include an open neighborhood of $0$.  It is possible however to give a description of the asymptotic behavior of the resolvent $(\Delta-\lambda^2)^{-1}$ of the Laplacian $\Delta$ of an $\AC$-metric when $\lambda=ik$ is in the imaginary axis and $k\searrow 0$.  Since $k^2$ has the interpretation of an energy in quantum mechanics, this asymptotic behavior is often referred to as a low energy limit and is in some sense  the opposite of the semiclassical limit, which consists instead to study what happens when $k^2$ tends to infinity.

More precisely, developing and using a pseudodifferential calculus initially considered by Melrose and Sa Barreto, Guillarmou and Hassell, in a series of two papers \cite{GH1,GH2}, provided a fine pseudodifferential characterization of the low energy limit 
$$
  \lim_{k\to0+} (\Delta+k^2)^{-1}
$$
of the resolvent  of the Laplacian of an $\AC$-metric and used it to obtain results about the boundedness of the Riesz transform.  In \cite{Sher2013,Sher2015},  Sher used instead this description of the low energy limit to give a precise description of the long time asymptotic to the heat kernel of an $\AC$-metric and to study the behavior of the regularized  determinant of the Laplacian under conic degenerations.  All these results were subsequently generalized by Guillarmou-Sher \cite{GS} to the setting where the scalar Laplacian is replaced by the  Hodge Laplacian,  allowing them in particular to describe the limiting behavior of analytic torsion under conic degenerations.  

A somewhat related setting where the pseudodifferential characterization of the low energy limit of the resolvent is used is in the Cheeger-M\"uller theorem for wedge metrics obtained in \cite{ARS3}.  Indeed, the overall strategy of \cite{ARS3} was to describe the limiting behavior of analytic torsion for a family of closed Riemannian metrics degenerating to a wedge metric.  In particular,  this required a uniform description of the resolvent of the Hodge Laplacian under such a degeneration, which in turn could be obtained provided one could invert a model operator of the form
\begin{equation}
P= \Delta_{\AC}+ \Delta_E
\label{int.1}\end{equation}
with $\Delta_{\AC}$ the Hodge Laplacian of an $\AC$-metric and $\Delta_E$ the Euclidean Laplacian on $\bbR^q$.  But taking the Fourier transform of \eqref{int.1} in the Euclidean factor yields 
\begin{equation}
    \Delta_{\AC}+|\xi|^2,
\label{int.2}\end{equation}
whose inverse, setting $k:=|\xi|$, can be described all the way down to $|\xi|=0$ thanks to the results of \cite{GH1,GH2,GS}.  This is precisely what was needed to take the inverse Fourier transform of \eqref{int.2} and obtained a pseudodifferential characterization of the inverse $(\Delta_{\AC}+\Delta_E)^{-1}$ fitting exactly where it should in the wedge-surgery double space of \cite{ARS3}.

Another setting where the model operator \eqref{int.1} naturally arises is in the study of the Hodge Laplacian of a quasi-asymptotically conical metric ($\QAC$-metric).  This type of metrics was introduced by Degeratu-Mazzeo \cite{DM2018} as a generalization of the quasi-asymptotically locally Euclidean metrics ($\QALE$-metrics) of Joyce \cite{Joyce}.  Without entering in the fine details of the definition of such metrics, let us say that one of the simplest non-trivial example of such metric is a Cartesian product of two $\AC$-metrics, so that \eqref{int.1} can be seen indeed as a Hodge Laplacian associated to a $\QAC$-metric.  

Having in mind this sort of application, the purpose of this paper is to generalize the pseudodifferential characterization of the low energy limit of the resolvent of \cite{GH1,GH2,GS} in two different directions:
\begin{enumerate}
\item[(i)] Characterize the limit as $k\searrow 0$ when $(\Delta+k^2)^{-1}$ is replaced by 
$$
     (\eth+k\gamma)^{-1},
$$
where $\eth$ is a Dirac operator and $\gamma$ is a self-adjoint operator of order $0$ such that
$$
     \gamma^2=\Id, \quad \gamma\eth+\eth\gamma=0;
$$
\item[(ii)] Do it not only for asymptotically conical metrics, but also for the  more general class of fibered boundary metrics of \cite{Mazzeo-MelrosePhi,HHM2004}.  
\end{enumerate} 
One motivation for (i) is to characterize the inverse of an operator of the form 
\begin{equation}
 D= \eth+ \eth_E
\label{int.3}\end{equation} 
with $D$ a Dirac operator on $X\times \bbR^q$, $\eth$ a Dirac operator on $X$ (associated to an $\AC$-metric or more generally a fibered boundary metric) and $\eth_E$ a Euclidean Dirac operator on $\bbR^q$.  Indeed, in this case, taking the Fourier transform of \eqref{int.3} in the $\bbR^q$-factor yields
\begin{equation}
       \widehat{D}= \eth+ i\cl(\xi),  \quad \xi\in \bbR^q,
\label{int.4}\end{equation}    
where $\cl(\xi)$ denotes Clifford multiplication by $\xi$.  Restricting \eqref{int.4} to the half-line generated by $\eta\in\bbS^{q-1}\subset \bbR^q$, that is, setting $\xi=k\eta$ with $k\in [0,\infty)$, we obtain precisely 
\begin{equation}
   \eth+ k\gamma \quad \mbox{with} \quad \gamma=i\cl(\eta).
\label{int.5}\end{equation}
Hence, understanding $(\eth+k\gamma)^{-1}$ in the limit $k\searrow 0$ will allow, as $\eta\in \bbS^{q-1}$ varies, to understand $(\eth+i\cl(\xi))^{-1}$ as $\xi\to 0$.  
Concerning (ii), let us for the moment remind the reader that fibered boundary metrics are a natural generalization of the class of $\AC$-metrics modelled at infinity by a fiber bundle over a Riemannian cone (at least when they are product-type at infinity up to some order).  More precisely, if $M$ is a compact manifold with boundary $\pa M$,  one starts with a fiber bundle $\phi: \pa M\to Y$ whose base and fibers are closed manifolds and considers metrics $g_Y$ and $g_{\pa M}=\phi^*g_Y+\kappa$ on $Y$ and $\pa M$ making $\phi$ a Riemannian submersion.  The modelled at infinity is then $(\cC_{\phi},g_{\cC_{\phi}})$ with $\cC_{\phi}=(0,\infty)\times \pa M$ and 
\begin{equation}
  g_{\cC_{\phi}}= dr^2+ r^2\phi^*g_{Y}+ \kappa,
\label{int.6}\end{equation}
so that $\phi$ extends to a fiber bundle $\cC_{\phi}\to \cC$ over the cone $\cC=(0,\infty)\times Y$ which is a Riemannian submersion with respect to the metrics $g_{\cC_{\phi}}$ and $g_{\cC}=dr^2+r^2g_Y$ on $\cC_{\phi}$ and $\cC$.  In dimension $4$, an important class of examples is the one given by asymptotically locally flat  gravitational instantons ($\ALF$-gravitational instantons), in which case the Riemannian cone $(\cC,g_{\cC})$ has cross-section a quotient of the $2$-sphere with its standard metric and the fiber bundle $\phi$ is a circle bundle.  Those include in particular the natural hyperK\"ahler metric on the universal cover of the reduced moduli space of centered $\SU(2)$-monopoles of magnetic charge $2$.   Another important class of examples is given by the asymptotically locally conical metrics ($\ALC$-metrics) with $G_2$-holonomy of \cite{BGGG,FHN1,FHN2}, in which case $\phi$ is a circle bundle over a $5$-dimensional base.    

To formulate the main result of this paper, let $g_{\phi}$ be a fibered boundary metric which is product-type to order $2$ (in the sense of Definition~\ref{phi.7} below) on the interior of the manifold with boundary $M$.  In particular, at infinity, $g_{\phi}$ is modelled by a metric of the form \eqref{int.6}.  Let $E\to M$ be a Clifford module for the associated Clifford bundle and let $\eth_{\phi}$ be the Dirac operator associated to a choice of Clifford connection.  As explained in \S~\ref{fdt.0}, the operator $\eth_{\phi}$ naturally restricts to an elliptic family of fiberwise operators $D_v$ on the fibers of $\phi$.  We suppose that the nullspaces of the members of the family have all the same dimension and hence form a vector bundle $\ker D_v\to Y$ over $Y$.  For instance, by Hodge theory, this is always the case when $\eth_{\phi}$ is the Hodge-deRham operator associated to the metric $g_{\phi}$.  In any case, as shown in Definition~\ref{dt.6} and Lemma~\ref{dt.7} below, given such a vector bundle $\ker D_v\to Y$, there is a well-defined holomorphic family 
$$
     \lambda\mapsto I(D_b,\lambda)
$$    
of elliptic first order operators on $Y$ acting on sections of $\ker D_v$.  This family, called 
the indicial family of $\eth_{\phi}$, is invertible except for a discrete set of values that are called indicial roots.  For our result to hold, we need to assume that 
\begin{equation}
       \Re\lambda\in [-1,0] \; \Longrightarrow \; I(D_b,\lambda) \quad \mbox{is invertible}.
\label{int.7}\end{equation}
Finally, let $\gamma\in\CI(M;\End(E))$ be a self-adjoint operator such that 
\begin{equation}
    \gamma^2=\Id_E \quad \mbox{and} \quad \gamma\eth_{\phi}+ \eth_{\phi}\gamma=0.
\label{int.8}\end{equation}
\begin{Theorem}
If there is a well-defined kernel bundle $\ker D_v\to Y$ such that \eqref{int.7} and \eqref{int.8} hold, then $(\eth_{\phi}+k\gamma)^{-1}$ is an element of the pseudodifferential calculus $\Psi^*_{k,\phi}(M;E)$ of low energy fibered boundary operators introduced in \eqref{kfb.16c} below.
\label{int.9}\end{Theorem}
We refer to Theorem~\ref{le.46} below for a more precise statement of the result.  Even if we are not given $\gamma$ as in \eqref{int.8}, we can use Theorem~\ref{int.9} to obtain a corresponding result for the square of $\eth_{\phi}$.
\begin{Corollary}
If there is a well-defined kernel bundle $\ker D_v\to Y$ such that \eqref{int.7} holds, then $(\eth_{\phi}^2+k^2)^{-1}$ is an element of the pseudodifferential calculus $\Psi^*_{k,\phi}(M;E)$.  
\label{int.10}\end{Corollary}
We refer to Corollary~\ref{le.47} for a detailed statement of the result.  One advantage of using Dirac operators to derive Corollary~\ref{int.10} is that this requires slightly less control on the metric at infinity, see in particular the discussion at the end of \S~\ref{le.0} below.  The fact that the Dirac operator is of order 1 instead of order 2 also yields simplifications in the construction of the parametrix.     

When $\eth_{\phi}$ is the Hodge-deRham operator of $g_{\phi}$, we can give an alternative formulation to \eqref{int.7}.  In this case, $\ker D_v\to Y$ essentially corresponds to the vector bundle of fiberwise harmonic forms, and as such is naturally a flat vector bundle.  There is in particular an associated Hodge-deRham operator 
\begin{equation}
 \mathfrak{d}=\delta^{\ker D_v}+ d^{\ker D_v} \quad \mbox{acting on} \quad  \Omega^*(Y;\ker D_v),
\label{int.11}\end{equation}
where $d^{\ker D_v}$ is the exterior differential associated to $\ker D_v$ and $\delta^{\ker D_v}$ is its formal adjoint.
\begin{Corollary}
Let $\eth_{\phi}$ be the Hodge-deRham operator of $g_{\phi}$.  Suppose that 
\begin{equation}
\begin{gathered}
     H^q(Y;\ker D_v)=\{0\}, \quad q\in\{ \frac{h-1}2, \frac{h}2, \frac{h+1}2\}, \\
    \Spec(d^{\ker D_v}\delta^{\ker D_v}+  \delta^{\ker D_v} d^{\ker D_v})_{\frac{h}2}>\frac{3}4, \\
     \Spec (d^{\ker D_v}\delta^{\ker D_v})_{\frac{h+1}2}>1,
\end{gathered}
\label{int.12b}\end{equation}
where $h=\dim Y$, $H^q(Y;\ker D_v)$ is the de Rham cohomology group of degree $q$ associated to the flat vector bundle $\ker D_v$  and $\Spec(A)_q$ denotes the part of the spectrum of $A$ coming from forms of degree $q$.  Then the conclusion of Corollary~\ref{int.10} holds for $(\eth_{\phi}^2+k^2)^{-1}$.  Moreover, if there is $\gamma$ such that \eqref{int.8} holds, then the conclusion of Theorem~\ref{int.9}  holds for $(\eth_{\phi}+k\gamma)^{-1}$.
 \label{int.12}\end{Corollary}
 \begin{Remark}
 The authors wish to acknowledge that in a parallel work by Grieser, Talebi and Vertman \cite{GTV}, a result similar to the first part of Corollary~\ref{int.12} was obtained independently and simultaneously using partly different methods, in particular working directly with the Hodge Laplacian, and relying on a split-pseudodifferential calculus with parameter, which specifies the asymptotics of the Schwartz kernel with respect to a splitting of differential forms into fiberwise harmonic forms and their orthogonal complement.
 \end{Remark}

When $Y=\pa M$ and $\phi$ is the identity map, so that $g_{\phi}$ is in fact an $\AC$-metric, the part of Corollary~\ref{int.12} involving $(\eth^2_{\phi}+k^2)^{-1}$ corresponds to \cite[Theorem~1]{GS}, though, as discussed just after Corollary~\ref{le.48} below, our assumption \eqref{int.12b} may be  slightly less restrictive then those of \cite{GS} when $h$ is odd.  Notice also that restricting Corollary~\ref{int.12} to forms of degree $0$ gives a corresponding statement for the low energy limit of the resolvent of the scalar Laplacian, though in this case our assumption \eqref{int.12b} is probably not optimal and could possibly be improved by working directly with the scalar Laplacian.  

As suggested above, Theorem~\ref{int.9} can be used to characterize the inverse of the Dirac operator \eqref{int.3} with $\eth=\eth_{\phi}$, that is, the inverse of 
\begin{equation}
      \eth_{\sus}= \eth_{\phi} + \eth_{E}.
\label{int.14}\end{equation}
Indeed, in terms of its Fourier transform
\begin{equation}
\widehat{\eth}_{\sus}(\xi)= \eth_{\phi}+ i\cl(\xi),
\label{int.15}\end{equation} 
its inverse is given by 
$$
           (\eth_{\sus})^{-1}= \frac{1}{(2\pi)^q}\int_{\bbR^q} e^{ix\cdot \xi}(\eth_{\phi}+ i\cl(\xi))^{-1} d\xi
$$
and a closer analysis of the description of $(\eth_{\phi}+ i\cl(\xi))^{-1}$ provided by Theorem~\ref{int.9} yields the following result.
\begin{Corollary}
If $\dim Y>1$, the inverse of $\eth_{\sus}$ is a conormal distribution on a certain manifold with corners described in \eqref{sus.4b} below.   
\label{int.16}\end{Corollary}
Referring to Theorem~\ref{sus.6} below for more details, let us point out that, in agreement with the fact that $\eth_{\sus}$ is not fully elliptic, the inverse $(\eth_{\sus})^{-1}$ is not quite a suspended operator, though it can be understood as an element of an enlarged pseudodifferential suspended $\phi$-calculus.

Our main motivation for proving Corollary~\ref{int.16} is to study Dirac operators associated to yet another class metrics, namely the class of quasi-fibered boundary metrics ($\QFB$-metrics) introduced in \cite{CDR}.  Indeed, in the companion paper \cite{KR1}, we construct a parametrix for the Hodge-deRham operator of a $\QFB$-metric and one of the key steps is to use Corollary~\ref{int.16} to invert a model precisely of the form \eqref{int.14}.   As the name suggests, $\QFB$-metrics are to fibered boundary metrics what $\QAC$-metrics are to $\AC$-metrics.  According to \cite{FKS}, an important example of $\QFB$-metrics is given by the hyperK\"ahler metric on the reduced moduli space of $\SU(2)$-monopoles of charge $k$ on $\bbR^3$.  We know also from \cite{Carron2011} that the Nakajima metric on the Hilbert scheme of $n$ points on $\bbC^2$ is an example of $\QALE$ metric.  In fact, building on these results, we use the parametrix construction of \cite{KR1} to make progress in \cite{KR2} on the Sen conjecture \cite{Sen} and the Vafa-Witten conjecture \cite{Vafa-Witten}, which are conjectures from string theory and $S$-duality making predictions about the reduced $L^2$-cohomology of such moduli spaces.  

To prove our main result, the strategy, as in \cite{GH1}, is to introduce a suitable double space, $M^2_{k,\phi}$, that is, a suitable manifold with corners $M^2_{k,\phi}$ on which the Schwartz kernel of $(\eth_{\phi}+k\gamma)^{-1}$ will admit a description as a conormal distributions with polyhomogeneous expansion at the various boundary hypersurfaces of the double space.  Compared to the double space in \cite{GH1}, the main difference is that there is one more boundary hypersurface and that one other boundary hypersurface, corresponding to $\sc$ in \cite{GH1}, is slightly different in nature.  Given such a double space and the corresponding calculus of pseudodifferential operators, one important step in the construction of the inverse of $(\eth_{\phi}+k\gamma)$ is to show that $\eth_{\phi}$ is Fredholm when acting on suitable weighted Sobolev spaces.  Thanks to the thesis of Vaillant \cite{Vaillant}, which among other things derived a corresponding Fredholm result for the related geometry of fibered cusp metrics, it has been known for some time by experts that such Fredholm result holds.  In particular, a precise statement is provided in \cite[Proposition~16]{HHM2004} when $\eth_{\phi}$ is the Hodge-deRham operator.  Assuming some conditions on the metric, such result follows from a general Fredholm criterion obtained by Grieser and Hunsicker in \cite[Theorem~13]{GH2014}.  

However, since this result is central in proving our main result, we take the opportunity to provide a complete explicit proof for Dirac operators which gives a prelude of the techniques used later in the paper.  First, to extend the statement of \cite[Proposition~16]{HHM2004} to a Dirac operator $\eth_{\phi}$ with well-defined bundle $\ker D_v\to Y$, let $\Pi_h$ denote the fiberwise $L^2$-projection from fiberwise $L^2$-sections of $E\to\pa M$ onto sections of $\ker D_v\to Y$.  Let $\widetilde{\Pi}_h$ denote a smooth extension of $\Pi_h$, first to a collar neighborhood of $\pa M$, and then to all of $M$ using cut-off functions.  Let $L^2_{\phi}(M;E)$ and $H^1_{\phi}(M;E)$ be the $L^2$-space and the $L^2$-Sobolev space of order $1$ associated to the fibered boundary metric $g_{\phi}$ and a choice of bundle metric and connection for $E\to M$.  Let $H^1_b(M;E)$ be the $L^2$-Sobolev space of order $1$ associated to a choice of $b$-metric in the sense of \cite{MelroseAPS}.  Finally, let $x\in \CI(M)$ be a boundary defining function, which, near $\pa M$, corresponds to $\frac{1}r$ in terms of the model metric \eqref{int.6}. Notice that  $L^2_\phi(M;E)=x^{\frac{h+1}2}L^2_b(M;E)$, but that such a simple relation does not hold for $H^1_{\phi}(M;E)$ and $H^1_b(M;E)$.  Then \cite[Proposition~16]{HHM2004} admits the following generalization (see also Corollary~\ref{dt.24} and Corollary~\ref{dt.27} below for alternative formulations).
\begin{Theorem}
If $\delta\in \bbR$ is not a critical weight of the indicial family $I(D_b,\lambda)$ of $\eth_{\phi}$, then $\eth_{\phi}$ induces Fredholm operators
\begin{equation}
\eth_{\phi}: x^{\delta}\left( \widetilde{\Pi}_hx^{\frac{h+1}2}H^1_b(M;E)+ x(\Id-\widetilde{\Pi}_h)H^1_{\phi}(M;E) \right)\to x^{\delta+1}L^2_{\phi}(M;E)
\label{int.13a}\end{equation}
and 
\begin{equation}
\eth_{\phi}: x^{\delta}\left( \widetilde{\Pi}_hx^{\frac{h+1}2}H^1_b(M;E)+ (\Id-\widetilde{\Pi}_h)H^1_{\phi}(M;E) \right)\to x^{\delta}\left( x\widetilde{\Pi}_hL^2_{\phi}(M;E)+ (\Id-\widetilde{\Pi}_h)L^2_{\phi}(M;E)  \right).
\label{int.13b}\end{equation}
\label{int.13}\end{Theorem} 
To prove this result, our strategy, as in the thesis of Vaillant \cite{Vaillant} for fibered cusp Dirac operators, consists in constructing a sufficiently good parametrix for $\eth_{\phi}$ within the large $\phi$-calculus of \cite{Mazzeo-MelrosePhi}, see Theorem~\ref{dt.10} below for the precise statement.  Besides establishing Theorem~\ref{int.13}, our parametrix is used in Corollary~\ref{dt.22} to show that elements in the kernel of $\eth_{\phi}$ are smooth sections admitting a polyhomogeneous expansion at infinity.  More importantly, for our main result, our parametrix in Corollary~\ref{dt.31} is used to show that the inverse of \eqref{int.13a} defined on the complements of the cokernel of $\eth_{\phi}$ is a pseudodifferential operator of order $-1$ in the large $\phi$-calculus.  In particular, this inverse fits nicely on one of the boundary hypersurfaces of the double space $M^2_{k,\phi}$, allowing us to construct a good approximate inverse to $(\eth_{\phi}+k\gamma)$ within $\Psi^*_{k,\phi}(M;E)$.   

The paper is organized as follows.  In \S~\ref{phi.0}, we make a quick review of the $\phi$-calculus of Mazzeo and Melrose.  This is used in \S~\ref{fdt.0} to construct a parametrix for fibered boundary Dirac operators and derive few consequences, for instance Theorem~\ref{int.13}.  We introduce our calculus of low energy fibered boundary pseudodifferential operators in \S~\ref{kfb.0}.  After constructing a suitable triple space for our calculus in \S~\ref{ts.0}, we can describe how operators compose in \S~\ref{com.0}.  After introducing a few symbol maps in \S~\ref{sm.0}, we can finally provide the desired pseudodifferential characterization of the inverses of $(\eth_{\phi}+k\gamma)$ and $(\eth^2_{\phi}+k^2)$.  This is used in \S~\ref{sus.0} to give a pseudodifferential characterization of the inverse of the suspended operator $\eth_{\sus}$ in \eqref{int.14}.  In Appendix~\ref{cbu.0}, we establish a result about the commutativity of certain blow-ups of $p$-submanifolds that turns out to be useful in \S~\ref{kfb.0} in providing two different points of view on the double space $M^2_{k,\phi}$.

\begin{Acknowledgement}
The authors are grateful to Rafe Mazzeo for helpful conversations and two referees for detailed reports and valuable comments and suggestions. 
CK was supported by NSF Grant No.\ DMS-1811995. In addition, this material is based in part on work 
supported by the NSF
under Grant No.\ DMS-1440140 while CK was in residence at the
Mathematical Sciences Research Institute (MSRI) in Berkeley, California, during the
Fall 2019 semester.  FR was supported by NSERC and a Canada Research chair.  This project was initiated in the Fall 2019 during the program Microlocal Analysis at the MSRI.  The authors would like to thank the MSRI for its hospitality and for creating a stimulating environment for research.  
\end{Acknowledgement}

\numberwithin{equation}{section}

\section{Fibered boundary pseudodifferential operators}\label{phi.0}

In this section, we will review briefly the definitions and main properties of the $\phi$-calculus of Mazzeo-Melrose \cite{Mazzeo-MelrosePhi}.  Here and throughout the paper, we will in particular assume that the reader has some familiarity with manifolds with corners as presented in \cite{MelroseMWC}.  What we will need can be found for instance in \cite[Chapter~2]{Grieser} or\cite[\S~2]{hmm}.

Let $M$ be a compact manifold with boundary $\pa M$ equipped with a fiber bundle $\phi: \pa M\to Y$ over a closed manifold $Y$.  Let also $x\in\CI(M)$ be a boundary defining function, that is, $x>0$ on $M\setminus \pa M$, $x=0$ on $\pa M$ and $dx$ is nowhere zero on $\pa M$.  In terms of this data, the space of $\phi$-vector fields is given by 
\begin{equation}
 \cV_{\phi}(M)= \{ \xi \in \cV_b(M) \; | \; \phi_*(\xi|_{\pa M})=0, \quad \xi x\in x^2\CI(M)\},
 \label{phi.1}\end{equation}
where $\cV_b(M)$, the algebra of $b$-vector fields of \cite{MelroseAPS}, consists of smooth vector fields tangent to the boundary of $M$.  The definition of $\cV_{\phi}(M)$ depends obviously on $\phi$, but it also depends on the choice of boundary defining function $x$.  Two boundary defining functions $x_1$ and $x_2$ will give the same Lie algebra of $\phi$-vector fields if and only if the function $\left. \frac{x_1}{x_2}\right|_{\pa M}$ is constant on the fibers of $\phi: \pa M\to Y$.  In local coordinates  $(x,y_1,\ldots,y_h, z_1,\ldots z_v)$ near $\pa M$ with $(y_1,\ldots,y_h)$ coordinates on $Y$ such that $\phi$ is locally given by
\begin{equation}
          (y_1,\ldots,y_h,z_1,\ldots,z_v)\mapsto (y_1,\ldots, y_h),
\label{phi.2}\end{equation}
the space of $\phi$-vector fields is locally spanned by 
\begin{equation}
  x^2\frac{\pa}{\pa x}, x\frac{\pa}{\pa y_1}, \ldots, x\frac{\pa}{\pa y_h}, \frac{\pa}{\pa z_1},\ldots \frac{\pa}{\pa z_{v}}.
\label{phi.2}\end{equation}
By the Serre-Swan theorem, there is a corresponding vector bundle ${}^{\phi}TM\to M$, the $\phi$-tangent bundle, and a  map of vector bundles 
\begin{equation}
      a_{\phi}: {}^{\phi}TM\to TM
\label{phi.3}\end{equation}
inducing a natural  identification
\begin{equation}
            \CI(M;{}^{\phi}TM)= \cV_{\phi}(M).  
\label{phi.4}\end{equation}  
In other words, ${}^{\phi}TM\to M$ is a Lie algebroid with anchor map $a_{\phi}$.  The anchor map $a_{\phi}$ is neither injective nor surjective when restricted to the boundary $\pa M$.   The kernel of $a_{\phi}|_{\pa M}$ is in fact a vector bundle ${}^{\phi}N\pa M\to \pa M$ on $\pa M$ inducing the short exact sequence of vector bundles
\begin{equation}
\xymatrix{
 0 \ar[r] & {}^{\phi} N\pa M \ar[r] & {}^{\phi}TM|_{\pa M} \ar[r]^-{a_{\phi}} & T(\pa M/Y) \ar[r] & 0,
}
\label{phi.4b}\end{equation}
where $T(\pa M/Y)$ is the vertical tangent bundle of the fiber bundle $\phi:\pa M\to Y$.  In terms of \eqref{phi.2}, $x^2\frac{\pa}{\pa x}, x\frac{\pa}{\pa y_1}, \ldots, x\frac{\pa}{\pa y_h}$ are local sections of ${}^{\phi}N\pa M$. As explained in \cite[(7)]{Mazzeo-MelrosePhi}, there is in fact a canonical isomorphism 
\begin{equation}
       {}^{\phi}N\pa M= \phi^*({}^{\phi}NY)
\label{phi.4c}\end{equation}
for some natural vector bundle ${}^{\phi}NY\to Y$ on $Y$.

The anchor map $a_{\phi}$ induces however an isomorphism on the interior of $M$.  In the terminology of \cite{ALN04}, this means that $(M,\cV_{\phi}(M))$ is a Lie structure at infinity.  In particular, if $g_{\phi}$ is a choice of bundle metric on ${}^{\phi}TM\to M$, then it induces a Riemannian metric on $M\setminus \pa M$, also denoted $g_{\phi}$, via the isomorphism
\begin{equation}
  a_{\phi}: {}^{\phi}TM|_{M\setminus \pa M}\to T(M\setminus \pa M).
\label{phi.5}\end{equation} 
We refer to such a Riemannian metric as a \textbf{fibered boundary metric} or a \textbf{$\phi$-metric}.  By the discussion in \cite{ALN04}, such a metric is complete, of infinite volume and of bounded geometry.   
 
If $c: \pa M\times [0,\delta)\to M$ is a collar neighborhood of $\pa M$ compatible with the boundary defining function $x$ in the sense that $c^*x=\pr_2: \pa M\times [0,\delta)\to [0,\delta)$ is the projection on the second factor,  then a natural example of $\phi$-metric is given by one such that 
\begin{equation}
    c^*g_{\phi}= \frac{dx^2}{x^4}+ \frac{\phi^* g_Y}{x^2} + \kappa,
\label{phi.6}\end{equation}
where $g_Y$ is a Riemannian metric on $Y$ and $\kappa\in\CI(\pa M;S^2(T^*(\pa M))$ is a symmetric $2$-tensor such that $\phi^*g_Y+\kappa$ is a Riemannian metric on $\pa M$ making $\phi:\pa M\to Y$ a Riemannian submersion with respect to $\phi^*g_Y+\kappa$ and $g_Y$.  

\begin{definition}
A \textbf{product-type} $\phi$-metric is a $\phi$-metric $g_{\phi}$ taking the form \eqref{phi.6} in some collar neighborhood $c: \pa M\times [0,\delta)\to M$ compatible with the boundary defining function $x$. More generally, a $\phi$-metric is said to be \textbf{product-type up to order} $k\in\bbN$ if it is a product-type metric up to a term in 
$x^k\CI(M;S^2({}^{\phi}T^*M))$.    
\label{phi.7}\end{definition}
In this paper, we will exclusively work with $\phi$-metrics which are product-type up to order $2$.  An important class metrics conformally related to $\phi$-metrics is the class of fibered cusp metrics.  
\begin{definition}
A \textbf{fibered cusp metric} is a Riemannian metric $g_{\fc}$ on $M\setminus \pa M$ such that 
$$
         g_{\fc}= x^2g_{\phi}
$$
for some $\phi$-metric.  Such a metric is said to be of product-type (respectively product-type up to order $k$) if the conformally related $\phi$-metric $g_{\phi}$ is product-type (respectively product-type up to order $k$).
\label{phi.8}\end{definition}

Like a $\phi$-metric, a fibered cusp metric is complete.  However, if the fibers of $\phi$ are not $0$-dimensional, its volume is finite and it has zero injectivity radius.  Moreover, except in special cases, its curvature is not bounded.  

Within the classes of $\phi$-metrics and fibered cusp metrics, there are special subclasses corresponding to specific choices of fiber bundles $\phi: \pa M\to Y$.  One can consider for instance the case where $Y$ is a point, in which case product-type $\phi$-metrics correspond to metrics with infinite cylindrical ends, while product-type fibred cusp metrics corresponds to metrics with cusp ends.  The other extreme is to take $Y=\pa M$ and $\phi$ to be the identity map, in which case the $\phi$-vector fields correspond to the scattering vector fields of \cite{MelroseGST}, a product-type $\phi$-metric correspond to a metric with an infinite conical end and a product-type fibered cusp metric  corresponds to a metric with infinite cylindrical end.  

The differential operators geometrically constructed from a $\phi$-metric, like the Hodge Laplacian or a Dirac operator, fit in the more general class of differential $\phi$-operators.  The space $\Diff^k_{\phi}(M)$ of differential $\phi$-operators of order $k$ corresponds to differential operators generated by multiplication by an element of $\CI(M)$ and the composition of up to $k$ $\phi$-vector fields.  In other words, $\Diff^*_{\phi}(M)$ is the universal enveloping algebra of $\cV_{\phi}(M)$ with respect to $\CI(M)$.  As explained in \cite{Mazzeo-MelrosePhi,ALN04}, given vector bundles $E$ and $F$ over $M$, one can more generally define the space $\Diff^k_{\phi}(M;E,F)$ of differential $\phi$-operators of order $k$ acting from sections of $E$ to sections of $F$.  

To construct good parametrices for differential $\phi$-operators, Mazzeo and Melrose introduced the notion of pseudodifferential $\phi$-operators.  This is done by defining their Schwartz kernels on a suitable double space, namely the $\phi$-double space.  To define it, one starts with the manifold with corners $M^2=M\times M$.  Denote by $x$ and $x'$ the boundary defining functions of the boundary hypersurfaces $\pa M\times M$ and $M\times \pa M$ obtained by lifting $x\in \CI(M)$ via the projections on the left and right factors.  Blowing up the corner $\pa M\times \pa M$ gives the $b$-double space
\begin{equation}
    M^2_b= [M^2;\pa M \times \pa M] \quad \mbox{with blow-down map} \quad \beta_b:M^2_b\to M^2.
\label{phi.9}\end{equation}
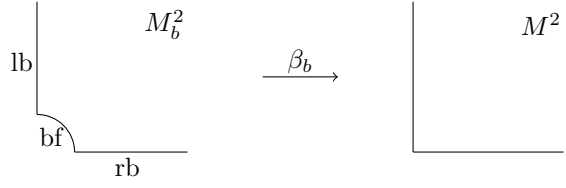
\begin{figure}[h]
\begin{tikzpicture}
\draw(0.5,0) arc [radius=0.5, start angle=0, end angle=90];
\draw(0.5,0)--(2,0);
\draw(0,0.5)--(0,2); 
\draw[->] (3,1)--(4,1);
\node at (3.5,1.2) {$\beta_b$};
\draw(5,0)--(7,0);
\draw(5,0)--(5,2);
\node at (1.2,-0.2) {$\rb$};
\node at (-0.2,1.2) {$\lb$};
\node at (0.2,0.2) {$\fb$};
\node at (1.7,1.7) {$M^2_b$};
\node at (6.7,1.7) {$M^2$};
\end{tikzpicture}
\caption{The $b$-double space }
\label{fig.1}\end{figure}
 The manifold with corners $M^2_b$ has now three boundary hypersurfaces, namely the lift $\lf$ and $\rf$ of the old boundary hypersurfaces $\pa M\times M$ and $M\times \pa M$, as well as a new boundary hypersurfaces $\fb$ created by the blow-up of $\pa M\times \pa M$.  The boundary hypersurface $\bf$ is naturally diffeomorphic to 
\begin{equation}
      \pa M\times \pa M\times [0,\frac{\pi}2]
\label{phi.10}\end{equation}
where the coordinate in the factor $[0,\frac{\pi}2]$ can be taken to be $\theta= \arctan \left( \frac{x}{x'} \right)$.  With respect to this identification, we can consider the $p$-submanifold 
\begin{equation}
   \Phi= \{  (p,q,\theta)\in \pa M\times \pa M\times [0,\frac{\pi}2] \; | \; \phi(p)=\phi(q), \; \theta=\frac{\pi}4\}.
\label{phi.11}\end{equation}
The $\phi$-double space is then the manifold with corners obtained from $M^2_b$ by blowing up the $p$-submanifold $\Phi$,
\begin{equation}
   M^2_{\phi}= [M^2_b;\Phi] \quad \mbox{with blow-down map} \quad \beta_{\phi}: M^2_{\phi}\to M^2.
\label{phi.12}\end{equation}
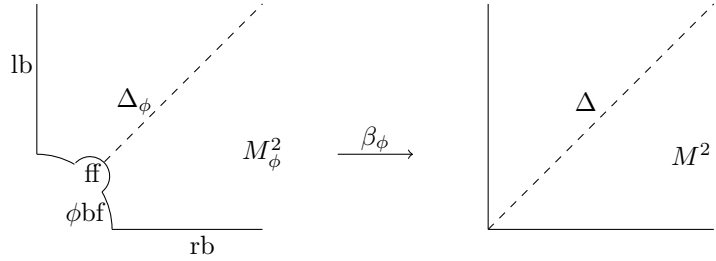
\begin{figure}[h]
\begin{tikzpicture}
\draw(1,0) arc [radius=1, start angle=0, end angle=30];
\draw(0,1) arc [radius=1, start angle=90, end angle=60];
\draw(0.866,0.5) arc[radius=0.261, start angle=-52.5, end angle=137.5];
\draw(1,0)--(3,0);
\draw(0,1)--(0,3); 
\draw[->] (4,1)--(5,1);
\node at (4.5,1.2) {$\beta_\phi$};
\draw(6,0)--(9,0);
\draw(6,0)--(6,3);
\draw[dashed] (6,0)--(9,3);
\node at (7.3,1.7) {$\Delta$};
\draw [dashed] (0.9,0.9)--(3,3);
\node at (1.3,1.7) {$\Delta_{\phi}$};
\node at (2.2,-0.2) {$\rb$};
\node at (-0.2,2.2) {$\lb$};
\node at (0.65,0.2) {$\fbf$};
\node at (0.75,0.75) {$\ff$};
\node at (3,1) {$M^2_{\phi}$};
\node at (8.7,1) {$M^2$};
\end{tikzpicture}
\caption{The $\phi$-double space }
\label{fig.2}\end{figure}

On $M^2_{\phi}$ we denote again by $\lf$ and $\rf$ the boundary hypersurfaces corresponding to the lifts of $\lf$ and $\rf$ from $M^2_b$ to $M^2_{\phi}$.  We also denote by $\fbf$ the lift of $\fb$ to $M^2_{\phi}$ and by $\ff$ the new boundary hypersurface created by the blow-up of $\Phi$.  

Let $\Delta_{\phi}$ be the lift of the diagonal $\Delta\subset M\times M$ to $M^2_{\phi}$.  As shown in \cite{Mazzeo-MelrosePhi}, one of the main features of the $\phi$-double space is that the lift from the left or from the right of $\phi$-vector fields are transverse to $\Delta_{\phi}$.  This suggests to define pseudodifferential $\phi$-operators as conormal distributions with respect to $\Delta_{\phi}$ on $M^2_{\phi}$.  Let 
\begin{equation}
   {}^{\phi}\Omega(M)= |\Lambda^{\dim M}({}^{\phi}T^*M)|
\label{phi.13}\end{equation}
be the bundle of $\phi$-densities on $M$.  If $\pi_R=\pr_R\circ \beta_{\phi}$ and $\pi_{L}=\pr_L\circ\beta_{\phi}$ with $\pr_R: M\times M\to M$ and $\pr_L: M\times M\to M$ the projections on the right and left factor, then on $M^2_{\phi}$ we can consider the bundle of right $\phi$-densities
\begin{equation}
  {}^{\phi}\Omega_R(M)= \pi^*_R({}^{\phi}\Omega(M)),
\label{phi.14}\end{equation}
as well as the homomorphism bundle 
\begin{equation}
        \Hom_{\phi}(E,F)= \pi_L^*E\otimes \pi_R^* F^*
\label{phi.15}\end{equation}
for $E$ and $F$ vector bundles over $M$.  
\begin{definition}
Let $E$ and $F$ be vector bundles on $M$.  The \textbf{small calculus of pseudodifferential $\phi$-operators} acting from sections of $E$ to sections of $F$ is the union over $m\in\bbR$ of the spaces
\begin{equation}
\Psi^m_{\phi}(M;E,F)= \{  \kappa\in I^m(M^2_{\phi},\Delta_{\phi}; \Hom_{\phi}(E,F)\otimes {}^{\phi}\Omega_R(M)) \quad | \quad \kappa\equiv 0 \; \mbox{at} \; \pa M^2_{\phi}\setminus \ff\},
\label{phi.16b}\end{equation}
where $I^m(M^2_{\phi},\Delta_{\phi}; \Hom_{\phi}(E,F)\otimes {}^{\phi}\Omega_R(M))$ is, in the sense of \cite[Definition~18.2.6]{Hormander3}, the space of conormal distributions of order $m$ with respect to $\Delta_{\phi}$ taking value in the vector bundle $\Hom_{\phi}(E,F)\otimes {}^{\phi}\Omega_R(M)$ and $\kappa\equiv 0$ at $\pa M^2_{\phi}\setminus \ff$ means that the Taylor series of $\kappa$ is trivial at all boundary hypersurfaces of $M^2_{\phi}$ except possibly at $\ff$.  
\label{phi.16}\end{definition}
As shown in \cite{Mazzeo-MelrosePhi}, an operator $P\in \Psi^m_{\phi}(M;E,F)$ induces an operator
$$
     P: \CI(M;E)\to \CI(M;F).  
$$
The calculus is also closed under composition in that 
$$
      \Psi^m_{\phi}(M;F,G)\circ \Psi^{m'}_{\phi}(M;E,F)\subset \Psi^{m+m'}_{\phi}(M;E,G).
$$
Furthermore, simple criteria are provided in \cite{Mazzeo-MelrosePhi} to determine when an operator is bounded, compact or Fredholm when acting on weighted $L^2$-Sobolev spaces associated to a $\phi$-metric.  For instance, we know from \cite[Lemma~12]{Mazzeo-MelrosePhi} that a $\phi$-operator  $K$ of negative order  is compact when acting on the $L^2$-space of a $\phi$-metric provided its normal operator $N_{\ff}(K)$, that is, its restriction to $\ff$, vanishes.

As for the $b$-calculus however, some parametrix constructions require  a larger calculus.  If $\cE$ is, in the sense of \cite[\S~4]{Melrose1992}, an index family for the boundary hypersurfaces of $M^2_{\phi}$, one can more generally consider the spaces
\begin{equation}
\begin{gathered}
\Psi^{-\infty,\cE}_{\phi}(M;E,F)= \cA^{\cE}_{\phg}(M^2_{\phi};\Hom_{\phi}(E,F)\otimes {}^{\phi}\Omega_R(M)), \\
\Psi^{m,\cE}_{\phi}(M;E,F)= \Psi^{m}_{\phi}(M;E,F)+ \Psi^{-\infty,\cE}_{\phi}(M;E,F), \quad m\in\bbR,
\end{gathered}
\label{phi.17}\end{equation}
where $\cA^{\cE}_{\phg}(M^2_{\phi};\Hom_{\phi}(E,F)\otimes {}^{\phi}\Omega_R(M))$ denotes the space of polyhomogeneous sections of $\Hom_{\phi}(E,F)\otimes {}^{\phi}\Omega_R(M)$ with polyhomogenous  expansions compatible with the index family $\cE$ in the sense of \cite[(23)]{Melrose1992}.  
Using the pushforward theorem of \cite{Melrose1992}, one can show as in \cite[(26)]{Vaillant} or \cite[Theorem~3.3]{ARS1} that these operators act on polyhomogeneous sections as follows.
\begin{proposition}
Let $A\in \Psi^{m,\cE}_{\phi}(M;E,F)$ and $\sigma\in \cA^{\cF}_{\phg}(M;E)$ with index family $\cE$ and index set $\cF$ such that 
$$
\Re(\cE|_{\rf}+\cF)>h+1,
$$
where $h=\dim Y$ is the dimension of the base of the fiber bundle $\phi:\pa M\to Y$.
Then the action of $A$ on $\sigma$ is well-defined, giving a polyhomogeneous section $A\sigma\in \cA^{\cG}(M;F)$ with index set $\cG$ given by
$$
  \cG= \cE|_{\lf}\overline{\cup}(\cE|_{\ff}+\cF)\overline{\cup}(\cE|_{\fbf}+\cF-h-1),
$$
where $h$ is the dimension of the base $Y$ and $\overline{\cup}$ denotes the extended union of index sets of  \cite[(43)]{Melrose1992}.
\label{phi.17b}\end{proposition}

Similarly, the $\phi$-triple space of \cite{Mazzeo-MelrosePhi} and the pushforward theorem of \cite{Melrose1992} can be used to show as in \cite[Theorem~2.11]{Vaillant} or \cite[Theorem~3.4]{ARS1} that this larger class of $\phi$-operators behaves well under composition.

\begin{proposition}
Let $\cE$ and $\cF$ be index families for the boundary hypersurfaces of $M^2_{\phi}$ such that 
$$
             \Re(\cE|_{\rf})+\Re(\cF|_{\lf})>h+1.
$$
where $h=\dim Y$ as in Proposition~\ref{phi.17b}.  
Then given $A\in \Psi^{m,\cE}_{\phi}(M;F,G)$ and $B\in \Psi^{m',\cF}_{\phi}(M;E,F)$, their composition is well defined with
$$
    A\circ B\in \Psi^{m+m',\cG}_{\phi}(M;E,G),
$$
where $\cG$ is the index family given by
\begin{equation}
\begin{aligned}
\cG|_{\lf} &= (\cE|_{\lf})\overline{\cup} (\cE|_{\fbf}+\cF|_{\lf}-h+1)\overline{\cup}(\cE|_{\ff}+\cF|_{\lf}), \\
\cG|_{\rf} &= (\cF|_{\rf})\overline{\cup}(\cE|_{\rf}+\cF|_{\fbf}-h-1)\overline{\cup}(\cE|_{\rf}+\cF|_{\ff}), \\
\cG|_{\fbf}&=(\cE|_{\lf}+\cF|_{\rf})\overline{\cup}(\cE|_{\fbf}+ \cF|_{\fbf}-h-1)\overline{\cup}(\cE|_{\fbf}+\cF|_{\ff})\overline{\cup}(\cE|_{\ff}+ \cF|_{\fbf}), \\
\cG|_{\ff}&= (\cE|_{\lf}+\cF|_{\rf})\overline{\cup}(\cE|_{\fbf}+\cF|_{\fbf}-h-1)\overline{\cup}(\cE|_{\ff}+\cF|_{\ff}).
\end{aligned}
\label{phi.18b}\end{equation}
\label{phi.18}\end{proposition}

\section{Fredholm fibered boundary Dirac operators} \label{fdt.0}

Let $M$ be a compact manifold with boundary $\pa M$ equipped with a fiber bundle $\phi:\pa M\to Y$ over a closed manifold $Y$.  Fix a boundary defining function $x\in\CI(M)$ and let $g_{\phi}$ be a product-type fibered boundary metric up to order 2.  Let $E\to M$ be a Hermitian vector bundle and consider an elliptic formally self-adjoint first order fibered boundary operator $\eth_{\phi}\in\Diff^1_{\phi}(M;E)$.  An example to keep in mind is the situation where $E$ is a Clifford module for the Clifford bundle of the $\phi$-tangent bundle and $\eth_{\phi}$ is the Dirac operator associated to a choice of Clifford connection.  

Instead of $\eth_{\phi}$ acting formally on $L^2_{\phi}(M;E)$, it is convenient to consider equivalently the fibered boundary operator
\begin{equation}
D_{\phi}= x^{-\frac{h+1}2}\eth_{\phi}x^{\frac{h+1}2} \quad \mbox{acting formally on} \; L^2_b(M;E)=x^{-\frac{h+1}2}L^2_{\phi}(M;E),
\label{fdt.1}\end{equation}
where $h:= \dim Y$.  In this way, one important model operator, the indicial family of Definition~\ref{dt.6} below, can be defined essentially by Mellin transform of a restriction to $\fbf$, in direct analogy with the indicial family of \cite{MelroseAPS} for $b$-operator.  This will in particular ease the use of results from \cite{MelroseAPS} for the construction of the parametrix.   

Since $\eth_{\phi}$ is formally self-adjoint with respect to $L^2_{\phi}(M;E)$, notice that $D_{\phi}$ is formally self-adjoint with respect to $L^2_b(M;E)$.  

\begin{definition}
The \textbf{vertical family} is the family of vertical operators $D_v\in\Diff^1(\pa M/Y;E)$ obtained by restricting the action of $D_{\phi}$ to the boundary $\pa M$.
\label{dt.2}\end{definition}
The vertical family is closely related to the normal operator $N_{\ff}(D_{\phi})$ of $D_{\phi}$ obtained by restricting $D_{\phi}$ to $\ff$ as a conormal distribution.  As described in \cite[\S~4]{Mazzeo-MelrosePhi}, the normal operator is a family of suspended operators in the fibers of $\phi:\pa M\to Y$.  A direct computation  shows that 
\begin{equation}
  Y\ni p \mapsto N_{\ff}(D_{\phi})_p= D_v|_{\phi^{-1}(p)}+ \eth_h(p),
\label{dt.3}\end{equation}
where $p\mapsto \eth_h(p)$ is a family of fiberwise translation invariant elliptic first order differential operators associated to the vector bundle ${}^{\phi}N\pa M\to \pa M$ of \eqref{phi.4b} restricted to $\phi^{-1}(p)$.      We will assume that $\eth_h$ is in fact a family of Euclidean Dirac operators anti-commuting with $D_v$.  As the next lemma shows, this condition is automatically satisfied if $\eth_{\phi}$ is a Dirac operator, for instance if it is the Hodge-deRham operator of the metric $g_{\phi}$.
\begin{lemma}
If $\eth_{\phi}$ is a Dirac operator, then $\eth_h$ is a family of Euclidean Dirac operators anti-commuting with $D_v$.  
\label{dt3.b}\end{lemma}  
\begin{proof}
Let $p\in Y$ be given.  Since $g_{\phi}$ is product-type up to order 2, notice that under the identification 
\begin{equation}
        {}^{\phi}N\pa M|_{\phi^{-1}(p)}= \phi^{-1}(p)\times {}^{\phi}N_pY
\label{dt.3c}\end{equation}
coming form \eqref{phi.4c}, the metric induced by $g_{\phi}$ corresponds to a Cartesian product.  On the other hand, the Clifford module $E$ used to define $\eth_{\phi}$ induces one on this Cartesian product that we will  denote by $E_p$.  This bundle $E_p$ is in fact naturally the pullback of $E|_{\phi^{-1}(p)}$ via the bundle projection ${}^{\phi}N\pa M|_{\phi^{-1}(p)}\to \phi^{-1}(p)$.  Similarly, there is an induced Clifford connection $\nabla^{E_p}$ which is just the pull-back of the Clifford connection of $E|_{\phi^{-1}(p)}$.  With respect to this data, the normal operator $N_{\ff}(D_{\phi})$ restricted to \eqref{dt.3c} is just the corresponding Dirac operator with $D_v|_{\phi^{-1}(p)}$ the part acting on the fibers of $\phi^{-1}(p)\times {}^{\phi}N_pY\to {}^{\phi}N_pY$ (the operator is the same for each fiber) and $\eth_h(p)$ is the part acting on the fibers of ${}^{\phi}N\pa M|_{\phi^{-1}(p)}\to \phi^{-1}(p)$.  In particular, $\eth_h(p)$ is a family of Euclidean Dirac operators.  To see that $D_v|_{\phi^{-1}(p)}$ and $\eth_h(p)$ anti-commute, it suffices to check that $c(e_1)\nabla^{E_p}_{e_1}$ and $c(e_2)\nabla^{E_p}_{e_2}$ anti-commute, where   $e_1$ and $e_2$ are vector fields on $\phi^{-1}(p)$ and ${}^{\phi}N_pY$ lifted to the Cartesian product \eqref{dt.3c} and $c(e_i)$ denotes Clifford multiplication by $e_i$.

But in this case, $\nabla_{e_1}e_2= \nabla_{e_2}e_1=0$, so using that $\nabla^{E_p}$ is a Clifford connection, we compute that 
\begin{equation}
\begin{aligned}
c(e_1)\nabla^{E_p}_{e_1}c(e_2)\nabla^{E_p}_{e_2}&= c(e_1)[\nabla^{E_p}_{e_1},c(e_2)]\nabla^{E_p}_{e_2}+ c(e_1)c(e_2) \nabla^{E_p}_{e_1}\nabla^{E}_{e_2} \\
&= c(e_1)c(\nabla_{e_1}e_2)\nabla^{E}_{e_2}+c(e_1)c(e_2) \nabla^{E_p}_{e_1}\nabla^{E}_{e_2} \\
&= c(e_1)c(e_2) \nabla^{E_p}_{e_1}\nabla^{E}_{e_2}.
\end{aligned}
\label{ac.1}\end{equation} 
Similarly, 
\begin{equation}
  c(e_2)\nabla^{E_p}_{e_2}c(e_1)\nabla^{E_p}_{e_1}=c(e_2)c(e_1) \nabla^{E_p}_{e_2}\nabla^{E}_{e_1}.
\label{ac.2}\end{equation}
Now, the curvature of $(E_p,\nabla^{E_p})$ is just the pull-back of the curvature of $E|_{\phi^{-1}(p)}$, which implies that $[\nabla^{E_p}_{e_1},\nabla^{E_p}_{e_2}]=0$.  Since $c(e_1)c(e_2)=-c(e_2)c(e_1)$, we thus deduce from \eqref{ac.1} and \eqref{ac.2} that $c(e_1)\nabla^{E_p}_{e_1}$ and $c(e_2)\nabla^{E_p}_{e_2}$ anti-commute as claimed.
\end{proof}

To be able to construct a good parametrix, we will make the following assumption.

\begin{assumption}
The nullspaces of the various fiberwise operators of the family $D_v$ form a vector bundle 
$$
     \ker D_v\to Y.
$$
\label{dt.4}\end{assumption}
Using the restriction of the metric $g_{\phi}$ to the fibers of $\phi:\pa M\to Y$ and the Hermitian metric of $E$, we can define a family of $L^2$-projections
\begin{equation}
\Pi_h: \CI(Y; L^2(\pa M/Y;E))\to \CI(Y;\ker D_v)
\label{dt.5}\end{equation}
onto $\ker D_v$, where $L^2(\pa M/Y;E)\to Y$ is the infinite rank vector bundle with fiber above $y\in Y$ given by $L^2(\phi^{-1}(y);E)$.  This can be used to define a natural indicial family.

\begin{definition}
The indicial family $\bbC\ni \lambda \mapsto I(D_b,\lambda)\in \Diff^1(Y;\ker D_v)$ associated to $D_{\phi}$ is defined by
$$
     I(D_b,\lambda)u:= \Pi_h\left(  \left(x^{-\lambda}(x^{-1}D_{\phi})x^{\lambda}\widetilde{u}\right)|_{\pa M}\right), \quad u\in\CI(Y;\ker D_v),
$$
where $\widetilde{u}\in\CI(M;E)$ is such that $\widetilde{u}|_{\pa M}=u$.  As the notation suggests, the indicial family $I(D_b,\lambda)$ is the Mellin transform of the operator
\begin{equation}
    D_b:=cx\frac{\pa}{\pa x}+ D_Y,  \quad \mbox{with} \quad c:= \left.\frac{\pa}{\pa \lambda}I(D_b,\lambda)\right|_{\lambda=0} \quad \mbox{and} \quad  D_Y:= I(D_b,0).
\label{dt.8b}\end{equation}
\label{dt.6}\end{definition}
The interested reader may look at \cite[\S~5.2]{HHM2004} for nice intuitive explanations motivating  Definition~\ref{dt.6}.  

\begin{lemma}
The indicial family $I(D_b,\lambda)$ is well-defined, namely $I(D_b,\lambda)$ does not depend on the choice of extension $\widetilde{u}$.  
\label{dt.7}\end{lemma}
\begin{proof}
Essentially by definition of the Lie algebra of fibered boundary vector fields, notice first that 
$$
[D_{\phi},x]\in x^2\CI(M;\End(E)).
$$  
Moreover, if $\widetilde{u}_1$ and $\widetilde{u}_2$ are two choices of extensions of $u$, then 
$\widetilde{u}_1-\widetilde{u}_2= xw$ for some $w\in \CI(M;E)$, so that 
\begin{equation}
\begin{aligned}
(x^{-\lambda}(x^{-1}D_{\phi})x^{\lambda}(\widetilde{u}_1-\widetilde{u}_2))|_{\pa M} &= (x^{-\lambda-1}D_{\phi}x^{\lambda+1}w)|_{\pa M}= (x^{-\lambda-1}(x^{\lambda+1}D_{\phi} + [D_{\phi},x^{\lambda+1}]))w \\
&= (D_{\phi}w+ x^{-\lambda-1}(\lambda+1)x^{\lambda}[D_{\phi},x]w)|_{\pa M} \\
&= (D_{\phi} w)|_{\pa M}= D_v(w|_{\pa M}).
\end{aligned}
\label{dt.8}\end{equation}
Now, we see from \eqref{dt.3} that the formal self-adjointness of $D_{\phi}$ on $L^2_b(M;E)$ implies the formal self-adjointness of $D_v$.  This implies in particular that the image of $D_v$ is orthogonal to its kernel, hence that
$$
 \Pi_h((x^{-\lambda}(x^{-1}D_{\phi})x^{\lambda}(\widetilde{u}_1-\widetilde{u}_2))|_{\pa M})= \Pi_h(D_v(w|_{\pa M}))=0,
$$ 
showing that $I(D_b,\lambda)u$ does not depend on the choice of smooth extension $\widetilde{u}$ as claimed.
\end{proof}

We will now give a more detailed description of the indicial family when $\eth_{\phi}$ is a Dirac operator, see \eqref{dt.8e} below.  This is important for two reasons:
 \begin{enumerate}
 \item  it will then be easier to determine for which weights Theorem~\ref{int.13} in the introduction will apply;
 \item such a detailed description will play a crucial role in the proof of the pseudodifferential characterization of the low energy limit, notably through the proof of Lemma~\ref{le.29} below.
 \end{enumerate}
To give this more detailed description of the indicial family, recall first that by assumption, $g_{\phi}$ is modelled at infinity by the metric 
\begin{equation}
  g_{\cC_{\phi}}= \frac{dx^2}{x^4}+ \frac{g_Y}{x^2} + \kappa
\label{em.1}\end{equation}
on $(0,\infty)\times \pa M$ with the map $\Id\times \phi: (0,\infty)\times \pa M\to (0,\phi)\times Y$ inducing a Riemannian submersion onto the Riemannian cone $\left((0,\infty)\times Y, \frac{dx^2}{x^4}+\frac{g_Y}{x^2}\right).$
 On the other hand, $\ker D_v$ is naturally a Clifford module for the tangent bundle $TY\to Y$ via the natural map
\begin{equation}
    \begin{array}{lcl}
       TY&\to & {}^{\phi}TM|_{\pa M} \\
         \xi& \mapsto & x\xi.
    \end{array}
\label{dt.8c}\end{equation}
  By \cite[Proposition~10.12, Lemma~10.13]{BGV}, the Dirac operator corresponding to the model metric $g_{\cC_{\phi}}$ is
\begin{equation}
   \eth_{\cC_{\phi}}= D_v+ \widetilde{\eth}_{\cC}
\label{em.2}\end{equation}
where $\widetilde{\eth}_{\cC}$ is the horizontal Dirac operator induced by the connection of $\phi:\pa M\to Y$ and  the Clifford connection 
\begin{equation}
\nabla^E+ \frac{1}{2} c(\omega),
\label{em.3}\end{equation} 
where $\omega$ is the $\Lambda^2T^*(\pa M)$-valued $1$-form on $\pa M$ of \cite[Definition~10.5]{BGV} defined by
$$
 \omega(X)(Y,Z)=S(X,Z)(Y)-S(X,Y)(Z)+ \frac12(\Omega(X,Z),Y) - \frac12(\Omega(X,Y),Z)+ \frac12(\Omega(Y,Z),X)
$$
with $S$ and $\Omega$ the second fundamental form and curvature of  the Riemannian submersion $\phi:\pa M\to Y$, while  $c(\omega)$ is defined in \cite[Proposition~10.12(2)]{BGV} by
$$
  c(\omega)= \frac12 \sum_{abc} \omega(e_a)(e_b,e_c) e^a\otimes c(e^b)c(e^c)
$$
with $e_a$ a local frame for $T(\pa M)$ and $e^a$ its dual frame.

 Using the projection $\Pi_h$ on $\cC_{\phi}$, this yields a corresponding Dirac operator $\eth_{\cC}=\Pi_h\widetilde{\eth}_{\cC}\Pi_h$ on $\ker D_v$ with Clifford connection 
\begin{equation}
  \Pi_h(\nabla^E+ \frac{c(\omega)}{2})\Pi_h.
\label{em.3b}\end{equation}

As described above, the term $c(\omega)$ involves the second fundamental form and the curvature of $\phi:\pa M\to Y$.  Those depend only on the fiberwise metric, so really are pull-back of forms on $\pa M$ via the projection $(0,\infty)\times \pa M\to \pa M$.  However, when measured with respect to the metric $g_{\cC_{\phi}}$, that is, in terms of the $\phi$-tangent bundle, the part involving the curvature is $\mathcal{O}(x^2)$ when $x\searrow 0$, so does not contribute to the indicial family $I(D_b,\lambda)$.  However, the part coming from the second fundamental form is $\mathcal{O}(x)$, so does contribute to the indicial family.    

To describe this more explicitly, suppose first that $\cC=(0,\infty)\times Y$ is spin and consider the Dirac operator $\eth_{\cC}^{\cS}$ associated to the cone metric 
\begin{equation}
  g_{\cC}=\frac{dx^2}{x^4}+ \frac{g_Y}{x^2}
\label{do.1}\end{equation}
and acting on the sections of the spinor bundle $\cS$ over $\cC$.  If $\psi$ is a section of $\cS|_{\{1\}\times Y}$, let 
$\overline{\psi}\in\CI(\cC;\cS)$ be the section obtained by parallel transport of $\psi$ along geodesics emanating from the tip of the cone.  This induces a decomposition
\begin{equation}
    \CI(\cC;\cS)\cong \CI((0,\infty))\widehat{\otimes} \CI(Y;\cS|_{\{1\}\times Y}). 
\label{do.2}\end{equation}
By \cite[Proposition~2.5]{Chou}, the Dirac operator takes the form
\begin{equation}
cx^2\frac{\pa}{\pa x}+ x\left( \eth_Y^{\cS}-\frac{ch}2 \right)
\label{do.2}\end{equation}
in terms of this decomposition, where $h=\dim Y$, $c$ is Clifford multiplication by $x^2\frac{\pa}{\pa x}$ (which corresponds to the $c$ of \eqref{dt.8b}) and $\eth_Y^{\cS}$ is the Dirac operator on $(Y,g_Y)$ acting on sections of $\cS|_{\{1\}\times Y}$.  If we twist the spinor bundle by a Euclidean vector bundle $\cW$ with orthogonal connection, there is a  corresponding twisted Dirac operator $\eth^{\cS\otimes \cW}_{\cC}$.  We will suppose that $\cW$ is constructed geometrically from $(\cC, g_{\cC})$ and the spin structure, or else that it is the pull-back of a Euclidean vector bundle with orthogonal connection on $Y$.

Again, parallel transport along geodesics emanating from the tip of the cone induces a decomposition
\begin{equation}
    \CI(\cC;\cS)=\CI((0,\infty))\widehat{\otimes} \ \CI(Y;(\cS\otimes \cW)|_{\{1\}\times Y})
\label{do.3}\end{equation}  
in terms of which \eqref{do.2} is replaced by
\begin{equation}
 cx^2\frac{\pa}{\pa x}+ x\left(\eth^{\cS\otimes\cW}_Y+N^{\cS\otimes\cW}-\frac{ch}2  \right),
\label{do.4}\end{equation}
where now $\eth_Y^{\cS\otimes \cW}$ is the Dirac operator on $(Y,g_Y)$ acting on sections of $(\cS\otimes \cW)|_{\{1\}\times Y}$ and $N^{\cS\otimes\cW}$ is a self-adjoint operator  of order zero acting on sections of 
  $\cS\otimes\cW$ which anti-commutes with $c$.  For instance, if $\cW$ is the pull-back of an Euclidean bundle with orthogonal connection on $Y$, then $N^{\cS\otimes \cW}=0$.  Since the computations considered were local on $Y$ and since a spin structure always exists at least locally on $Y$, we see that \eqref{do.4} extends to Dirac operators by \cite[Proposition~3.40]{BGV}.  Thus, if $\cE$ is a Clifford module with Clifford connection on $(\cC,g_{\cC})$ and $\eth^{\cE}$ is the corresponding Dirac operator, then in terms of the decomposition
\begin{equation}
  \CI(\cC;\cE)\cong \CI((0,\infty))\widehat{\otimes} \  \CI(Y;\cE|_{\{1\}\times Y}),
\label{do.5}\end{equation}    
we have that 
\begin{equation}
      \eth^{\cE}= cx^2\frac{\pa}{\pa x}+ x\left( \eth^{\cE}_Y+N^{\cE}-\frac{ch}2\right)
\label{do.6}\end{equation}
with $\eth^{\cE}_Y$ the Dirac operator of $\cE|_{\{1\}\times Y}$ on $Y$ and $N^{\cE}$ is a self-adjoint term of order zero anti-commuting with $c$.

We would like to apply \eqref{do.6} to the operator $\eth_{\cC}= \Pi_h \widetilde{\eth}_{\cC}\Pi_h$.  However, we must be careful because of the extra term $\frac12 \cl(\omega)$.  First, because the form $\omega$ is a pull-back of a form on $\pa M$, parallel transport along geodesics emanating from the tip of the cone is the same whether we use $\Pi_h \nabla^E\Pi_h$ or \eqref{em.3b}.  This yields again a decomposition
$$
  \CI(\cC;\ker D_v)\cong \CI((0,\infty))\widehat{\otimes} \ \CI(Y;\ker D_v|_{\{1\}\times Y})
$$
in terms of which  we have
\begin{equation}
  \eth_{\cC}= cx^2\frac{\pa}{\pa x}+ x\left( \eth_Y-\frac{ch}2 \right)+ x^2\cV_{\Omega},
\label{dt.8d}\end{equation} 
where $\eth_Y=\hat{\eth}_Y+N$ with $\hat{\eth}_Y$ the Dirac operator induced by the connection
\begin{equation}
   \Pi_h(\nabla^E+ \frac{\hat{\omega}}2)\Pi_h
\label{dt.8f}\end{equation}
with $\hat{\omega}$ the part of $\omega$ involving the second fundamental form of $\phi:\pa M\to Y$,  $N\in\CI(Y;\End(\ker D_v))$ is a self-adjoint operator of order $0$ anti-commuting with $c$ and $x^2\cV_{\Omega}$ is the part of $\Pi_h\frac{\cl(\omega)}2\Pi_h$ coming from the curvature of $\phi: \pa M\to Y$.  
Hence, in terms of this description, the operator $x\left( cx\frac{\pa}{\pa x}+ D_Y \right)$ in \eqref{dt.8b} is obtained from \eqref{dt.8d} by suppressing the curvature term $x^2\cV_{\Omega}$,
\begin{equation}
xD_b=x\left( cx\frac{\pa}{\pa x}+ D_Y \right)= x^{-\frac{h+1}2}\left( cx^2\frac{\pa}{\pa x}+ x\left(\eth_Y-\frac{ch}2\right)\right)x^{\frac{h+1}2}= cx^2\frac{\pa}{\pa x}+ x\left(\eth_Y+\frac{c}2\right).
\label{dt.8e}\end{equation}
Since $N$ is self-adjoint and anti-commutes with $c$, the operator $\eth_Y$ is also self-adjoint and anti-commutes with $c$, a fact that will be useful in the proof of Lemma~\ref{le.29} below.    

Among Dirac operators, our main motivating example is the Hodge-deRham operator acting on forms with values in a flat vector bundle.  In this case, the bundle $\ker D_v$ corresponds to the bundle of fiberwise harmonic forms.  By \cite[Proposition~15]{HHM2004}, this is a flat vector bundle with respect to the connection \eqref{em.3b}.  Now, if $\overline{\eta}$ is a $\ker D_v$-valued $k$-form on $\cC$  obtained by parallel transport of its restriction $\eta$ to $\{1\}\times Y$ along geodesics emanating from the tip of the cone, then there is a decomposition
$$
  \overline{\eta}= \overline{\alpha}+ \frac{dx}{x^2}\wedge \overline{\beta}, \quad \eta=\alpha+ dx\wedge \beta, \quad \overline{\alpha}=\frac{\alpha}{x^k}, \; \overline{\beta}= \frac{\beta}{x^{k-1}}
$$ 
for some $\ker D_v$-valued forms $\alpha$ and $\beta$ on $Y$.  In terms of this decomposition, we know from \cite[Proposition~15]{HHM2004} that the operator $\eth_Y= \hat{\eth}_Y+N$ in \eqref{dt.8e} is such that
\begin{equation}
   \hat{\eth}_Y= \left( \begin{array}{cc} \mathfrak{d} & 0 \\ 0 & -\mathfrak{d}   \end{array} \right)
\label{do.7}\end{equation}
with $\mathfrak{d}$ the Hodge-deRham operator acting on $\Omega^*(Y;\ker D_v)$, while
\begin{equation}
   N= \left( \begin{array}{cc} 0 & \frac{h}2-\cN_Y \\ \frac{h}2-\cN_Y & 0  \end{array} \right) \quad \mbox{and} \quad c= \left( \begin{array}{cc} 0 & -1 \\ 1 & 0\end{array}  \right)
\label{do.8}\end{equation}
with $\cN_Y$ the number operator acting on a form in $\Omega^*(Y;\ker D_v)$ (of pure degree) by multiplying it by its degree.  The indicial family is therefore given in that case by
\begin{equation}
  I(D_b,\lambda)= \left( \begin{array}{cc}  \mathfrak{d} & -\lambda+ \frac{h-1}2-\cN_Y \\
                      \lambda+\frac{h+1}2-\cN_Y & -\mathfrak{d}  \end{array} \right).
\label{do.8b}\end{equation}

Keeping these examples in mind, let us come back to the indicial family $I(D_b,\lambda)$ and recall the following standard definition. 
\begin{definition}
An \textbf{indicial root} of the indicial family $I(D_b,\lambda)$ is a complex number $\zeta$ such that 
$$
I(D_b,\zeta):L^2_1(Y;\ker D_v)\to L^2(Y,\ker D_v)
$$
is not invertible, where $L^2_1(Y;\ker D_v)$ is the natural $L^2$-Sobolev space of order $1$ of sections of $\ker D_v\to Y$ with respect to $g_Y$.  A \textbf{critical weight} of the indicial family $I(D_b,\lambda)$ is a real number $\delta$ such that $\delta+i\nu$ is an indicial root for some $\nu\in\bbR$.  In other words, $\delta$ is a critical weight if it is the real part of some indicial root.  We will denote by $\Crit(D_b)$ the set of critical weights of the indicial family $I(D_b,\lambda)$.
\label{dt.9}\end{definition}
\begin{remark}
Since $x^{-\frac12}D_{\phi}x^{-\frac12}= x^{\frac12}(x^{-1}D_{\phi})x^{-\frac12}$ is formally self-adjoint, notice that the indicial roots are real and that  $\lambda$ is an indicial root of $I(D_b,\lambda)$ if and only if $-1-\lambda$ is an indicial root.  
\label{dt.9b}\end{remark}

For instance, the indicial roots of the Hodge-deRham operator can be described in terms of the eigenvalues of $\mathfrak{d}= d^{\ker D_v}+ \delta^{\ker D_v}$ as the next lemma shows.

\begin{lemma}
The indicial roots of the indicial family \eqref{do.8b} are given by 
\begin{equation}
\begin{gathered}
      (q-\frac{h+1}2), -(q-\frac{h-1}2), \quad \mbox{if}  \quad H^q(Y;\ker D_v)\ne \{0\}, \\
      \bigcup \left\{ \ell\pm \sqrt{\zeta+(q-\frac{h-1}2)^2} \quad | \quad \ell\in\{-1,0\}, \quad \zeta \in \Spec(\delta^{\ker D_v}d^{\ker D_v})_q\setminus\{0\}  \right\}  \\
      \bigcup \left\{ \ell\pm \sqrt{\zeta+ (q-\frac{h+1}2)^2}, \quad |  \quad \ell\in \{-1,0\}, \quad \zeta \in \Spec(d^{\ker D_v}\delta^{\ker D_v})_q\setminus \{0\}  \right\}.
\end{gathered}
\label{do.8d}\end{equation}
In particular, in agreement with Remark~\ref{dt.9b}, $\lambda$ is an indicial root if and only if $-1-\lambda$ is an indicial root.  
\label{do.8c}\end{lemma}
\begin{proof}
This is a standard computation.  We can proceed for instance as in the proof of \cite[Proposition~2.3]{ARS3}.  In fact, the indicial family of \cite[Proposition~2.3]{ARS3}, after suitable identifications, corresponds to $I(D_b,-\lambda)$, since it is the indicial family of the same operator, but considered at the opposite end of the cone.  Hence, \eqref{do.8d} follows by flipping the sign of the indicial roots in \cite[Proposition~2.3]{ARS3}.
\end{proof}

\begin{theorem}
Suppose that the operator $D_{\phi}$ satisfies Assumption~\ref{dt.4} and that  $\delta\in\bbR$ is not a critical weight of the indicial family $I(D_b,\lambda)$.  Let $\mu>0$ be such that $(\delta-\mu,\delta+\mu)\cap\Crit(D_b)=\emptyset$.  Then, in the notation of \S~\ref{phi.0}, there exists $Q\in \Psi^{-1,\cQ}_{\phi}(M;E)$ and $R\in\Psi^{-\infty,\cR}_{\phi}(M;E)$ such that
$$
           (x^{-\delta}D_{\phi}x^{\delta}) Q=\Id-R,
$$ 
where $\cQ$ is an index family such that 
$$
   \inf \Re(\cQ|_{\lf})\ge\mu, \quad \inf \Re(\cQ|_{\rf})\ge h+\mu, \quad \inf \Re (\cQ|_{\fbf})\ge h, \quad \inf \Re(\cQ|_{\ff})\ge 0,
$$
and $\cR$ is an index family giving the empty set at all boundary hypersurfaces except at $\rf$, where we have instead 
$$
   \inf\Re(\cR|_{\rf})\ge h+\mu.
$$
Moreover, the term $A$ of order $h$ at $\fbf$ of $Q$ is such that $A=\Pi_h A\Pi_h$. 
Here, an inequality of the form $\inf \Re (\cE)\ge a$ for $\cE$ an index set and $a\in \bbR$ means, in the equality case, that if $(a+i\nu, k)\in \cE$ with $\nu\in\bbR$, then $k=\nu=0$.  Finally, each term $r$ of order $h+1$ or less in the asymptotic expansion of $R$ at $\rf$ is such that $b\Pi_h =b$.  \label{dt.10}\end{theorem}
The construction of the parametrix $Q$ will involve few steps and is closely related to the resolvent construction of Vaillant \cite[\S~3]{Vaillant} for fibered cusp Dirac operators.

\subsection*{Step 0: Symbolic inversion}
We can first use ellipticity to do a symbolic inversion.
\begin{proposition}
There exist $Q_0\in \Psi^{-1}_{\phi}(M;E)$ and $R_0\in\Psi^{-\infty}_{\phi}(M;E)$ such that
$$
      (x^{-\delta}D_{\phi}x^{\delta})Q_0= \Id-R_0.
$$
\label{dt.11}\end{proposition}
\begin{proof}
The operator $x^{-\delta}D_{\phi}x^{\delta}$ is elliptic with principal symbol
$$
      {}^{\phi}\sigma_1(x^{-\delta}D_{\phi}x^{\delta})= {}^{\phi}\sigma_1(D_{\phi})={}^{\phi}\sigma_1(\eth_{\phi}),
$$
so we can find $Q_0'\in \Psi^{-1}_{\phi}(M;E)$ with principal symbol 
$$
     {}^{\phi}\sigma_{-1}(Q_0')= ({}^{\phi}\sigma_1(\eth_{\phi}))^{-1},
$$ 
so that
$$
    (x^{-\delta}D_{\phi}x^{\delta})Q_0'= \Id-R_0' \quad \mbox{for some} \quad R_0'\in \Psi^{-1}_{\phi}(M;E).
$$
Proceeding inductively, we then define more generally $Q_0^{(k)}= Q_0'R_0^{(k-1)}\in \Psi^{-k}_{\phi}(M;E)$ and $R^{(k)}_0\in \Psi^{-k}_{\phi}(M;E)$ such that 
$$
        (x^{-\delta}D_{\phi}x^{\delta})\left( \sum_{j=1}^k Q_0^{(j)} \right)= \Id-R_0^{(k)}.
$$
Taking an asymptotic sum over the $Q_0^{(k)}$ then gives the desired operator $Q_0$.   
\end{proof}

\subsection*{Step 1: Removing the error term at $\ff$}

In this step, we improve the parametrix so that the error term vanishes at the front face $\ff$.
\begin{proposition}
There exist $Q_1\in\Psi^{-1,\cQ_1}_{\phi}(M;E)$ and $R_1\in\Psi^{-\infty,\cR_1}_{\phi}(M;E)$ such that 
$$
    (x^{-\delta}D_{\phi}x^{\delta})Q_1=\Id-R_1,
$$ 
where the index families $\cQ_1$ and $\cR_1$ are the empty set at $\rf$ and $\lf$ and given otherwise by 
\begin{equation}
\cQ_1|_{\ff}=\bbN_0, \quad \cQ_1|_{\fbf}=\bbN_0+ h, \quad \cR_1|_{\ff}=\bbN_0+1, \quad \cR_1|_{\fbf}=\bbN_0+ h+1.
\label{dt.12b}\end{equation}
 Moreover, the leading term $A$ of $Q_1$ at $\fbf$ is such that $A=\Pi_h A\Pi_h $.  
\label{dt.12}\end{proposition}
\begin{proof}
We need to find $Q_1'$ such that 
\begin{equation}
 N_{\ff}(x^{-\delta}D_{\phi}x^{\delta}Q_1')= N_{\ff}(R_0),
\label{dt.13}\end{equation}
for then it suffices to take $Q_1=Q_0+ Q_1'$.  To solve \eqref{dt.13}, we can decompose $N_{\ff}(R_0)$ using the fiberwise projection $\Pi_h$ onto the bundle $\ker D_v$,
$$
N_{\ff}(R_0)= \Pi_h N_{\ff}(R_0) + (\Id-\Pi_h)N_{\ff}(R_0),
$$
where the right hand side makes sense since $\Pi_h$ can be regarded as an element of $\Psi^0_{\sus({}^{\phi}NY)}(\pa M/Y;E)$, the space of ${}^{\phi}NY$-suspended families of pseudodifferential operators of order $0$ of \cite{Mazzeo-MelrosePhi}.  Now, recall from \eqref{dt.3} that
\begin{equation}
   N_{\ff}(x^{-\delta}D_{\phi}x^{\delta})= N_{\ff}(D_{\phi})= D_v+\eth_h,
\label{dt.14}\end{equation} 
where $\eth_h$ is a family of Euclidean Dirac operators in the fibers of ${}^{\phi}N\pa M\to \pa M$ anti-commuting with $D_v$.  In particular, $\eth_h$ commutes with $\Pi_h$.  On the range of 
$\Id-\Pi_h$, the operator $D_v+\eth_h$ is on each fiber an invertible suspended operator in the sense of \cite{Mazzeo-MelrosePhi}, so has an inverse  $(D_v+\eth_h)^{-1}_{\perp}\in\Psi^{-1}_{\sus({}^{\phi}NY)}(\pa M/Y;E)$.  On the range of $\Pi_h$, we can apply instead \cite[Corollary~A.4]{ARS1} to invert $\eth_h$ as a weighted $b$-operator.  Thus, it suffices to take $Q_1'$ such that 
\begin{equation}
   N_{\ff}(Q_1')= (\eth_h)^{-1}\Pi_h(N_{\ff}(R_0)) + (D_v+\eth_h)^{-1}_{\perp} (\Id-\Pi_h)N_{\ff}(R_0).
\label{dt.15}\end{equation}

The price to pay is that by \cite[Corollary~A.4]{ARS1}, the image of $(\eth_h)^{-1}$ has an expansion at infinity with index set $J_{h+1}$ such that $\inf \Re(J_{h+1})=h$.  This expansion corresponds to a non-trivial expansion of $Q_1'$ at $\fbf$, so that $Q_1=Q_0+Q_1'\in\Psi^{-1,\cQ_1}_{\phi}(M;E)$ with $\cQ_1$ as in the statement of the proposition.  Since $Q_1$ is $\cO(x_{\fbf}^{h})$ at $\fbf$, the same is true for $R_1$.  Moreover, at $\ff$, we must have
$$
    N_{\ff}(x^{-\delta}D_{\phi}x^{\delta})N_{\ff}(Q_1)=\Id
$$
which means that
$$
     N_{\ff}(Q_1)= \Pi_{h}\eth_{h}^{-1}\Pi_h+(D_v+\eth_h)^{-1}_{\perp}. 
$$
Thus, the top order term $A$ at $\fbf$ of $Q_1$ comes from the expansion of $\Pi_h\eth^{-1}_h\Pi_h$, which is just a family of Green functions of Euclidean Dirac operators, those being of the form $\frac{\cl(u)}{|u|^{h+1}}$ in terms of the Euclidean variable $u$ and Clifford multiplication.  In particular, the index set of $Q_1$ at $\fbf$ is just $\bbN_0+h$.    Hence, choosing suitably the definition of $Q_1'$ on $\fbf$, we can assume that 
$A=\Pi_h A\Pi_h$.  Since by definition $D_v$ acts trivially on such an operator, this implies that $x^{-\delta}Dx^{\delta}Q_1$ vanishes instead at order $h+1$ at $\fbf$ so that $R_1$ must also be $\cO(x_{\fbf}^{h+1})$ at $\fbf$, that is, $\cR_1|_{\fbf}=\bbN_0+h+1$.    
\end{proof}

\subsection*{Step 2: preliminary step to remove the error term at $\fbf$}

Since $R_1$ has a term of order $h+1$ at $\fbf$, it cannot be compact as an operator acting on $L^2_b(M;E)$.  This is because being $\cO(x_{\fbf}^{h+1})$ in terms of right $\phi$-densities corresponds to being $\cO(1)$ in terms of right $b$-densities.  To get rid of the term of order $h+1$ at $\fbf$, we can first remove the expansion of this term at $\ff\cap \fbf$, in fact just the expansion of this term lying in the range of $\Pi_h$.  
\begin{proposition}
There exist $Q_2\in \Psi^{-1,\cQ_2}_{\phi}(M;E)$ and $R_2\in\Psi^{-\infty,\cR_2}_{\phi}(M;E)$ such that 
$$
    (x^{-\delta}D_{\phi}x^{\delta})Q_2= \Id -R_2,  
$$
where $\cQ_2$ and $\cR_2$ are index families given by the empty set at $\lf$ and $\rf$ and such that
$$
\cQ_2|_{\ff}= \bbN_0, \quad \cQ_2|_{\fbf}=\bbN_0+ h, \quad \cR_2|_{\ff}=\bbN_0+1, \quad \inf\Re(\cR_2|_{\fbf})\ge h+1.
$$
Moreover, the term $A$ of order $h$ at $\fbf$ of $Q_2$ is such that $A=\Pi_hA\Pi_h$, while $R_2$ is such that its term $B$ of order $h+1$ at $\fbf$ is such that $\Pi_hB=\Pi_h B\Pi_h$ has a trivial expansion at $\ff\cap \fbf$.   
\label{dt.16}\end{proposition}
\begin{proof}
Writing $Q_2=Q_1+\widetilde{Q}_2$, we need to find $\widetilde{Q}_2$ such that the term $B$ at order $h+1$ of  
$$
        (x^{-\delta}D_{\phi}x^{\delta})\widetilde{Q}_2-R_1
$$
at $\fbf$ is such that $\Pi_h B$ vanishes to infinite order at $\ff\cap \fbf$.  Let $r_1^o$ denote the part of the restriction at order $h+1$ to $\fbf$ of $R_1$ whose image is in the range of $\Pi_h$.  By step 1, $r_1^o=I(D_b,\delta)A_1$ where $A_1$ is the term of order $h$ of $Q_1$ at $\fbf$.  In particular, $r_1^o=\Pi_h r_1^o\Pi_h$.  Now, we need to find $q_2^o$ such that
$$
    I(D_b,\delta)(q_2^o)- r_1^o
$$
vanishes to infinite order at $\ff\cap \fbf$.  To construct such a term $q_{2}^o$, which can be achieved working locally near $\ff\cap \fbf$, the idea is to use \cite[Lemma~5.44]{MelroseAPS}.  We refer to \cite[Proposition~4.14 and Proposition~A.7]{ARS1} or \cite[Proposition~3.17]{Vaillant} for further details.  Extending $q_2^o$ smoothly off $\fbf$, thinking of it as a term of order $h$ there, we obtain $\widetilde{Q}_2$ as desired.  Clearly, $q_2^o=\Pi_hq_2^o\Pi_h$, so that the terms of order $h$ of $Q_2$ at $\fbf$ and the term of order $h+1$ of $R_2$ at $\fbf$ are as claimed. 
\end{proof}

\subsection*{Step 3: Removing the error term at $\fbf$}

\begin{proposition}
There exist $Q_3\in\Psi^{-1,\cQ_3}_{\phi}(M;E)$ and $R_3\in\Psi^{-\infty,\cR_3}_{\phi}(M;E)$ such that 
$$
     (x^{-\delta}D_{\phi}x^{\delta})Q_3=\Id- R_3,
$$
where $Q_3$ and $R_3$ are index families such that 
\begin{equation}
\begin{gathered}
\inf\Re(\cQ_3|_{\lf})\ge\mu, \quad \inf\Re(\cQ_3|_{\rf})\ge h+\mu, \quad \cQ_3|_{\ff}=\bbN_0, \quad \cQ_3|_{\fbf}=\bbN_0+ h, \\
\inf\Re(\cR_3|_{\lf})> \mu, \quad \inf\Re(\cR_3|_{\rf})\ge h+\mu, \quad \cR_3|_{\ff}=\bbN_0+1, \quad \cR_3|_{\fbf}=\bbN_0+ h+2. 
\end{gathered}
\end{equation}
Moreover, the term $A$ of order $h$ at $\fbf$ of $Q_3$ is such that $A=\Pi_hA\Pi_h$, while any term $r$ in the expansion of $R_3$ at $\rf$ is such that $r\Pi_h=r$. 
\label{dt.17}\end{proposition}
\begin{proof}
Let $B$ be the term of order $h+1$ of $R_2$ at $\fbf$ and write 
$$
     B= b^o+ b^{\perp}, \quad b^o=\Pi_h B=\Pi_h B\Pi_h, \quad b^{\perp}= (\Id-\Pi_h)B. 
$$
By Proposition~\ref{dt.16}, $b^o$ can be thought of as a smooth kernel on the interior of the $b$-front face $Y^2\times (0,\infty)_s$ of the $b$-double space of $Y \times [0,1)_x$, where $s=x/x'$ is the usual coordinate.  Hence, since $\delta$ is not a critical weight, this suggests to consider 
$$
    q_3^o(s)= \frac{1}{2\pi}\int_{\bbR} e^{i\xi \log s}I(D_b,\delta+i\xi)^{-1}\widehat{b}^o(\xi) d\xi,
$$
where 
$$
    \widehat{b}^o(\xi)= \int_{0}^{\infty} e^{-i\xi\log s}b^o(s) \frac{ds}{s}
$$
is the Mellin transform of $b^o(s).$  Since $I(D_b,\lambda)$ has no critical weight in $(\delta-\mu,\delta+\mu)$, notice that in the strip $\delta-\mu\le \Re \lambda\le \delta+\mu$, $I(D_b,\lambda)^{-1}$ has at most simple poles on the lines $\Re\lambda=\delta\pm\mu$, which means by the integral contour argument of \cite{MelroseAPS} that $e^{\mu|\log s|}q^o_3(s)$ is bounded.   
In this case, if $Q^o_3$ is a smooth extension of $q^o_3x^{-1}$, seen as a term of order $h$ at $\fbf$ in terms of $\phi$-densities, we have that 
$$
        (x^{-\delta}D_{\phi}x^{\delta}) Q^o_3= b^o
$$
at order $h+1$ at $\fbf$.  We can also assume that each term $a$ in the expansion of $Q^0_3$ at $\rf$ is such that $a\Pi_h=a$. Moreover, the boundedness of $e^{\mu|\log s|}q_3^o(s)$ ensures that  $Q^o_3$ has leading terms at least of order $x_{\lf}^{\mu}$ and $x_{\rf}^{h+\mu}$ at $\lf$ and $\rf$ respectively.  Since the term of order $\mu$ of $q_3^o(s)$ is killed by $x^{-\delta}D_{\phi}x^{\delta}$, we can assume the same is true for $Q_3^o$.  Hence, considering $\widetilde{Q}_3= Q_2+ Q^o_3$, we see that 
$$
       (x^{-\delta}D_{\phi}x^{\delta})\widetilde{Q}_3= \Id-\widetilde{R}_3
$$
with $\widetilde{R}_3$ similar to $R_2$, but with term $\widetilde{B}$ of order $h+1$ at $\fbf$ such that 
$$
     \Pi_h\widetilde{B}=0
$$
and with leading terms at $\lf$ and $\rf$ at least of order $x^{\mu+\nu}_{\lf}$ and $x^{h+\mu}_{\rf}$ for some $\nu>0$.  
Moreover, each term $r$ in the expansion of $\widetilde{R}_3$ at $\rf$ is such that $r\Pi_h=r$.  To get rid of $\widetilde{B}$, it suffices then to consider 
$$
     q_3^{\perp}:= D_v^{-1}\widetilde{B}
$$
and a smooth extension $Q_3^{\perp}$ having $q_3^{\perp}$ as a term of order $h+1$ at $\fbf$.  By construction, $x^{-\delta}D_{\phi}x^{\delta}Q^{\perp}_3$ has term of order $h+1$ at $\fbf$ precisely given by $\widetilde{B}$.  Hence it suffices to take $Q_3=\widetilde{Q}_3 + Q_3^{\perp}$.  Since $q^o_3=\Pi_hq^o_3\Pi_h$, notice that term of order $h$ at $\fbf$ of $Q_3$ is as claimed.  By our choice of $Q_3$, notice that $R_3$ has no term of order $\mu$ at $\lf$ and the expansion at $\rf$ is as claimed.  
\end{proof}

\subsection*{Step 4: Removing the expansion at $\lf$}

\begin{proposition}
There exist $Q_4\in \Psi^{-1,\cQ_4}_{\phi}(M;E)$ and $R_4\in\Psi^{-\infty,\cR_4}_{\phi}(M;E)$ such that
$$
      (x^{-\delta}D_{\phi}x^{\delta})Q_4=\Id-R_4,
$$
where $\cQ_4$ and $\cR_4$ are index families such that 
\begin{equation}
\begin{gathered}
\inf\Re(\cQ_4|_{\lf})\ge\mu, \quad \inf\Re(\cQ_4|_{\rf})\ge h+\mu, \quad \cQ_4|_{\ff}=\bbN_0, \quad \cQ_4|_{\fbf}=\bbN_0+ h, \\
\cR_4|_{\lf}=\emptyset, \quad \inf\Re(\cR_4|_{\rf})\ge h+\mu, \quad \cR_4|_{\ff}=\bbN_0+1, \quad \cR_4|_{\fbf}=\bbN_0+h+2. 
\end{gathered}
\end{equation}
Moreover, the term $A$ of order $h$ of $Q_4$ at $\fbf$ is such that $A=\Pi_h A\Pi_h$.
\label{dt.18}\end{proposition}
\begin{proof}
Proceeding as in the proof of \cite[Lemma~5.44]{MelroseAPS}, we can find $\widetilde{Q}_4$ defined near $\lf$ such that $(x^{-\delta}D_{\phi}x^{\delta})\widetilde{Q}_4$ has the same expansion as $R_3$ at $\lf$.  Indeed, if $R_3$ has a term $x^{\alpha}r_{\alpha}$  of order $\alpha$ in its expansion at $\lf$, then we can first look at $\Pi_h r_{\alpha}$ and look for $q_{\alpha}$ such that
$$
  I(D_b,\delta+\alpha-1)q_{\alpha}= \Pi_h r_{\alpha}.
$$ 
This is possible provided $ I(D_b,\delta+\alpha-1)$ is invertible, in which case we have that
$$
     (x^{-\delta} D_{\phi}x^{\delta})x^{\alpha-1}q_{\alpha}= x^{\alpha}\Pi_h r_{\alpha}+ x^{\alpha}r_{\alpha}^{\perp}+ \cO(x^{\alpha+1}),
$$
where $r_{\alpha}^{\perp}$ is such that $\Pi_h r_{\alpha}^{\perp}=0$ by Lemma~\ref{dt.7}.  Hence, picking $q^{\perp}_{\alpha}$ such that 
$$
    D_vq^{\perp}_{\alpha}= (r_{\alpha}-\Pi_hr_{\alpha}- r_{\alpha}^{\perp}),
$$
we see that 
$$
    (x^{-\delta} D_{\phi}x^{\delta})(x^{\alpha-1}q_{\alpha}+ x^{\alpha}q_{\alpha}^{\perp})= x^{\alpha}r_{\alpha}+ \cO(x^{\alpha+1}),
$$
so that we found a way to remove the term $x^{\alpha}r_{\alpha}$.  If instead $I(D_b,\delta+\alpha-1)$ is not invertible, we can remove $\Pi_h r_{\alpha}$ by replacing $x^{\alpha-1}q_{\alpha}$ by a term of the form
$$
       x^{\alpha-1}(q_{\alpha}+ q_{\alpha,1}\log x),
$$
and then proceeding as before.  Similarly, for a term of order $x^{\alpha}(\log x)^k r_{\alpha,k}$, we have more generally to replace $x^{\alpha-1}q_{\alpha}$ by $x^{\alpha-1}(\log x)(q_{\alpha}+q_{\alpha,1}\log x)$.  In any case, we can in this manner recursively remove all the terms in the expansion of $R_3$ at $\lf$, sot that $\widetilde{Q}_4$ can be obtained by taking a Borel sum.
Hence, setting $Q_4=Q_3+\widetilde{Q}_4$ gives the desired operator.
\end{proof}

\subsection*{Step 5:  Proof of Theorem~\ref{dt.10}}

\begin{proof}[Proof of Theorem~\ref{dt.10}]
To prove the theorem, we can remove the expansions of the error term at $\fbf$ and $\ff$ by using a Neumann series argument.  First, choose $S_5$ to be an asymptotic sum 
$$
    S_5\sim \sum_{i=1}^{\infty} R_4^i
$$
at $\fbf$ and $\ff$.  This is possible since from the composition rules of fibered boundary operators (see for instance \cite[Theorem~2.11]{Vaillant}), the index family of $(R_4)^i$ iterates away at $\fbf$ and $\ff$ while it is stabilizing at $\rf$.  Taking 
$Q= Q_4(\Id+S_5)$ then gives the desired operator.

\end{proof}

The parametrix of Theorem~\ref{dt.10} has various implications.

\begin{corollary}
If $\sigma\in x^{\alpha}H^{-\infty}_b(M;E)$ with $\alpha\in\bbR$  is such that  $f:=D_{\phi}\sigma\in \cA^{\cF}_{\phg}(M;E)$ for some index set $\cF$, then $\sigma\in \cA^{\cE}_{\phg}(M;E)$ for some  index set $\cE$ depending on $\cF$ and $\alpha$ such that $\inf\Re\cE>\alpha$.
\label{dt.19}\end{corollary}
\begin{proof}
Take $\delta\ge-\alpha$ large enough so that  $x^{\delta-1}f\in L^2_b(M;E)$ and $\delta-1$ is not a critical weight of $I(D_b,\lambda)$.  By Theorem~\ref{dt.10}, there exist $Q\in\Psi^{-1,\cQ}_{\phi}(M;E)$ and $R\in\Psi^{-\infty,\cR}_{\phi}(M;E)$ such that 
$$
      (x^{1-\delta}D_{\phi}x^{\delta-1})Q= \Id -R.
$$
Conjugating by $x$, this gives the following parametrix for the corresponding fibered cusp operator,
$$
    x^{-1}(x^{1-\delta}D_{\phi}x^{\delta-1})Qx= \Id -x^{-1}Rx.$$
Taking the adjoint and using that $D_{\phi}$ is formally self-adjoint, we find that 
\begin{equation}
       xQ^*x^{-1}(x^{\delta}D_{\phi}x^{-\delta})= \Id-xR^*x^{-1}.
\label{dt.20}\end{equation}
Applying both side of this equation to $x^{\delta}\sigma\in H^{-\infty}_b(M;E)$ yields
\begin{equation}
         x^{\delta}\sigma= xQ^*x^{\delta-1}f+ (xR^*x^{-1})x^{\delta}\sigma.
\label{dt.21}\end{equation}
Now, $(xRx^{-1})^*x^{\delta}\sigma$ is well-defined since $xR^*x^{-1}$ vanishes rapidly at $\rf, \ff$ and $\fbf$ and $(xR^*x^{-1})x^{\delta}\sigma\in \cA^{\cR|_{\rf}-h}_{\phg}(M;E)$.  On the other hand, since by our assumption on $\delta$, $\inf\Re (\cQ|_{\lf}+\cF+\delta-1)>0$, we can apply Proposition~\ref{phi.17b} to conclude that $xQ^*x^{\delta-1}f\in\cA^{\cG}_{\phg}(M;E)$ with
$$
   \cG= (\cQ|_{\rf}-h)\overline{\cup}(\cQ|_{\ff}+\cF+\delta)\overline{\cup}(\cQ|_{\fbf}+\cF+\delta-h-1).
$$ 
Hence, we see from \eqref{dt.21} that $\sigma$ is polyhomogeneous, from which the result follows.
\end{proof}

The particular case where $f=0$ yields the following.

\begin{corollary}
For each $\alpha\in\bbR$, the kernel of $D_{\phi}$ in $x^{\alpha}L^2_b(M;E)$ is finite dimensional and its elements are polyhomogeneous.  Moreover, if for some $\mu>0$, $(\alpha-\mu,\alpha+\mu)\cap\Crit(D_b)=\emptyset$, then elements of that kernel have their leading term at least of order $x^{\alpha+\mu}$ in their polyhomogeneous expansion at $\pa M$.  
\label{dt.22}\end{corollary}
\begin{proof}
Polyhomogeneity is a consequence Corollary~\ref{dt.19} with $f=0$.  With $\delta\ge -\alpha$ as in the proof of Corollary~\ref{dt.19}, the finite dimension of the kernel follows from the fact that $x^{\delta}\ker_{x^{\alpha}L^2_b}D_{\phi}\subset \ker_{L^2_b}(x^{\delta}D_{\phi}x^{-\delta})$ and that by \eqref{dt.21}, $xR^*x^{-1}$, which is a compact operator when acting on $L^2_b(M;E)$, restricts to be the identity on this subspace.  

Now, if $(\alpha-\mu,\alpha+\mu)\cap\Crit(D_b)=\emptyset$, then by Remark~\ref{dt.9b} we can take $\delta=-\alpha$, so that
$$
          x^{\delta}\sigma=  (xR^*x^{-1})x^{\delta}\sigma \in \cA_{\phg}^{\cR|_{\rf}-h}
$$
with $\inf\Re \cR|_{\rf}\ge h+\mu$, so that $\sigma$ has leading term of order at least $x^{-\delta+\mu}=x^{\alpha+\mu}$ in its polyhomogeneous expansion at $\pa M$.  
\end{proof}

The parametrix of Theorem~\ref{dt.10} can also be used to obtain a Fredholm criterion.  Let $\cD_{\delta-1}$ be the minimal domain of the fibered cusp operator $x^{1-\delta}(x^{-1}D_{\phi})x^{\delta-1}$ acting on $L^2_b(M;E)$.  Since the fibered cusp metric $g_{\fc}:=x^2g_{\phi}$ is complete, recall that a standard argument shows that there is in fact only one closed extension since the maximal domain is equal to the minimal domain.  Let $\widetilde{\Pi}_h$ be a smooth extension of $\Pi_h$, first to a collar neighborhood of $\pa M$ and then to all of $M$ using a cut-off function.  Then one can readily check that 
\begin{equation}   
\cD_{\delta-1}= \widetilde{\Pi}_hH^1_b(M;E)+ x(\Id-\widetilde{\Pi}_h)x^{-\frac{h+1}2}H^1_{\phi}(M;E).
\label{dt.23}\end{equation}

\begin{corollary}
If $\delta-1$ is not a critical weight of $I(D_b,\lambda)$, then
$$
         D_{\phi}: x^{\delta-1}\cD_{\delta-1}\to x^\delta L^2_b(M;E)
$$
is a Fredholm operator, that is, the operator
$$
      \eth_{\phi}: x^{\delta-1}x^{\frac{h+1}2}\cD_{\delta-1}\to x^{\delta}L^2_{\phi}(M;E)
$$
is Fredholm, where $x^{\delta-1}x^{\frac{h+1}2}\cD_{\delta-1}$ is a domain in $x^{\delta-1}L^2_{\phi}(M;E)$.
\label{dt.24}\end{corollary}
\begin{proof}
We need to construct a parametrix for $x^{-1}D_{\phi}$ acting formally on $x^{\delta-1}L^2_b(M;E)$, that is, we need a parametrix for $x^{1-\delta}(x^{-1}D_{\phi})x^{\delta-1}$ acting formally on $L^2_b(M;E)$.  First,  since $\delta-1$ is not a critical weight, we know by  Theorem~\ref{dt.10} that  there exist $Q_{\delta-1}\in\Psi^{-1,\cQ_{\delta-1}}_{\phi}(M;E)$ and $R_{\delta-1}\in\Psi^{-\infty,\cR_{\delta-1}}_{\phi}(M;E)$ such that 
$$
         (x^{1-\delta}(D_{\phi})x^{\delta-1})Q_{\delta-1}= \Id- R_{\delta-1}.
$$
Conjugating by $x$ then gives
\begin{equation}
  (x^{1-\delta}(x^{-1}D_{\phi})x^{\delta-1})(Q_{\delta-1}x)= \Id - x^{-1}R_{\delta-1}x.
\label{dt.25}\end{equation}
Similarly, by Remark~\ref{dt.9b}, $-\delta$ is not a critical weight, so we see from Theorem~\ref{dt.10} that there exist $Q_{-\delta}\in\Psi^{-1,\cQ_{-\delta}}_{\phi}(M;E)$ and $R_{-\delta}\in\Psi^{-\infty,\cR_{-\delta}}_{\phi}(M;E)$ such that 
$$
          (x^{\delta}D_{\phi}x^{-\delta})Q_{-\delta}= \Id -R_{-\delta}.
$$
Taking the adjoint and using that $D_{\phi}$ is formally self-adjoint, we thus see, after conjugating by $x$, that
\begin{equation}
    (xQ^*_{-\delta})x^{1-\delta}(x^{-1}D_{\phi})x^{\delta-1}= \Id -xR^*_{-\delta}x^{-1}.  
\label{dt.26}\end{equation}
Since the terms $A_{\delta-1}$ and $A^*_{-\delta}$ of order $h$ of $Q_{\delta-1}$ and $Q^*_{-\delta}$ are such that $A_{\delta-1}=\Pi_hA_{\delta-1}\Pi_h$ and $A^*_{-\delta}=\Pi_h A^*_{-\delta}\Pi_h$, we see that $Q_{\delta-1}$ and $Q^*_{-\delta}$ induce bounded operators
$$
     Q_{\delta-1}x: L^2_b(M;E)\to \cD_{\delta-1}, \quad xQ^*_{-\delta}: L^2_b(M;E)\to \cD_{\delta-1}.
$$  

Hence, since both $x^{-1}R_{\delta-1}x$ and $xR^{*}_{-\delta}x^{-1}$ act as compact operators on $L^2_b(M;E)$ and $\cD_{\delta-1}$, we deduce from \eqref{dt.25} and \eqref{dt.26} that 
$$
            x^{1-\delta}(x^{-1}D_{\phi})x^{\delta-1}: \cD_{\delta-1}\to L^2_b(M;E)
$$
is Fredholm, from which the result follows.
\end{proof}

\begin{corollary}
If $\delta$ is not a critical weight, then 
$$
   D_{\phi}: x^{\delta}\widetilde{\Pi}_hH^1_b(M;E)+ x^{\delta}(\Id-\widetilde{\Pi}_h)x^{-\frac{h+1}2}H^1_{\phi}(M;E)\;\longrightarrow \;x^{\delta+1}\widetilde{\Pi}_{h}L^2_b(M;E)+ x^{\delta}(\Id-\widetilde{\Pi}_h)L^2_{b}(M;E)
$$
is a Fredholm operator.  
\label{dt.27}\end{corollary}
\begin{proof}
We need to show that 
\begin{equation}
   x^{-\delta}D_{\phi}x^{\delta}: \widetilde{\Pi}_hH^1_b(M;E)+ (\Id-\widetilde{\Pi}_h)x^{-\frac{h+1}2}H^1_{\phi}(M;E)\;\longrightarrow \;x\widetilde{\Pi}_{h}L^2_b(M;E)+(\Id-\widetilde{\Pi}_h)L^2_{b}(M;E)
\label{dt.27b}\end{equation}
is a Fredholm operator.  
By Remark~\ref{dt.9b}, we know that both $\delta$ and $-1-\delta$ are not critical weights.  Hence, applying Theorem~\ref{dt.10} gives operators 
$Q_{\delta}, R_{\delta}, Q_{-1-\delta}$ and $R_{-1-\delta}$ such that
\begin{gather}
\label{dt.28}  (x^{-\delta}D_{\phi}x^{\delta})Q_{\delta}= \Id-R_{\delta}, \\
\label{dt.29}    (xQ^*_{-1-\delta}x^{-1})(x^{-\delta}D_{\phi}x^{\delta})=\Id -xR_{-1-\delta}^*x^{-1}.
\end{gather}
Thanks to the fact that each term $r$ of order $h+1$ or less in the expansion of $R_{\delta}$ is such that $r\Pi_h=r$, we see that $R_{\delta}$ is a compact operator when acting on
$$
  x\widetilde{\Pi}_{h}L^2_b(M;E)+(\Id-\widetilde{\Pi}_h)L^2_{b}(M;E)\subset L^2_b(M;E).
  $$
Hence, we see from \eqref{dt.28} that $Q_{\delta}$ is a right inverse modulo compact operators.  On the other hand,
since $xR^*_{-1-\delta}x^{-1}$ is a compact operator when acting on
$$
\widetilde{\Pi}_hH^1_b(M;E)+ (\Id-\widetilde{\Pi}_h)x^{-\frac{h+1}2}H^1_{\phi}(M;E)\subset L^2_b(M;E),
$$
we see from \eqref{dt.29} that $xQ_{-1-\delta}^*x^{-1}$ is a left inverse modulo compact operators.   Hence, we see that \eqref{dt.27b} is invertible modulo compact operators and must therefore be Fredholm. 

\end{proof}

\begin{remark}
When $\eth_{\phi}$ is the Hodge-deRham operator of $g_{\phi}$,  Corollaries~\ref{dt.19}, \ref{dt.22}, \ref{dt.24} and \ref{dt.27} correspond to \cite[Proposition~16]{HHM2004}.
\end{remark}

Finally, we can use Theorem~\ref{dt.10} to give a pseudodifferential description of the inverse of $D_{\phi}$ when it is inverted as a Fredholm operator.  More precisely, for $\delta-1$ not a critical weight, consider the Fredholm operator
\begin{equation}
   x^{1-\delta}(x^{-1}D_{\phi})x^{\delta-1}: \cD_{\delta-1}\to L^2_{b}(M;E)
\label{dt.30}\end{equation}
of Corollary~\ref{dt.24}.  Let $\mu>0$ be such that the interval $(\delta-1-\mu,\delta-1+\mu)$ contains no critical weight of the indicial family $I(D_b,\lambda)$.  By Remark~\ref{dt.9b}, the interval $(-\delta-\mu,-\delta+\mu)$ is a also free of critical weights of the indicial family $I(D_b,\lambda)$.   Let $P_1$ be the orthogonal projection in $L^2_b(M;E)$ onto the kernel of \eqref{dt.30}.  By Corollary~\ref{dt.22}, $P_1\in\Psi^{-\infty,\cE}(M;E)$ is a very residual operator in the sense of \cite{MazzeoEdge}, where $\cE=(\cE_{\lf},\cE_{\rf})$ is an index family with $\inf\Re\cE_{\lf}\ge\mu$ and $\inf\Re\cE_{\rf}\ge\mu$.   
Similarly, let $P_2$ be the orthogonal projection onto the orthogonal complement of the range of \eqref{dt.30} in $L^2_b(M;E)$.  From the formal self-adjointness of $D_{\phi}$, on can check that the orthogonal complement of the range of \eqref{dt.30} is given by $\ker_{L^2_b}(x^{\delta}D_{\phi}x^{-\delta})$.  By Corollary~\ref{dt.22}, this space is finite dimensional and its elements are polyhomogeneous.  Hence, we also have that $P_2\in\Psi^{-\infty,\cF}(M;E)$ is very residual with $\cF=(\cF_{\lf},\cF_{\rf})$ an index family such that $\inf\Re\cF_{\lf}\ge \mu$ and $\inf\Re\cF_{\rf}\ge \mu$.   

Now, by Corollary~\ref{dt.24}, there is a bounded operator $G_{\delta-1}: L^2_b(M;E)\to \cD_{\delta-1}$ such that
\begin{gather}
G_{\delta-1}(x^{1-\delta}(x^{-1}D_{\phi})x^{\delta-1})= \Id -P_1, \\
(x^{1-\delta}(x^{-1}D_{\phi})x^{\delta-1})G_{\delta-1}=\Id -P_2.
\label{dt.30b}\end{gather}
\begin{corollary}
Suppose $\delta-1$ is not a critical weight of the indicial family $I(D_b,\lambda)$.  Let $\mu>0$ be such that $(\delta-1-\mu,\delta-1+\mu)\cap \Crit(D_b)=\emptyset$.  Then the inverse $G_{\delta-1}$ is an element of $\Psi^{-1,\cG}_{\phi}(M;E)$ with index family $\cG$ such that 
$$
   \inf \Re(\cG|_{\lf})\ge\mu, \quad \inf \Re(\cG|_{\rf})\ge h+1+\mu, \quad \inf \Re (\cG|_{\fbf})\ge h+1, \quad \inf \Re(\cG|_{\ff})\ge 1.
$$
Moreover, the term $A$ of order $h+1$ at $\fbf$ of $G_{\delta-1}$ is such that $A=\Pi_h A\Pi_h$.
\label{dt.31}\end{corollary}
\begin{proof}
We follow the approach of \cite[Theorem~4.20]{MazzeoEdge}.  Using \eqref{dt.25}, we have that
\begin{equation}
\begin{aligned}
G_{\delta-1}&= G_{\delta-1}\Id= G_{\delta-1}\left[ (x^{1-\delta}(x^{-1}D_{\phi})x^{\delta-1})(Q_{\delta-1}x)+ x^{-1}R_{\delta-1}x   \right] \\
&= (\Id-P_1)Q_{\delta-1} x+ G_{\delta-1}(x^{-1}R_{\delta-1}x).
\end{aligned}
\label{dt.32}\end{equation}
Using instead \eqref{dt.26}, we have that
\begin{equation}
\begin{aligned}
G_{\delta-1}&= \Id G_{\delta-1}= \left[ (xQ^*_{-\delta}) (x^{1-\delta}(x^{-1}D_{\phi})x^{\delta-1})+ xR^*_{-\delta}x^{-1} \right]G_{\delta-1} \\
&=(xQ^*_{-\delta})(\Id-P_2)+ xR^*_{-\delta}x^{-1}G_{\delta-1}.
\end{aligned}
\label{dt.33}\end{equation}
Thus, inserting \eqref{dt.33} into \eqref{dt.32}, we find that
\begin{equation}
G_{\delta-1}= Q_{\delta-1}x- P_1(Q_{\delta-1}x)+ xQ^*_{-\delta}x^{-1}R_{\delta-1}x- (xQ^*_{-\delta})P_2(x^{-1}R_{\delta-1}x)+ xR^*_{-\delta}x^{-1}G_{\delta-1}(x^{-1}R_{\delta-1}x).
\label{dt.34}\end{equation}
Since $xR^*_{-\delta}x^{-1}$ and $x^{-1}R_{\delta-1}x$ are very residual and $G_{\delta-1}$ is a bounded operator on $L^2_b(M;E)$, we see by the semi-ideal property of very residual operators that the last term in \eqref{dt.34} is very residual.  Hence, the result follows from \eqref{dt.34} and the result about composition of fibered boundary operators.

\end{proof}

\section{The low energy fibered boundary operators} \label{kfb.0}

In this section, we will introduce the natural calculus of pseudodifferential operators  associated to the low energy limit of  Dirac fibered boundary operators.  First,
on the manifold $M\times [0,\infty)_k$, we consider the lift of fibered boundary vector fields
\begin{equation}
\cV_{k,\phi}(M\times [0,\infty)_k)= \{\xi \in \cV(M\times [0,\infty)_k) \; | \; (\pr_2)_*\xi=0, \quad \xi|_{M\times \{k\}}\in \cV_{\phi}(M)\; \forall k\in[0,\infty)\}, 
\label{kfb.1}\end{equation}
where $\pr_2:M\times [0,\infty)_k\to [0,\infty)_k$ is the projection on the second factor.  We can also consider this lift on the transition single space of \cite{Kottke}
\begin{equation}
   M_t= [M\times [0,\infty)_k; \pa M\times \{0\}],
\label{kfb.2}\end{equation}
where we denote by $\sc$, $\zf$ and $\tf$ the boundary hypersurfaces of the lifts of $\pa M\times [0,\infty)_k$, $M\times \{0\}$ and $\pa M\times \{0\}$.
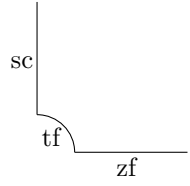
\begin{figure}[h]
\begin{tikzpicture}
\draw(0.5,0) arc [radius=0.5, start angle=0, end angle=90];
\draw(0.5,0)--(2,0);
\draw(0,0.5)--(0,2); 
\node at (1.2,-0.2) {$\zf$};
\node at (-0.2,1.2) {$\sc$};
\node at (0.2,0.2) {$\tf$};
\end{tikzpicture}
\caption{The transition single space $M_t$ }
\label{fig.3}\end{figure}

\begin{definition}
The \textbf{Lie algebra of $k,\phi$-vector fields} is the Lie algebra of vector fields on $M_{t}$ generated over $\CI(M_t)$ by the lift of vector fields in $\cV_{k,\phi}(M\times [0,\infty)_k)$ to $M_t$.  The space of \textbf{differential $k,\phi$-operators} is the universal enveloping algebra over $\CI(M_t)$ of $\cV_{k,\phi}(M_t)$.  In other words, the space $\Diff^m_{k,\phi}(M_t)$ of differential $k,\phi$-operators of order $m$ is generated by multiplication by elements of $\CI(M_{t})$ and up to $m$ vector fields in $\cV_{k,\phi}(M_t)$. 
\label{kfb.3}\end{definition}
If $E$ and $F$ are vector bundles on $M_t$, one can consider more generally the space
\begin{equation}
\Diff^m_{k,\phi}(M_t;E,F):= \Diff^m_{k,\phi}(M_t) \otimes_{\CI(M_t)}\CI(M_t;E^*\otimes F).
\label{kfb.4}\end{equation}

Using the Serre-Swan theorem, there is in fact a vector bundle ${}^{k,\phi}TM_t\to M_t$, the $k,\phi$-tangent bundle, inducing the natural identification 
\begin{equation}
      \cV_{k,\phi}(M_t)= \CI(M_t;{}^{k,\phi}TM_t).
\label{la.1}\end{equation}  
This identification is induced by an anchor map $a: {}^{k,\phi}TM_t\to TM_t$ giving ${}^{k,\phi}TM_t$ a Lie algebroid structure.  By construction, on $\zf$, ${}^{k,\phi}TM_t$ is just the $\phi$-tangent bundle of Mazzeo-Melrose \cite{Mazzeo-MelrosePhi},
\begin{equation}
    {}^{k,\phi}TM_t|_{\zf}\cong {}^{\phi}TM.
\label{la.2}\end{equation}
On the other hand, on $\sc$ and $\tf$, let ${}^{k,\phi}N_{\sc}$ and ${}^{k,\phi}N_{\tf}$ be the kernels of the anchor map, so that there are the short exact sequences of vector bundles 
\begin{equation}
\begin{gathered}
\xymatrix{
0 \ar[r] & {}^{k,\phi}N_{\sc}\ar[r] & {}^{k,\phi}TM_t|_{\sc} \ar[r]^{a}&  {}^{k,\phi}V_{\sc}\ar[r] & 0,
} \\
\xymatrix{
0 \ar[r] & {}^{k,\phi}N_{\tf}\ar[r] & {}^{k,\phi}TM_t|_{\tf} \ar[r]^{a}&  {}^{k,\phi}V_{\tf}\ar[r] & 0,
}
\end{gathered}
\label{la.3}\end{equation}
where ${}^{k,\phi}V_{\sc}$ and ${}^{k,\phi}V_{\tf}$ are the vertical tangent bundles associated to the fiber bundles 
\begin{equation}
   \phi_{\sc}:=\phi\times \Id_{[0,\infty)_k}: \sc\to Y\times [0,\infty)_k, \quad \phi_{\tf}:= \phi\times \Id_{[0,\frac{\pi}2]_{\theta}}: \tf\to Y\times [0,\frac{\pi}2]_{\theta},
\label{la.4}\end{equation}
induced by the $\phi$ and the natural identifications $\sc\cong\pa M\times [0,\infty)_k$ and $\tf\cong \pa M\times [0,\frac{\pi}2]_{\theta}$, the function $\theta= \arctan\frac{x}{k}$ being the natural angular coordinate on $\tf$.
Using  the coordinates \eqref{phi.2}, we can consider the coordinates 
$$
     X=\frac{x}{k}, k, y_1,\ldots, y_h, z_1,\ldots,z_v
$$
near $\sc$ on $M_t$, in terms of which ${}^{k,\phi}TM_t$ is locally spanned by
$$
   kX^2\frac{\pa}{\pa X}, kX\frac{\pa}{\pa y_1}, \ldots, kX\frac{\pa}{\pa y_h}, \frac{\pa}{\pa z_1}, \ldots, \frac{\pa}{\pa z_v},  
$$
so that ${}^{k,\phi}N_{\sc}$  and ${}^{k,\phi}N_{\tf}$ are locally spanned by 
$$
kX^2\frac{\pa}{\pa X}, kX\frac{\pa}{\pa y_1}, \ldots, kX\frac{\pa}{\pa y_h}
$$
on $\sc$ and $\tf$ respectively.

The vector bundles ${}^{k,\phi}N_{\sc}$ and ${}^{k,\phi}N_{\tf}$ are in fact pull-backs of vector bundles with respect to $\phi_{\sc}$ and $\phi_{\tf}$.  To see this, notice that the fiber bundles $\phi_{\sc}$ and $\phi_{\tf}$ induce as well the short exact sequences
\begin{equation}
\begin{gathered}
\xymatrix{
0 \ar[r] & {}^{k,\phi}V_{\sc} \ar[r] & {}^{k,\phi}TM_t|_{\sc} \ar[r]^{(\phi_{\sc})_*} & \phi_{\sc}^*({}^{k,\phi}N_{\sc}Y)\ar[r] & 0,  
}
\\
\xymatrix{
0 \ar[r] & {}^{k,\phi}V_{\tf} \ar[r] & {}^{k,\phi}TM_t|_{\tf} \ar[r]^{(\phi_{\tf})_*} & \phi_{\tf}^*({}^{k,\phi}N_{\tf}Y)\ar[r] & 0,  
}
\end{gathered}
\label{la.5}\end{equation}
with ${}^{k,\phi}N_{\sc}Y= \pr_1^*NY$ and ${}^{k,\phi}N_{\tf}Y= \widetilde{\pr}_1^* NY$, where $NY={}^{\sc}T(Y\times [0,1))|_{Y\times\{0\}}\cong TY\times \bbR$ is the restriction of the scattering tangent bundle on $Y\times [0,1)$ to the boundary $Y\times \{0\}$ and $\pr_1$, $\widetilde{\pr}_1$ denote the projections onto $Y$ in the Cartesian product $Y\times [0,\infty)_k$ and $Y\times[0,\frac{\pi}2]_{\theta}$ respectively.  In particular, the inclusions ${}^{k,\phi}V_{\sc}\to {}^{k,\phi}TM_t|_{\sc}$ and ${}^{k,\phi}V_{\tf}\to {}^{k,\phi}TM_t|_{\tf}$ induce splitting for the short exact sequences in \eqref{la.3},
\begin{equation}
           {}^{k,\phi}TM_t|_{\sc}= {}^{k,\phi}N_{\sc}\oplus {}^{k,\phi}V_{\sc}, \quad {}^{k,\phi}TM_t|_{\tf}= {}^{k,\phi}N_{\tf}\oplus {}^{k,\phi}V_{\tf}.
\label{la.6}\end{equation}
Hence, we see from \eqref{la.5} and \eqref{la.6} that
\begin{equation}
\begin{gathered}
{}^{k,\phi}N_{\sc}= (\phi_{\sc})_*{}^{k,\phi}TM_t|_{\sc}= \phi_{\sc}^*({}^{k,\phi}N_{\sc}Y), \\
{}^{k,\phi}N_{\tf}= (\phi_{\tf})_*{}^{k,\phi}TM_t|_{\tf}= \phi_{\tf}^*({}^{k,\phi}N_{\tf}Y),
\end{gathered}
\label{la.7}\end{equation}

If $\phi: \pa M\to Y$ is the identity map with $Y=\pa M$, $\cV_{\phi}(M)$ corresponds to the Lie algebra of scattering vector fields $\cV_{\sc}(M)$.  In this case, we denote $\cV_{k,\phi}(M_t)$ by $\cV_{k,\sc}(M_t)$.  One can check that the vector fields of this Lie algebra, as elements of $\cV(M_t)$, vanish to order one at the boundary hypersurface $\tf$ corresponding to the blow-up of $\pa M\times \{0\}$.     
\begin{definition}(\cite{Kottke}) In the case $\pa M=Y$ and $\phi=\Id$, the \textbf{Lie algebra of transition vector fields} on $M_t$ is given by 
$$
        \cV_t(M_t):= \frac{1}{x_{\tf}}\cV_{k,\sc}(M_t),
$$
where $x_{\tf}$ is a choice of boundary defining function for $\tf$.  The space of \textbf{differential transition operators} is the universal enveloping algebra over $\CI(M_t)$ of the Lie algebra of transition vector fields.  Thus, the space $\Diff^m_{t}(M_t)$ of differential transition operators of order $m$ is generated by multiplication by elements of $\CI(M_t)$ and up to $m$ transition vector fields.  For $E$ and $F$ vector bundles on $M_t$, we define more generally the space of differential transition operators of order $m$ acting from sections of $E$ to sections of $F$ by
$$
  \Diff^m_t(M_t;E,F):= \Diff^m_t(M_t)\otimes_{\CI(M_t)} \CI(M_t;E^*\otimes F).
$$  
\label{kfb.5}\end{definition}

To define the associated space of pseudodifferential operators, we need first to introduce a double space.
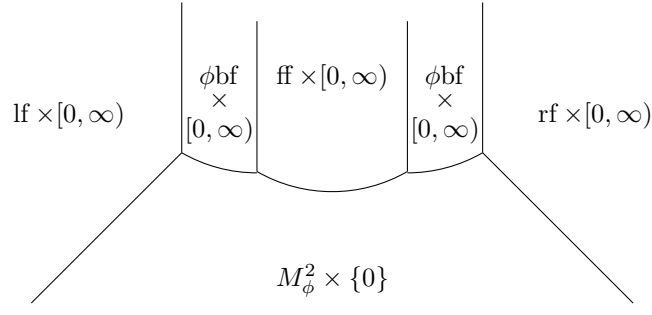
\begin{figure}[h]
\begin{tikzpicture}
\draw(0,0)--(-2,-2);
\draw(0,0)--(0,2);
\draw(0,0) arc [radius=2, start angle=240, end angle=270];
\draw(1,-0.25) arc [radius=2, start angle=240, end angle=300];
\draw(1,-0.25)--(1,1.75);
\draw(3,-0.25)--(3,1.75);
\draw(4,0) arc [radius=2, start angle=300, end angle=270];
\draw(4,0)--(4,2);
\draw(4,0)--(6,-2);
\node at (2,-1.7) {$M^2_{\phi}\times\{0\}$};
\node at (-1.5,0.5) {$\lf\times [0,\infty)$};
\node at (5.5,0.5) {$\rf\times[0,\infty)$};
\node at (2,1) {$\ff\times [0,\infty)$};
\node at (0.5,1) {$\fbf$};
\node at (0.5,0.7){$\times$};
\node at (0.5,0.3) {$[0,\infty)$};
\node at (3.5,1) {$\fbf$};
\node at (3.5,0.7){$\times$};
\node at (3.5,0.3) {$[0,\infty)$};\end{tikzpicture}
\caption{The space $M^2_{\phi}\times [0,\infty)_k$ }
\label{fig.4}\end{figure}

\begin{definition}
The $k,\phi$-double space associated to $(M,\phi)$ and a choice of boundary defining function $x\in\CI(M)$ is the manifold with corners 
\begin{equation}
   M^2_{k,\phi}= [M^2_{\phi}\times [0,\infty)_k; \fbf\times \{0\}, \lf\times \{0\}, \rf\times\{0\}, \ff\times \{0\}]
\label{kfb.6b}\end{equation}
with blow-down map 
$$
     \beta_{k,\phi}: M^2_{k,\phi}\to M^2\times[0,\infty)_k.
$$
\label{kfb.6}\end{definition}
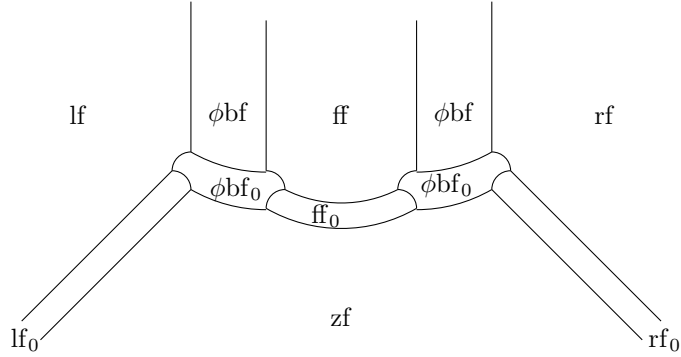
\begin{figure}[h]
\begin{tikzpicture}
\draw(0,0)--(-2,-2);
\draw(0,0) arc [radius=0.25, start angle=0, end angle=90];
\draw(-0.25,0.25)--(-2.25,-1.75);
\draw(-0.25,0.25) arc [radius=0.25, start angle=180, end angle=90];
\draw(0,0.5)--(0,2.5);
\draw(0,0) arc [radius=2, start angle=240, end angle=270];
\draw(0,0.5) arc [radius=2, start angle=240, end angle=270];
\draw(1,-0.25) arc [radius=0.25, start angle=180, end angle=90];
\draw(1.25,0) arc [radius=0.25, start angle=0, end angle=90];
\draw(1,0.25)--(1,2.25);
\draw(1,-0.25) arc [radius=2, start angle=240, end angle=300];
\draw(3,-0.25) arc [radius=0.25, start angle=0, end angle=90];
\draw(1.25,0) arc [radius=1.6610, start angle=243.158, end angle=296.842];
\draw(2.75,0) arc [radius=0.25, start angle=180, end angle=90];
\draw(3,0.25)--(3,2.25);
\draw(4,0) arc [radius=2, start angle=300, end angle=270];
\draw(4,0.5) arc [radius=2, start angle=300, end angle=270];
\draw(4,0) arc [radius=0.25, start angle=180, end angle=90];
\draw(4.25,0.25) arc [radius=0.25, start angle=0, end angle=90];
\draw(4,0.5)--(4,2.5);
\draw(4,0)--(6,-2);
\draw(4.25,0.25)--(6.25,-1.75);
\node at (2,-1.7) {$\zf$};
\node at (-2.2,-2) {$\lf_0$};
\node at (6.3,-2) {$\rf_0$};
\node at (-1.5,1) {$\lf$};
\node at (5.5,1) {$\rf$};
\node at (2,1) {$\ff$};
\node at (0.5,1) {$\fbf$};
\node at (3.5,1) {$\fbf$};
\node at(0.6,0) {$\fbf_0$};
\node at (3.4,0.05) {$\fbf_0$};
\node at (1.8,-0.35) {$\ff_0$};
\end{tikzpicture}
\caption{The $k,\phi$-double space $M^2_{k,\phi}$ }
\label{fig.5}\end{figure}

On $M^2_{k,\phi}$, the lifts of $M^2\times\{0\}$, $\lf\times [0,\infty)_k$, $\rf\times [0,\infty)_k$, $\ff\times [0,\infty)_k$ and $\fbf\times[0,\infty)_k$ will be denoted by $\zf, \lf, \rf, \ff$ and $\fbf$, while the new boundary hypersurfaces created by the blow-ups of $\fbf\times\{0\}$, $\ff\times\{0\}$, $\lf\times\{0\}$ and $\rf\times\{0\}$ will be denoted by
$\fbf_0$, $\ff_0$, $\lf_0$ and $\rf_0$.  

When $\phi$ is the identity map, the space $M^2_{k,\phi}=M^2_{k,\Id}$ is intimately related to the $b$-$\sc$ transition double space $M^2_{t}$ of \cite{Kottke} (denoted $M^2_{k,\sc}$ in \cite{GH1}),
\begin{equation}
  M_t^2:= [M^2_b\times [0,\infty)_k; \fb\times\{0\}, \Delta_b\cap \fb\times[0,\infty)_k, \lf\times\{0\}, \rf\times\{0\}] 
\label{kfb.6a}\end{equation}
with blow-down map 
$$
\beta_t: M^2_t\to M^2\times [0,\infty)_k,
$$ 
where $\Delta_b\subset M^2_b$ is the lifted diagonal, $\fb\subset M^2_b$ is the $b$-front face and $\lf$ and $\rf$ are the lifts of $\pa M\times M$ and $M\times \pa M$ to $M^2_b$.  If $\fb_0\subset M_t^2$ denotes the face created by the blow-up of $\fb\times\{0\}$, then the relation between $M^2_{k,\Id}$ and $M^2_t$ is given by 
\begin{equation}
   M^2_{k,\Id}=[M^2_{t}; \operatorname{bf}_0\cap \Delta_{k,\sc}],
\label{kfb.7}\end{equation}   
where $\Delta_{k,\sc}$ is the lift of the diagonal $\Delta_M\times [0,\infty)_k\subset M^2\times [0,\infty)_k$ to $M^2_{t}$.   Indeed, using the commutativity of blow-ups of Lemma~\ref{kqfb.7} below, one can check that in this setting, $M^2_{k,\Id}$ can alternatively be defined by
\begin{equation}
M^2_{k,\Id}=[ M^2_b\times[0,\infty)_k; \operatorname{bf}\times\{0\}, \Delta_b\cap\operatorname{bf}\times[0,\infty)_k, \Delta_b\cap \operatorname{bf}\times \{0\}, \lf\times\{0\}, \rf\times\{0\}].
\label{kfb.8}\end{equation}
  More precisely, Lemma~\ref{kqfb.7} is used to see that the blow-ups of $\fb\times\{0\}$ and $\Delta_b\cap\fb\times[0,\infty)_k$ commute provided we subsequently blow-up the lift of their intersection $\Delta_b\cap\fb\times\{0\}$.  As we will see, these two different, but nevertheless equivalent ways of constructing $M^2_t$ will be quite important for the construction of parametrices.

This can also be used to give the following alternative definition of $M^2_{k,\phi}$.  To see this, consider the $k,b$-double space 
\begin{equation}
M^2_{k,b}= [M^2_b\times[0,\infty)_k; \fb\times\{0\},\lf\times\{0\},\rf\times\{0\}].
\label{kfb.9c}\end{equation}

\begin{lemma}
The $k,\phi$-double space can alternatively be defined by 
\begin{equation}
   M^2_{k,\phi}=[M^2_{k,b}; \Phi_+, \Phi_0],
\label{kfb.9bb}\end{equation}
where $\Phi_+$ is the lift of $\Phi\times [0,\infty)_k\subset M^2_b\times [0,\infty)_k$ defined in \eqref{phi.11} to $M^2_{k,b}$ and $\Phi_0$ is the lift of $\Phi\times \{0\}$ to $M^2_{k,b}$.  
\label{kfb.9b}\end{lemma}
\begin{proof}
By Lemma~\ref{kqfb.7}, we can commute the blow-ups of $\fbf\times \{0\}$ and $\Phi\times [0,\infty)_k$ in Definition~\ref{kfb.6}, yielding
$$
 M^2_{k,\phi}= [M^2_b\times [0,\infty)_k; \fb\times\{0\},\Phi\times [0,\infty)_k, \Phi\times \{0\}, \lf\times \{0\}, \rf\times \{0\}].
$$ 
Since $\lf\times \{0\}$ and $\rf\times\{0\}$ do not intersect the lifts of $\Phi\times [0,\infty)_k$ and $\Phi\times \{0\}$ when $\fb\times \{0\}$ is first blown up, their blow-ups commute with those of $\Phi\times \{0\}$ and $\Phi\times[0,\infty)$, so that 
\begin{equation}
\begin{aligned}
  M^2_{k,\phi}&=[M^2_b\times [0,\infty)_k;\fb\times\{0\}, \lf \times \{0\}, \rf\times\{0\}, \Phi\times[0,\infty)_k,\Phi\times \{0\}] \\
  &= [M^2_{k,b}; \Phi_+,\Phi_0]
\end{aligned}  
\label{kfb.9d}\end{equation}
as claimed.  
\end{proof}

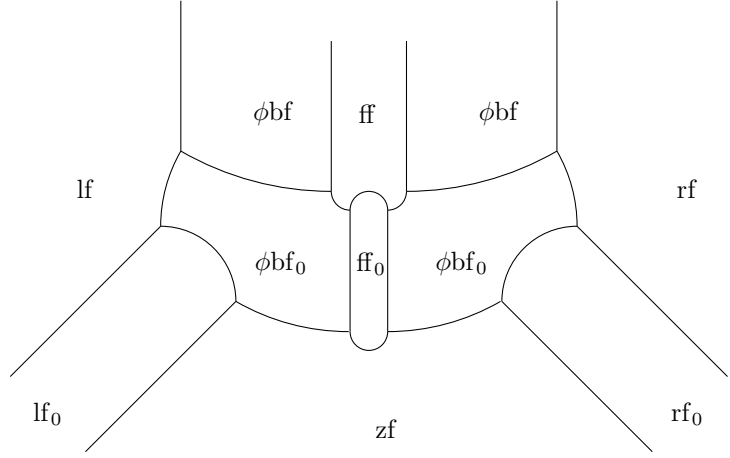
\begin{figure}[h]
\begin{tikzpicture}
\draw(0,0)--(-2,-2);
\draw(0,0) arc [radius=1, start angle=0, end angle=90];
\draw(-1,1)--(-3,-1);
\draw(-1,1) arc [radius=2, start angle=180, end angle=150];
\draw(-0.73205,2)--(-0.73205,4);
\draw(-0.73205,2) arc [radius=4, start angle=240, end angle=270];
\draw(1.26795,1.4641) arc[radius=0.25, start angle=180, end angle=270];
\draw(1.26795,1.4641)--(1.26795,3.4641);
\draw(1.51795,1.2141) arc[radius=0.25, start angle=180, end angle=0];
\draw (1.51795,1.2141)--(1.51795,-0.40192);
\draw (2.01795,1.2141)--(2.01795,-0.40192);
\draw (1.51795,-0.40192) arc [radius=0.25, start angle=180, end angle=360];
\draw (0,0) arc [radius=3, start angle=240, end angle=270];
\draw (2.01795,-0.40192) arc [radius=3, start angle=270, end angle=300];
\draw(2.01795,1.2141) arc[radius=0.25, start angle=270, end angle=360];
\draw  (2.26795,1.4641)--(2.26795,3.4641); 
\draw  (2.26795,1.4641) arc [radius=4, start angle=270, end angle=300];
\draw (4.26795,2)--(4.26795,4);
\draw (4.26795,2) arc[radius=2, start angle=30, end angle=0];
\draw(4.5359,1)--(6.5359,-1);
\draw (4.5359,1) arc [radius=1, start angle=90, end angle=180];
\draw (3.5359,0)--(5.5359,-2); 
\node at (2,-1.7) {$\zf$};
\node at (-2.5,-1.5) {$\lf_0$};
\node at (6,-1.5) {$\rf_0$};
\node at (-2,1.5) {$\lf$};
\node at (6,1.5) {$\rf$};
\node at (1.75,2.5) {$\ff$};
\node at (0.5,2.5) {$\fbf$};
\node at (3.5,2.5) {$\fbf$};
\node at(0.6,0.5) {$\fbf_0$};
\node at (3,0.5) {$\fbf_0$};
\node at (1.8,0.5) {$\ff_0$};
\end{tikzpicture}
\caption{Alternative picture of the $k,\phi$-double space $M^2_{k,\phi}$ from the point of view of Lemma~\ref{kfb.9b} }
\label{fig.6}\end{figure}

\begin{lemma}
The projections $\pr_{L}\times \Id_{[0,\infty)_k}: M^2\times [0,\infty)_k\to M\times [0,\infty)_k$ and 
$\pr_R\times \Id_{[0,\infty)_k}: M^2\times [0,\infty)_k\to M\times [0,\infty)_k$ lift to $b$-fibrations
$\pi_{k,\phi,L}: M^2_{k,\phi}\to M_t$ and $\pi_{k,\phi,R}: M^2_{k,\phi}\to M_t$, where we recall that 
$\pr_L:M^2\to M$ and $\pr_R: M^2\to M$ denote the projections on the left and on the right factors.
\label{kfb.9}\end{lemma}
\begin{proof}
By symmetry, it suffices to prove the result for $\pr_L\times \Id_{[0,\infty)_k}$.  First, by \cite{Mazzeo-MelrosePhi}, this projection lifts to a $b$-fibration
$$
        \pi_{\phi,L}\times \Id_{[0,\infty)}: M^2_{\phi}\times [0,\infty)_k\to M\times [0,\infty)_{k}.
$$
Applying \cite[Lemma~2.5]{hmm}, this lifts to a $b$-fibration
$$
    [M^2_{\phi}\times[0,\infty)_k;\fbf\times\{0\},\ff\times\{0\},\lf\times\{0\}].
$$
Finally, by \cite[Lemma~2.7]{hmm}, this further lifts to a $b$-fibration 
$$
     \pi_{k,\phi,L}: M^2_{\phi}\to M_t
$$
as desired.  
\end{proof}

Similarly, we know from \cite{Kottke} that $\pr_L\times [0,\infty)_k$ and $\pr_R\times [0,\infty)_k$ lift to $b$-fibrations
$$
          \pi_{t,L}: M^2_{t}\to M_t, \quad \pi_{t,R}:M^2_t\to M_t.
$$

Let $\Delta_{k,\phi}$ be the lift of $\Delta_M\times [0,\infty)_k\subset M^2\times [0,\infty)$ to $M^2_{k,\phi}$, where $\Delta_M$ is the diagonal in $M\times M$.  

\begin{lemma}
The lifts of $\cV_{k,\phi}(M)$ via the maps $\pi_{k,\phi,L}$ and $\pi_{k,\phi,R}$ are transversal to $\Delta_{k,\phi}$.  
\label{kfb.10}\end{lemma} 
\begin{proof}
By symmetry, it suffices to consider the lift by $\pi_{k,\phi,L}$.  Moreover, since it is a local statement near the lifted diagonal, it suffices to consider the blow-up of $\ff\times\{0\}$.  Now, by \cite[Lemma~5]{Mazzeo-MelrosePhi}, we know that the lift of $\cV_{k,\phi}(M\times[0,\infty))$ to $M^2_{\phi}\times [0,\infty)_k$ is transversal to the lifted diagonal.  In fact, if $y=(y^1,\ldots,y^h)$ denotes coordinates on the base of the fiber bundle $\phi: \pa M\to Y$ and $z=(z^1,\ldots,z^v)$ coordinates on the fibers,  then 
\begin{equation}
     S= \frac{\frac{x}{x'}-1}{x'}, Y=\frac{y-y'}{x'}, x', y', z,z',k
\label{kfb.11}\end{equation} 
are coordinates near the intersection of the lifted diagonal with $\ff\times [0,\infty)_k$ in $M^2_{\phi}\times [0,\infty)_k$.  In these coordinates, the lifted diagonal corresponds to $S=Y=0$, $z=z'$, while the lift from the left of $\cV_{k,\phi}(M\times [0,\infty)_k)$ is spanned by 
\begin{equation}
 (1+x'S)^2\frac{\pa}{\pa S}, \quad (1+x'S)\frac{\pa}{\pa Y^i},\quad \frac{\pa}{\pa z^j}. 
\label{kfb.12}\end{equation}
Now, blowing-up $\ff\times \{0\}$ corresponds to replace the coordinates \eqref{kfb.11} by 
\begin{equation}
  S,Y, y', z,z',  \quad r'= \sqrt{(x')^2+k^2}, \quad \theta= \arctan (\frac{x'}{k}).
\label{kfb.13}\end{equation}
In these new coordinates, the lift from the left of $\cV_{k,\phi}(M_t)$ is still spanned by \eqref{kfb.12}, that is,
by
$$
(1+r'(\sin\theta)S)^2\frac{\pa}{\pa S}, \quad (1+r'(\sin\theta)S)\frac{\pa}{\pa Y^i},\quad \frac{\pa}{\pa z^j}.
$$
Since the lifted diagonal still corresponds to $S=Y=0$, $z=z'$, transversality follows.
\end{proof}

Similarly, let $\Delta_{t}$ the lifted diagonal in $M^2_t$.  
\begin{lemma}
The lift of $\cV_t(M_t)$ via the maps $\pi_{t,L}$ and $\pi_{t,R}$ are transversal to the lifted diagonal $\Delta_t$.
\label{kfb.14}\end{lemma}
\begin{proof}
By symmetry, we only need to prove the result for $\pi_{t,L}$.  Moreover, since the statement is local near the lifted diagonal, the relevant blow-ups in \eqref{kfb.8} are those of $\fb\times\{0\}$ and $\Delta_b\cap\fb\times [0,\infty)_k$. Now, on $M^2_b\times [0,\infty)_k$, one can consider the coordinates $s=\frac{x}{x'}, x', y,y',k$ near $\fb\times[0,\infty)_k$, where $y=(y^1,\ldots,y^{n-1})$ represents coordinates on $\pa M$.  In these coordinates, the lift from the left of 
$\cV_{k,\sc}(M\times[0,\infty)_k)$ is spanned by $x's^2\frac{\pa}{\pa s}, x's\frac{\pa}{\pa y^i}$ and there is a lack of transversality at $x'=0$, that is, at $\fb\times[0,\infty)_k$.  Blowing up $\fb\times\{0\}$ corresponds to introducing the coordinates 
$$
      s, y, y', \quad r'= \sqrt{(x')^2+k^2}, \quad \theta=\arctan(\frac{x'}k). 
$$    
In these new coordinates, the lift of $\cV_{k,\sc}(M_t)$ is spanned by 
$$
      r'(\sin\theta)s^2\frac{\pa}{\pa s}, \quad r'(\sin\theta)\frac{\pa}{\pa y^i}.
$$
Since $r=\sqrt{x^2+k^2}$ is boundary defining function for $\tf$,  this means that the lift of $\cV_t(M_t)$ is spanned by 
$$
       \frac{(\sin\theta) s^2}{\sqrt{s^2\sin^2\theta+\cos^2\theta}}\frac{\pa}{\pa s},\quad \frac{\sin\theta}{\sqrt{s\sin^2\theta+ \cos^2\theta}}\frac{\pa}{\pa y^i}.      
$$ 
There is still a lack of transversality at $\sin\theta=0$, that is, at the lift of $\fb\times[0,\infty)_k$.  However, blowing up the lift of $\Delta_b\cap \fb\times [0,\infty)_k$ corresponds to introducing the coordinates
$$
      S=\frac{s-1}{\sin\theta}, \quad Y=\frac{y-y'}{\sin\theta}, r', \theta,
$$ 
in terms of which the lift from the left of $\cV_t(M_t)$ is locally spanned by 
$$
   \frac{ s^2}{\sqrt{s^2\sin^2\theta+\cos^2\theta}}\frac{\pa}{\pa s},\quad \frac{1}{\sqrt{s\sin^2\theta+ \cos^2\theta}}\frac{\pa}{\pa y^i}.   
$$ 
This is clearly transverse to the lifted diagonal given by $Y=0, S=0$ in those coordinates.  
\end{proof}

These transversality results allow us to give a simple description of the Schwartz kernels of differential $k,\phi$-operators and transition differential operators.  Starting with the former, consider the coordinates \eqref{kfb.13}.  In these coordinates, the Schwartz kernel of the identity operator takes the form
$$
  \kappa_{\Id}= \delta(-S)\delta(-Y)\delta(z'-z)\beta^*_{k,\phi}(\pr_R\times \Id_{[0,\infty)_k})^*\pr_1^*\nu_{\phi},
$$
where $\pr_1: M\times [0,\infty)_k\to M$ is the projection on the first factor and $\nu_{\phi}$ is some non-vanishing $\phi$-density.  Hence, 
$$
      \kappa_{\Id}\in \cD^0(\Delta_{k,\phi})\cdot \nu^R_{k,\phi}
$$
with 
\begin{equation}
\nu^R_{k,\phi}=\beta^*_{k,\phi}(\pr_R\times \Id_{[0,\infty)_k})^*\pr_1^*\nu_{\phi}
\label{rd.1}\end{equation}
 a lift from the right of a non-vanishing $k,\phi$-density and $\cD^0(\Delta_{k,\phi})$ is the space of smooth delta distributions on $\Delta_{k,\phi}$.  More generally, by the transversality of Lemma~\ref{kfb.10}, the Schwartz kernel of an operator $P\in \Diff^m_{k,\phi}(M_t)$ is of the form
$$
    \kappa_P= \pi^*_{k,\phi,L}P\cdot\kappa_{\Id}\in \cD^m(\Delta_{k,\phi})\cdot \nu^R_{k,\phi},
$$
where $\cD^m(\Delta_{k,\phi})=\Diff^m(M^2_{k,\phi})\cdot \cD^0(\Delta_{k,\phi})$ is the space of smooth delta distributions of order $m$ on $\Delta_{k,\phi}$.  In fact, the transversality of Lemma~\ref{kfb.10} ensures that there is a bijection between $\Diff^m_{k,\phi}(M_t)$ and $\cD^m(\Delta_{k,\phi})\cdot \nu^R_{k,\phi}$.  This suggests the following definition.  

\begin{definition}
Let $E$ and $F$ be vector bundles on transition single space $M_t$.  The small calculus of pseudodifferential $k,\phi$-operators acting from sections of $E$ to section of $F$ is the union over all $m\in \bbR$ of the spaces
\begin{equation}
\Psi^m_{k,\phi}(M;E,F):= \{  \kappa\in I^{m-\frac14}(M^2_{k,\phi},\Delta_{k,\phi};\Hom_{k,\phi}(E,F)\otimes {}^{k,\phi}\Omega_R(M)\; | \; \kappa\equiv 0 \; \mbox{at} \; \pa M^2_{k,\phi}\setminus\ff_{k,\phi}\}, 
\label{kfb.16}\end{equation}
where $\ff_{k,\phi}$ is the union of the boundary hypersurfaces of $M^2_{k,\phi}$ intersecting the lifted diagonal $\Delta_{k,\phi}$,  $\Hom_{k,\phi}(E,F)= \pi^*_{k,\phi,L}F\otimes \pi^*_{k,\phi,R}E^*$ and ${}^{k,\phi}\Omega_R(M):= \beta^*_{k,\phi}(\pr_R\times\Id_{[0,\infty)_k})^*\pr^*_1  {}^{\phi}\Omega(M)$.
 \label{kfb.15}\end{definition}
More generally, for $\cE$ an index family of the boundary hypersurfaces of $M^2_{k,\phi}$, we can consider the spaces
\begin{gather}
\label{kfb.16}
\Psi^{-\infty,\cE}_{k,\phi}(M;E,F):= \cA^{\cE}_{\phg}(M^2_{k,\phi};\Hom_{k,\phi}(E,F)\otimes {}^{k,\phi}\Omega_R(M)), \\
\label{kfb.16c} \Psi^{m,\cE}_{k,\phi}(M;E,F):= \Psi^{m}_{k,\phi}(M;E,F)+\Psi^{-\infty,\cE}_{k,\phi}(M;E,F), \quad m\in \bbR. 
\end{gather}

Recall from \cite{GH1,Kottke} that for $\phi=\Id$ on $\pa M$, the calculus of $b$-$\sc$ transition pseudodifferential operators admits a similar definition.  Let $\pi_{t,L}= (\pr_L\times \Id_{[0,\infty)_k})\circ\beta_t$ and $\pi_{t,R}= (\pr_R\times \Id_{[0,\infty)_k})\circ\beta_t$ be the analog of $\pi_{k,\phi,L}$ and $\pi_{k,\phi,R}$ and let $x_{\sc}$ be a boundary defining function for the boundary hypersurface $\sc$ in $M^2_t$.  Then the small calculus of $b$-$\sc$ transition pseudodifferential operators acting from sections of $E$ to sections of $F$ is defined as the union over $m\in \bbR$ of 
\begin{equation}
\Psi^m_t(M;E,F)=\{\kappa\in I^{m-\frac14}(M_t^2,\Delta_t;\Hom_t(E,F)\otimes {}^{t}\Omega_R(M)\; | \; \kappa\equiv 0 \; \mbox{at} \; \pa M^2_t\setminus \ff_t\},
\label{kfb.17}\end{equation}
where $\ff_t$ is the union of boundary hypersurfaces of $M^2_t$ intersecting the lifted diagonal, $\Hom_t(E,F)= \pi_{t,L}^*F\otimes \pi^*_{t,R}E^*$ and ${}^{t}\Omega_R(M)= x_{\sc}^{-n}\beta_t^*(\pr_R\times \Id_{[0,\infty)_k})^*\pr_1^*{}^b\Omega(M)$ with ${}^{b}\Omega(M)$ the bundle of $b$-densities on $M$ in the sense of \cite{MelroseAPS}.  If $\cE$ is an index family for the boundary hypersurfaces of $M^2_t$, we can consider more generally the spaces
\begin{equation}
\begin{gathered}
\Psi^{-\infty,\cE}_{t}(M;E,F):= \cA^{\cE}_{\phg}(M^2_t; \Hom_t(E,F)\otimes {}^{t}\Omega_R), \\
\Psi^{m,\cE}_{t}(M;E,F):= \Psi^m_t(M;E,F)+ \Psi^{-\infty,\cE}_t(M;E,F), \quad m\in \bbR.
\end{gathered}
\label{kfb.18}\end{equation}

\section{The triple space of low energy fibered boundary operators}\label{ts.0}
To obtain nice composition results for $k,\phi$-operators, we can follow the approach of Melrose and use a suitable triple space and apply the pushforward theorem of \cite[Theorem~5]{Melrose1992}.
To construct such a  $k,\phi$-triple space, we can start with the Cartesian product $M^3\times [0,\infty)_k$ and consider the projections $\pi_L,\pi_C,\pi_R: M^3\times[0,\infty)_k\to M^2\times [0,\infty)_k$ given by
$$
        \pi_L(m,m',m'',k)=(m,m',k), \quad \pi_C(m,m',m'',k)=(m,m'',k), \quad \pi_R(m,m',m'',k)=(m',m'',k).
$$
As in \cite{GH1}, let us use the $4$-digit binary codes for the faces of $M^3\times [0,\infty)_k$, where
$H_{0000}$ represent $(\pa M)^3\times \{0\}$, $H_{1010}$ stands for $M\times \pa M\times M\times \{0\}$, 
$H_{0011}$ stands for $(\pa M)^2\times M\times [0,\infty)_k$, and so on.  In this notation, recall from \cite{GH1, Kottke} that 
\begin{equation}
 M_{k,b}^3= [M^3\times [0,\infty)_k; H_{0000}, H_{1000}, H_{0100}, H_{0010}, H_{0001}, H_{1100}
, H_{1010}, H_{1001}, H_{0110}, H_{0101}, H_{001}]
\label{ts.1}\end{equation}
is obtained by blowing up all the corners of $M^3\times [0,\infty)_k$ in order of decreasing codimension.

\begin{lemma}
For $o\in\{L,C,R\}$, $\pi_o$ lift to a $b$-fibration
$$
    \pi^3_{b,o}: M^3_{k,b}\to M^2_{k,b}.
$$
\label{ts.2}\end{lemma}
\begin{proof}
By symmetry, it suffices to consider the case $o=L$.  By \cite[Lemma~2.5]{hmm}, the projection $\pi_L$ first lifts to a $b$-fibration
$$
[M^3\times [0,\infty)_k; H_{0000},H_{0010}]\to [M^2\times[0,\infty)_k; \pa M\times\pa M\times \{0\}].
$$
Applying \cite[Lemma~2.5]{hmm} three more times, this lifts to a $b$-fibration
\begin{equation}
[M^3\times [0,\infty)_k; H_{0000}, H_{0010}, H_{0100}, H_{0110}, H_{1000}, H_{1010}, H_{0001}, H_{0011}]\to M^2_{k,b}.
\label{ts.3}\end{equation}
Now, after the blow-up of $H_{0000}$, the lift of the corner $H_{0110}$ is disjoint from those of $H_{1000}$ and $H_{0001}$, while the lift of $H_{1010}$ is disjoint from the one of $H_{0001}$.  Hence, their blow-ups commute in \eqref{ts.3}, which can be rewritten
\begin{equation}
[M^3\times[0,\infty)_{k}; H_{0000}, H_{0010}, H_{0100}, H_{1000}, H_{0001}, H_{0110}, H_{1010}, H_{0011}]\to M^2_{k,b}.
\label{ts.4}\end{equation}
Hence, by \cite[Lemma~2.7]{hmm}, the $b$-fibration \eqref{ts.4} lifts to a $b$-fibration 
$$
  [M^3\times[0,\infty)_k; H_{0000}, H_{0010}, H_{0100}, H_{1000}, H_{0001}, H_{0110}, H_{1010}, H_{0011}, H_{1100}, H_{1001}, H_{0101}]\to M^2_{k,b}.
$$
The result then follows by the commutativity of the blow-ups of non-intersecting $p$-submanifolds.

\end{proof}

Let $H_{ij\ell m}^b$ be the boundary hypersurface of $M^3_{k,b}$ corresponding to the blow-up of $H_{ij\ell m}$.  Since $M^2_{k,\phi}$ is constructed from $M^2_{k,b}$ by blowing up the $p$-submanifolds $\Phi_+$ and $\Phi_0$ defined in Lemma~\ref{kfb.9b}, this suggests to look at the lifts of  $\Phi_+$ and $\Phi_0$   with respect to $\pi_{b,o}$ for $o\in\{L,C,R\}$ to construct the triple space of $M^2_{k,\phi}$.  For $\Phi_+$, this gives the $p$-submanifolds $G^+_o$ contained in $H_{0001}^b$ for each $o\in\{L,C,R\}$, as well a the p-submanifolds $J^+_L$ contained in $H^b_{0011}$, $J^+_C$ contained in $H^b_{0101}$ and $J^+_R$ contained in $H^b_{1001}$.  The $p$-submanifold $G^+_L$, $G^+_C$ and $G^+_R$ have a non-zero intersection.  To describe it, notice that there is a natural diffeomorphism
$$
   H^b_{0001}\cong \pa M^3\times L_b
$$    
where $L_b$ is the face corresponding to $H^b_{0001}$ inside $[0,1)^3_{k,b}$.  We have further that $L_b\cong G_b\times [0,\infty)_k$, where $G_b$ is the front face of the $b$-triple space $[0,1)^3_b$.  

\begin{lemma}
The intersection of any pair of $G^+_L$, $G^+_C$ and $G_{R}^+$ in $M^3_{k,b}$ is the $p$-submanifold 
$$
      K^+=G_L^+\cap G_C^+\cap G_R^+.
$$
\label{ts.5}\end{lemma}
\begin{proof}
Let $x,x'$ and $x''$ denote the boundary defining functions for each factor of $M^3$.  let $p_b\in G_b$ be the unique point of $G_b$ contained in the lifted diagonal on $[0,1)^3_b$.  Then under the identification $H^b_{0001}\cong \pa M^3\times G_b\times [0,\infty)_k$, we have that
$$
 G_L^+\cong \{ (m,m',m'',q,k)\in\pa M^3\times G_b\times [0,\infty)_k \; |\; \phi(m)=\phi(m'), \; x(q)= x'(q) \}
$$
and there are similar descriptions for $G^+_C$ and $G^+_R$.  From those descriptions, we see that the intersection of any pair in $G^+_L$, $G_C^+$ and $G^+_R$ is given by 
$$
K^+=\{ (m,m',m'',q,k)\in \pa M^3\times G_b\times [0,\infty)_k \; | \; \phi(m)=\phi(m')=\phi(m''), q=p_b\},
$$  
which is clearly a $p$-submanifold of $H^b_{0001}$.  

\end{proof}

Similarly, the lifts of $\Phi_0\in M^2_{k,b}$ with respect to $\pi^3_{b,L}$, $\pi^3_{b,C}$ and $\pi^3_{b,R}$ gives $p$-submanifolds 
$G_L$, $G_C$ and $G_R$ inside $H^b_{0000}$ as well as the $p$-submanifolds $J_L$, $J_C$ and $J_R$ inside 
$H^b_{0010}$, $H^b_{0100}$ and $H^b_{1000}$.  Again $G_L$, $G_C$ and $G_R$ have a non-trivial intersection.

\begin{lemma}
The intersection for any pair of $G_L$, $G_C$ and $G_R$ is the $p$-submanifold 
$$
    K=G_L\cap G_C\cap G_R.
$$
\label{ts.6}\end{lemma} 
\begin{proof}
For the boundary hypersurface $H^b_{0000}$, there is a natural diffeomorphism
\begin{equation}
   H^b_{0000}\cong \pa M^3\times D_b,
\label{ts.7}\end{equation}
where $D_b$ is the corresponding face $H^b_{0000}$ in $[0,1)^3_{k,b}$.  Let $E_b\subset D_b$ be the $p$-submanifold given by the intersection of $D_b$ with the lift of the diagonal 
$$
   \{(x,x,x,k)\in M^3\times [0,\infty)_k\; | \; x\in [0,1), k\in [0,\infty)_k \}\subset [0,1)^3\times [0,\infty)_k 
$$ 
to $[0,1)^3_{k,b}$.  Then, under the identification \eqref{ts.7}, the intersection of any pair of $G_L$, $G_C$ and $G_R$ is given by the $p$-submanifold 
$$
  K\cong \Delta^3_{\phi}\times E_b\subset \pa M^3\times D_b \cong H^b_{0000},
$$
where 
$$
     \Delta^3_{\phi}=\{(m,m',m'')\in \pa M^3\; | \; \phi(m)= \phi(m')= \phi(m'')\}
$$
is the triple fibered diagonal in $\pa M^3$.  
\end{proof}

This suggests to define the $k,\phi$-triple space by
\begin{equation}
  M^3_{k,\phi}=[M^3_{k,b}; K^+, G^+_L, G^+_C, G^+_R, J^+_L, J^+_C, J^+_R, K, G_L, G_C, G_R, J_L, J_C, J_R]
\label{ts.8}\end{equation}
with blow-down map
\begin{equation}
\beta^3_{k,\phi}: M^3_{k,\phi}\to M^3\times [0,\infty).
\label{ts.8b}\end{equation}
\begin{proposition}
For each $o\in \{L,C,R\}$, the $b$-fibration $\pi^3_{b,o}: M^3_{k,b}\to M^2_{k,b}$ lifts to a $b$-fibration 
$$
   \pi^3_{k,\phi,o}: M^3_{k,\phi}\to M^2_{k,\phi}.
$$
\label{ts.9}\end{proposition}
\begin{proof}
By symmetry, it suffices to check the result for $o=L$.  We can then essentially proceed as in the proof of \cite[Proposition~6]{Mazzeo-MelrosePhi}.  First, by \cite[Lemma~2.5]{hmm}, the map $\pi^3_{b,L}$ lifts to a $b$-fibration
$$
  [M^3_{k,b};G^+_L,J^+_L]\to [M^2_{k,b},\Phi_+].
$$
By \cite[Lemma~2.7]{hmm}, this further lift to a $b$-fibration 
$$
 [M^3_{k,b}; G_L^+, J_L^+, K^+, G_C^+, G_R^+, J_C^+, J_R^+]\to [M^2_{k,b}; \Phi_+].
$$
Using the commutativity of nested blow-ups and of blow-ups of non-intersecting $p$-submanifolds, this corresponds to a $b$-fibration
$$
 [M^3_{k,b};K^+,G_L^+, G_C^+, G_R^+, J_L^+, J_C^+, J_R^+]\to [M^2_{k,b},\Phi_+].
$$
Repeating this argument, but with $\Phi_+$, $K^+$, $G^+_o$ and $J^+_o$ replaced by $\Phi_0$, $K$, $G_o$ and $J_o$, we can check that this lifts further to a $b$-fibration
$$
   \pi^3_{k,\phi,L}: M^3_{k,\phi}\to [M^2_{k,b};\Phi_+,\Phi_0]= M^2_{k,\phi}
$$
as claimed.
\end{proof}

As in \cite{Mazzeo-MelrosePhi}, the $b$-fibrations $\pi^3_{k,\phi,o}$ for $o\in\{L,C,R\}$ behave well with respect to the lifted diagonals.  More precisely, for $o\in\{L,C,R\}$, set $\Delta^3_{k,\phi,o}:= \pi^{-1}_{k,\phi,o}(\Delta_{M,k})$ where $$
   \Delta_{M,k}=\{ (m,m,k) \in M^2\times [0,\infty) \; | \; m\in M, \; k\in [0,\infty)\}
$$  
is the diagonal.  These are clearly $p$-submanifolds.  Moreover, for $o\ne o'$, the intersection $\Delta^3_{k,\phi,o}\cap \Delta^3_{k,\phi,o'}$ is the $p$-submanifold $\Delta^3_{k,\phi,T}$ in $M^3_{k,\phi}$ given by the lift of the triple diagonal
$$
       \Delta^3_{M,k}= \{(m,m,m,k)\in M^3\times [0,\infty) \; | \; m\in M, \; k\in [0,\infty)\}.
$$

\begin{lemma}
For $o\ne o'$, the $b$-fibration $\pi^3_{k,\phi,o}$ is transversal to $\Delta^3_{k,\phi,o'}$ and induces a diffeomorphism 
$\Delta^3_{k,\phi,o'}\cong M^2_{k,\phi}$ sending $\Delta^3_{k,\phi,T}$ onto $\Delta_{k,\phi}\subset M^2_{k,\phi}$.
\label{ts.10}\end{lemma}
\begin{proof}
By symmetry, we can assume $o=L$ and $o'=C$.  Now, one can check that the corresponding statement for $M^3_{k,b}$ holds.  Doing the blow-ups in the order used to show that $\pi^3_{k,\phi,L}$ is a $b$-fibration, we can check step by step that transversality is preserved.  The diffeomorphism is then a direct consequence of the transversality statement.  
\end{proof}

\section{Composition of low energy fibered boundary operators}\label{com.0}

We can use the triple space of the previous section to describe the composition of $k,\phi$-pseudodifferential operators.  Let us denote by $H^{k,\phi}_{ijlm}$ the boundary hypersurface of $M^{3}_{k,\phi}$ corresponding to the lift of $H_{ijlm}$ in $M^3\times [0,\infty)$.  Let us denote by $\ff^+_T, \ff^+_{LT}, \ff^+_{CT}, \ff^+_{RT}, \ff^+_{L}, \ff^+_C$ and $\ff^+_R$ the boundary hypersurfaces corresponding to the blow-ups of $K^+,G_L^+, G_C^+, G_R^+, J_L^+, J_C^+$ and $J^+_R$ respectively, while let $\ff^0_T, \ff^0_{LT}, \ff^0_{CT}, \ff^0_{RT}, \ff^0_{L}, \ff^0_C$ and $\ff^0_R$ denote the boundary hypersurfaces of $M^3_{k,\phi}$ corresponding to the lifts of $K,G_L, G_C, G_R, J_L, J_C$ and $J_R$.  Using this notation, let us describe how the boundary hypersurfaces behave with respect to the three $b$-fibrations of Proposition~\ref{ts.9}.  For the $b$-fibration $\pi^3_{k,\phi,L}$, it sends $H^{k,\phi}_{1101}$ surjectively  onto $M^2_{k,\phi}$, and otherwise is such that 
\begin{equation}
\begin{aligned}
(\pi^3_{k,\phi,L})^{-1}(\zf)= H^{k,\phi}_{1110}\cup H^{k,\phi}_{1100}, \quad \quad &(\pi^3_{k,\phi,L})^{-1}(\ff_0)= \ff^0_T\cup \ff^0_{LT} \cup \ff^0_L, \\
(\pi^3_{k,\phi,L})^{-1}(\lf_0)= H^{k,\phi}_{0110}\cup H^{k,\phi}_{0100}\cup \ff^0_C, \quad \quad & (\pi^3_{k,\phi,L})^{-1}(\fbf_0)=H^{k,\phi}_{0000}\cup \ff^0_{CT}\cup \ff^0_{RT}\cup H^{k,\phi}_{0010}, \\
(\pi^3_{k,\phi,L})^{-1}(\rf_0)= H^{k,\phi}_{1010}\cup H^{k,\phi}_{1000}\cup \ff^0_R, \quad \quad & (\pi^3_{k,\phi,L})^{-1}(\ff)= \ff^+_{T}\cup \ff^+_{LT}\cup \ff^+_L, \\
(\pi^3_{k,\phi,L})^{-1}(\lf)= H^{k,\phi}_{0111}\cup H^{k,\phi}_{0101}\cup \ff^+_C, \quad \quad & (\pi^3_{k,\phi,L})^{-1}(\fbf)= H^{k,\phi}_{0001}\cup \ff^+_{CT}\cup \ff^+_{RT}\cup H^{k,\phi}_{0011}, \\
(\pi^3_{k,\phi,L})^{-1}(\rf)= H^{k,\phi}_{1011}\cup H^{k,\phi}_{1001}\cup \ff^+_{R}. \quad \quad &
\end{aligned}
\label{com.1}\end{equation}
For the $b$-fibration $\pi^3_{k,\phi,C}$, it sends surjectively $H^{k,\phi}_{1011}$ onto $M^2_{k,\phi}$, and otherwise  is such that
\begin{equation}
\begin{aligned}
(\pi^3_{k,\phi,C})^{-1}(\zf)= H^{k,\phi}_{1110}\cup H^{k,\phi}_{1010}, \quad \quad &(\pi^3_{k,\phi,C})^{-1}(\ff_0)= \ff^0_T\cup \ff^0_{CT} \cup \ff^0_C, \\
(\pi^3_{k,\phi,C})^{-1}(\lf_0)= H^{k,\phi}_{0110}\cup H^{k,\phi}_{0010}\cup \ff^0_L, \quad \quad & (\pi^3_{k,\phi,C})^{-1}(\fbf_0)=H^{k,\phi}_{0000}\cup \ff^0_{LT}\cup \ff^0_{RT}\cup H^{k,\phi}_{0100}, \\
(\pi^3_{k,\phi,C})^{-1}(\rf_0)= H^{k,\phi}_{1100}\cup H^{k,\phi}_{1000}\cup \ff^0_R, \quad \quad & (\pi^3_{k,\phi,C})^{-1}(\ff)= \ff^+_{T}\cup \ff^+_{CT}\cup \ff^+_C, \\
(\pi^3_{k,\phi,C})^{-1}(\lf)= H^{k,\phi}_{0111}\cup H^{k,\phi}_{0011}\cup \ff^+_L, \quad \quad & (\pi^3_{k,\phi,C})^{-1}(\fbf)= H^{k,\phi}_{0001}\cup \ff^+_{LT}\cup \ff^+_{RT}\cup H^{k,\phi}_{0101}, \\
(\pi^3_{k,\phi,C})^{-1}(\rf)= H^{k,\phi}_{1101}\cup H^{k,\phi}_{1001}\cup \ff^+_{R}. \quad \quad &
\end{aligned}
\label{com.2}\end{equation}
Finally, the $b$-fibration $\pi^3_{k,\phi,R}$ sends $H^{k,\phi}_{0111}$ surjectively onto $M^2_{k,\phi}$, and otherwise is such that
\begin{equation}
\begin{aligned}
(\pi^3_{k,\phi,R})^{-1}(\zf)= H^{k,\phi}_{1110}\cup H^{k,\phi}_{0110}, \quad \quad &(\pi^3_{k,\phi,R})^{-1}(\ff_0)= \ff^0_T\cup \ff^0_{RT} \cup \ff^0_R, \\
(\pi^3_{k,\phi,R})^{-1}(\lf_0)= H^{k,\phi}_{1010}\cup H^{k,\phi}_{0010}\cup \ff^0_L, \quad \quad & (\pi^3_{k,\phi,R})^{-1}(\fbf_0)=H^{k,\phi}_{0000}\cup \ff^0_{LT}\cup \ff^0_{CT}\cup H^{k,\phi}_{1000}, \\
(\pi^3_{k,\phi,R})^{-1}(\rf_0)= H^{k,\phi}_{1100}\cup H^{k,\phi}_{0100}\cup \ff^0_C, \quad \quad & (\pi^3_{k,\phi,R})^{-1}(\ff)= \ff^+_{T}\cup \ff^+_{RT}\cup \ff^+_R, \\
(\pi^3_{k,\phi,R})^{-1}(\lf)= H^{k,\phi}_{1011}\cup H^{k,\phi}_{0011}\cup \ff^+_L, \quad \quad & (\pi^3_{k,\phi,R})^{-1}(\fbf)= H^{k,\phi}_{0001}\cup \ff^+_{LT}\cup \ff^+_{CT}\cup H^{k,\phi}_{1001}, \\
(\pi^3_{k,\phi,R})^{-1}(\rf)= H^{k,\phi}_{1101}\cup H^{k,\phi}_{0101}\cup \ff^+_{C}. \quad \quad &
\end{aligned}
\label{com.3}\end{equation}

To see what happens to the lift of densities, the following lemma due to Melrose will be useful.
\begin{lemma}[Melrose]
Let $Y$ be a $p$-submanifold of a manifold with corners.  Let $w$ be the codimension of $Y$ within the smallest boundary face of $X$ containing $Y$.  Let $\beta$ be the blow-down map from $[X;Y]$ to $X$.  If $\rho_Y\in \CI([X;Y])$ is a boundary defining function for the new boundary hypersurface created by the blow-up of $Y$, then
$$
        \beta^*{}^{b}\Omega(X)= (\rho_Y^w){}^{b}\Omega([X;Y]).
$$ 
\label{lbd.1}\end{lemma}
Indeed, using this lemma, we see that
\begin{multline}
(\beta^3_{k,\phi})^*({}^{b}\Omega(M^3\times[0,\infty))=  \\
\left(\rho^2_{\ff^0_T}\rho_{\ff^+_T}^2\rho_{\ff^0_{LT}} \rho_{\ff^0_{CT}}\rho_{\ff^0_{RT}} \rho_{\ff^+_{LT}}\rho_{\ff^+_{CT}}\rho_{\ff^+_{RT}}\rho_{\ff^0_{L}}\rho_{\ff^0_{C}}\rho_{\ff^0_R}\rho_{\ff^+_L}\rho_{\ff^+_C}\rho_{\ff^+_R}\right)^{h+1}  ({}^{b}\Omega(M^3_{k,\phi})),
\label{com.4}\end{multline}
where $\rho_H$ denotes a boundary defining function for the boundary hypersurface $H$.  
We also compute that
\begin{equation}
     \pi^*_{k,\phi,R}(x)= \rho_{\rf_0}\rho_{\rf}\rho_{\fbf_0}\rho_{\fbf}\rho_{\ff_0}\rho_{\ff},
\label{com.5}\end{equation}
so that combining with Lemma~\ref{lbd.1}, we see that
\begin{equation}
   (\beta_{k,\phi})^*\left[(\pr_L\times\Id_{[0,\infty)})^*({}^{b}\Omega(M\times [0,\infty))) \cdot(\pr_R\times\Id_{[0,\infty)})^*\pr_1^*{}^{\phi}\Omega(M) \right]= (\rho_{\rf_0}\rho_{\rf}\rho_{\fbf_0}\rho_{\fbf})^{-h-1}({}^{b}\Omega(M^2_{k,\phi})).
\label{com.6}\end{equation}
Pulling back \eqref{com.5} to $M^3_{k,\phi}$ via $\pi^3_{k,\phi,L}$ and $\pi^3_{k,\phi,R}$ gives
\begin{multline}
(\pi^3_{k,\phi,L})^*(\pi_{k,\phi,R}^*(x))= \rho_{1010}\rho_{1000}\rho_{\ff^0_{R}}\rho_{1011}\rho_{1001}\rho_{\ff^+_{R}}\rho_{0000}\rho_{\ff^0_{CT}}\rho_{\ff^0_{RT}}\rho_{0010} \\
\cdot \rho_{\ff^0_T}\rho_{\ff^0_{LT}}\rho_{\ff^0_L}\rho_{\ff^+_T}\rho_{\ff^+_{LT}}\rho_{\ff^+_L}\rho_{0001}\rho_{\ff^+_{CT}}\rho_{\ff^+_{RT}}\rho_{0011},
\label{com.7}\end{multline}
and 
\begin{multline}
(\pi^3_{k,\phi,R})^*(\pi_{k,\phi,R}^*x)= \rho_{1100}\rho_{0100}\rho_{\ff^0_C}\rho_{1101}\rho_{0101}\rho_{\ff^+_C}\rho_{0000}\rho_{\ff^0_{LT}}\rho_{\ff^0_{CT}}\rho_{1000}\\
\cdot \rho_{\ff^0_{T}}\rho_{\ff_{RT}}\rho_{\ff^0_R}\rho_{\ff^+_{T}}\rho_{\ff^+_{RT}}\rho_{\ff^+_R}\rho_{0001}\rho_{\ff^+_{LT}}\rho_{\ff^+_{CT}}\rho_{1001},
\label{com.8}\end{multline}
where $\rho_{ijlm}$ stands for $\rho_{H^{k,\phi}_{ijlm}}$.   Hence, in terms of the $\phi$-density bundle ${}^{\phi}\Omega(M)= (x^{-h-1}){}^{b}\Omega(M)$ and 
$${}^{b}\Omega^3_L(M^3_{k,\phi}):=(\pi^3_{k,\phi,L})^*(\beta_{k,\phi})^*\left[(\pr_L\times\Id_{[0,\infty)})^*({}^{b}\Omega(M\times [0,\infty)))\right],
$$  we see that
\begin{equation}
{}^{b}\Omega^3_L(M^3_{k,\phi}) \cdot (\pi^3_{k,\phi,L})^*\beta_{k,\phi}^*\left[(\pr_R\times\Id_{[0,\infty)})^*\pr_1^*{}^{\phi}\Omega(M))\right]  \cdot (\pi^3_{k,\phi,R})^*\beta_{k,\phi}^*\left[(\pr_R\times\Id_{[0,\infty)})^*\pr_1^*{}^{\phi}\Omega(M))\right] 
\label{com.9}\end{equation}
corresponds to $(\rho^{\mathfrak{a}}) {}^{b}\Omega(M^3_{k,\phi})$ with multiweight $\mathfrak{a}$ such that 
\begin{multline}
\rho^{\mathfrak{a}}= \left( \rho_{\ff^0_{LT}}\rho_{\ff^0_{CT}}\rho_{\ff^0_{RT}}\rho_{\ff^+_{LT}}\rho_{\ff^+_{CT}}\rho_{\ff^+_{RT}}\rho_{\ff^0_{R}}\rho_{\ff^+_R}\rho^2_{0000}\rho^2_{1000}\rho^2_{0001}\rho^2_{1001}  \right. \\
\left. \rho_{1010}\rho_{1011}\rho_{0010}\rho_{0011}\rho_{1100}\rho_{0100}\rho_{1101}\rho_{0101} \right)^{-h-1}.
\label{com.10}\end{multline}

Hence, if $\kappa_A$ and $\kappa_B$ denote the Schwartz kernels of operators $A\in \Psi^{-\infty,\cE}_{k,\phi}(M)$ and 
$B\in\Psi^{-\infty,\cF}_{k,\phi}(M)$ and if ${}^{b}\nu^3_L$ is a nonvanishing section of ${}^{b}\Omega^3_L(M^3_{k,\phi})$, the above discussion and a careful computation shows that 
\begin{equation}
     {}^{b}\nu^3_L  \cdot (\pi^3_{k,\phi,L})^*\kappa_A \cdot (\pi^3_{k,\phi,R})^*\kappa_B \in \cA_{\phg}^{\cG}(M^3_{k,\phi};{}^{b}\Omega(M^3_{k,\phi}))
\label{com.10b}\end{equation}
with index family $\cG$ given by
\begin{equation}
\begin{array}{ll}
 \cG|_{H^{k,\phi}_{0000}}= \cE|_{\fbf_0}+ \cF|_{\fbf_0}-2(h+1), \quad &  \cG|_{\ff^+_T}= \cE|_{\ff}+ \cF|_{\ff}, \\
 \cG|_{H^{k,\phi}_{1000}}= \cE|_{\rf_0}+ \cF|_{\fbf_0}-2(h+1), \quad & \cG|_{\ff^+_{LT}}= \cE_{\ff}+ \cF|_{\fbf}-(h+1), \\
 \cG|_{H^{k,\phi}_{0100}}= \cE|_{\lf_0}+ \cF_{\rf_0}-(h+1) , \quad & \cG|_{\ff^+_{CT}}= \cE|_{\fbf}+\cF|_{\fbf}-(h+1),  \\
  \cG|_{H^{k,\phi}_{0010}}= \cE|_{\fbf_0}+ \cF|_{\lf_0}-(h+1), \quad & \cG|_{\ff^+_{RT}}= \cE|_{\fbf}+ \cF|_{\ff}-(h+1), \\
    \cG|_{H^{k,\phi}_{0001}}= \cE|_{\fbf}+\cF_{\fbf}-2(h+1), \quad & \cG|_{\ff^+_L}= \cE|_{\ff}+ \cF|_{\lf}, \\
    \cG|_{H^{k,\phi}_{1100}}= \cE|_{\zf}+ \cF|_{\rf_0}-(h+1), \quad & \cG|_{\ff^+_C}= \cE|_{\lf}+ \cF_{\rf}, \\
    \cG|_{H^{k,\phi}_{1010}}= \cE|_{\rf_0}+ \cF_{\lf_0}-(h+1), \quad & \cG|_{\ff^+_R}= \cE|_{\rf}+ \cF|_{\ff}-(h+1), \\
\cG|_{H^{k,\phi}_{1001}}= \cE|_{\rf}+ \cF|_{\fbf}-2(h+1), \quad & \cG|_{\ff^0_T}= \cE|_{\ff_0}+\cF|_{\ff_0}, \\
\cG|_{H^{k,\phi}_{0110}}= \cE|_{\lf_0}+ \cF|_{\zf}, \quad & \cG|_{\ff^0_{LT}}=\cE|_{\ff_0}+ \cF|_{\fbf_0}-(h+1), \\
\cG|_{H^{k,\phi}_{0101}}=\cE|_{\lf}+ \cF|_{\rf}-(h+1), \quad & \cG|_{\ff^0_{CT}}= \cE|_{\fbf_0}+ \cF|_{\fbf_0}-(h+1), \\
\cG|_{H^{k,\phi}_{0011}}=\cE|_{\fbf}+ \cF|_{\lf}-(h+1), \quad & \cG|_{\ff^0_{RT}}= \cE|_{\fbf_0}+ \cF|_{\ff_0}- (h+1), \\
\cG|_{H^{k,\phi}_{1110}}= \cE|_{\zf}+ \cF|_{\zf}, \quad & \cG|_{\ff^0_L}= \cE|_{\ff_0}+ \cF|_{\lf_0}, \\
\cG|_{H^{k,\phi}_{1101}}= \cF|_{\rf}-(h+1), \quad & \cG|_{\ff^0_C}= \cE_{\lf_0}+ \cF|_{\rf_0}, \\
\cG|_{H^{k,\phi}_{1011}}= \cE|_{\rf}+\cF|_{\lf}-(h+1), \quad & \cG|_{\ff^0_R}= \cE|_{\rf_0}+ \cF|_{\ff_0}- (h+1), \\
\cG|_{H^{k,\phi}_{0111}}= \cE|_{\lf}. &
\end{array} 
\label{com.11}\end{equation}
This yields the following composition result.
\begin{theorem}
Let $E$, $F$ and $G$ be vector bundles over the transition single space $M_t$.  Suppose that $\cE$ and $\cF$ are index families associated to $M^2_{k,\phi}$ such that 
$$
    \inf\Re\cE_{\rf}+ \inf\cF_{\lf}>h+1.
$$
Then given $A\in \Psi^{m,\cE}_{k,\phi}(M;F,G)$ and $B\in \Psi^{m',\cF}_{k,\phi}(M;E,F)$, their composition is well-defined with 
$$
     A\circ B \in \Psi^{m+m',\cK}_{k,\phi}(M;E,G),
$$
where $\cK$ is the index family such that 
\begin{equation}
\begin{array}{ll}
\cK|_{\zf}&= (\cE|_{\zf}+\cF|_{\zf})\overline{\cup} (\cE|_{\rf_0}+\cF|_{\lf_0}-h-1), \\
\cK|_{\lf_0}&= (\cE|_{\lf_0}+\cF|_{\zf})\overline{\cup} (\cE|_{\fbf_0}+\cF|_{\lf_0}-h-1))\overline{\cup}(\cE|_{\ff_0}+ \cF|_{\lf_0}), \\
\cK|_{\rf_0}&= (\cE|_{\zf}+\cF|_{\rf_0})\overline{\cup}(\cE|_{\rf_0}+\cF|_{\fbf_0}-h-1)\overline{\cup}(\cE|_{\rf_0}+\cF|_{\ff_0}), \\
\cK|_{\lf}&= \cE|_{\lf}\overline{\cup} (\cE|_{\fbf}+\cF|_{\lf}-h-1)\overline{\cup} (\cE|_{\ff}+ \cF|_{\lf}), \\
\cK|_{\rf}&= (\cF|_{\rf})\overline{\cup} (\cE|_{\rf}+\cF|_{\fbf}-h-1)\overline{\cup} (\cE|_{\rf}+\cF|_{\ff}), \\
\cK|_{\ff_0}&=(\cE|_{\ff_0}+\cF|_{\ff_0})\overline{\cup} (\cE|_{\fbf_0}+\cF|_{\fbf_0}-h-1)\overline{\cup} (\cE|_{\lf_0}+ \cF|_{\rf_0}), \\
\cK|_{\fbf_0}&= (\cE|_{\fbf_0}+ \cF|_{\fbf_0}-h-1)\overline{\cup}(\cE|_{\ff_0}+ \cF|_{\fbf_0})\overline{\cup}(\cE|_{\fbf_0}+ \cF|_{\ff_0})\overline{\cup}(\cE|_{\lf_0}+\cF|_{\rf_0}), \\
\cK|_{\ff}&= (\cE|_{\ff}+\cF|_{\ff})\overline{\cup}(\cE|_{\fbf}+\cF|_{\fbf}-h-1)\overline{\cup}(\cE|_{\lf}+\cF|_{\rf}), \\
\cK|_{\fbf}&= (\cE|_{\fbf}+ \cF|_{\fbf}-h-1)\overline{\cup}(\cE|_{\ff}+\cF|_{\fbf})\overline{\cup}(\cE|_{\fbf}+\cF|_{\ff})\overline{\cup}(\cE|_{\lf}+ \cF|_{\rf}).
\end{array}
\label{com.12b}\end{equation}
\label{com.12}\end{theorem}
\begin{proof}
For operators of order $-\infty$, it suffices to apply the pushforward theorem of \cite[Theorem~5]{Melrose1992} using \eqref{com.11}.  When the operators are of order $m$ and $m'$, we need to combine the pushforward theorem with Lemma~\ref{ts.10} to see that the composed operator is of the given order, \cf \cite[Proposition~B7.20]{EMM}.
\end{proof}
\begin{remark}
For $k>0$, that is, for the boundary hypersurfaces $\lf$, $\rf$, $\fbf$ and $\ff$, we recover as expected  from \eqref{com.12b} the composition result of \eqref{phi.18b} for $\phi$-operators.
\end{remark}

\begin{corollary}
If $E$, $F$ and $G$ are vector bundles over $M_t$, then 
$$
      \Psi^m_{k,\phi}(M;F,G)\circ \Psi^{m'}_{k,\phi}(M;E,F)\subset \Psi^{m+m'}_{k,\phi}(M;E,G).
$$ 
\label{com.13}\end{corollary}
\begin{proof}
It suffices to apply Theorem~\ref{com.12} with index families $\cE$ and $\cF$ given by the empty set except at $\ff$, $\ff_0$ and $\zf$, where it is given by $\bbN_0$.  
\end{proof}

Similarly, the triple space of \cite{GH1,Kottke} gives the following composition result for the $b$-$\sc$ transition calculus.
\begin{theorem}[\cite{GH1,Kottke}]
Let $A\in\Psi^{m,\cE}_t(M;E,G)$ and $B\in\Psi^{m',\cF}(M;E,F)$ be $b$-$\sc$ transition pseudodifferential operators with index families $\cE$ and $\cF$ given by the empty set at $\fb$, $\lf$ and $\rf$ and such that 
$$
   \inf\Re\cE|_{\sc}\ge 0, \quad \inf\Re\cF_{\sc}\ge 0.  
$$
In this case, $A\circ B\in \Psi^{m+m',\cG}_t(M;E,G)$ with index family $\cG$ given by
\begin{equation}
\begin{array}{ll}
\cG|_{\sc}= \cE|_{\sc}+ \cF|_{\sc}, & \cG|_{\zf}= (\cE_{\zf}+ \cF_{\zf})\overline{\cup}(\cE|_{\rf_0}+\cF|_{\lf_0}), \\
\cG|_{\fb_0}= (\cE|_{\lf_0}+\cF|_{\rf_0})\overline{\cup}(\cE|_{\fb_0}+\cF|_{\fb_0}), & \cG|_{\lf_0}=(\cE|_{\lf_0}+\cF|_{\zf})\overline{\cup}(\cE|_{\fb_0}+\cF|_{\lf_0}), \\
\cG|_{\rf_0}=(\cE|_{\zf}+ \cF|_{\rf_0})\overline{\cup}(\cE|_{\rf_0}+ \cF|_{\fb_0}), & \cG|_{\fb}=\cG|_{\lf}=\cG|_{\rf}=\emptyset.
\end{array}
\label{com.14b}\end{equation}
\label{com.14}\end{theorem}

\section{Symbol maps} \label{sm.0}

To define the principal symbol of an operator $A\in \Psi^{m,\cE}_{k,\phi}(M;E,F),$ it suffices to notice that its Schwartz kernel $\kappa_A$ has conormal singularities at the lifted diagonal $\Delta_{k,\phi}$, so has a principal symbol 
$$
\sigma_m(\kappa_A)\in S^{[m]}(N^*\Delta_{k,\phi};\End(E,F)).
$$  
We define the principal symbol of $A$, denoted ${}^{k,\phi}\sigma_m(A)$, to be $\sigma_m(\kappa_A)$.  By Lemma~\ref{kfb.10}, there is a natural identification $N^*\Delta_{k,\phi}\cong {}^{k,\phi}T^*M_t$, so that ${}^{k,\phi}\sigma_m(A)$ can be seen as an element of 
$S^{[m]}({}^{k,\phi}T^*M_t; \End(E,F))$. As for other pseudodifferential calculi, the principal symbol induces a short exact sequence
\begin{equation}
\xymatrix{
0 \ar[r] & \Psi^{m-1,\cE}_{k,\phi}(M;E,F)\ar[r] & \Psi^{m,\cE}_{k,\phi}(M;E,F) \ar[r]^-{{}^{k,\phi}\sigma_m} &
   S^{[m]}({}^{k,\phi}T^*M_t;\End(E,F)) \ar[r] & 0.
}
\label{sm.1}\end{equation} 
For the construction of good parametrices, we will however need other symbols capturing the asymptotic behavior of $k,\phi$-operators.  More precisely, for the boundary hypersurfaces $\zf$, $\ff_0$ and $\ff$ of $M^2_{k,\phi}$, we can define the normal operators of $A\in\Psi^{m,\cE}_{k,\phi}(M;E,F)$, for $\cE$ an index family such that $\inf\Re\cE|_{\zf}\ge 0$, $\inf\Re\cE|_{\ff_0}\ge 0$ and $\inf\Re \cE|_{\ff}\ge 0$,  by restriction of the Schwartz kernel $\kappa_A$ of $A$ to $\zf$, $\ff_0$ and $\ff$, 
\begin{equation}
 N_{\zf}(A)=\kappa_A|_{\zf}, \quad N_{\ff_0}(A)= \kappa_A|_{\ff_0}, \quad N_{\ff}(A)= \kappa_A|_{\ff}.
 \label{sm.2}\end{equation} 

Since the boundary hypersurface $\zf$ in $M^2_{\phi,k}$ is naturally identified with the $\phi$-double space $M^2_{\phi}$ of Mazzeo-Melrose \cite{Mazzeo-MelrosePhi}, the normal operator $N_{\zf}(A)$ can be seen as a $\phi$-operator.  In particular, in terms of the small calculus, there is a short exact sequence
\begin{equation}
\xymatrix{
0 \ar[r] & x_{\zf}\Psi^m_{k,\phi}(M;E,F) \ar[r] & \Psi^m_{k,\phi}(M;E,F) \ar[r]^-{N_{\zf}} & \Psi^m_{\phi}(M;E,F)\ar[r] &0,
}
\label{sm.3}\end{equation} 
where $x_{\zf}\in\CI(M^2_{k,\phi})$ is a boundary defining function for $\zf$.
\begin{proposition}
For $A\in \Psi^{m,\cE}_{k,\phi}(M;F,G)$ and $B\in \Psi^{m',\cF}_{k,\phi}(M;E,F)$ with index families $\cE$ and $\cF$ such that 
$$
\inf\cE|_{\zf}\ge 0, \quad \inf\cF|_{\zf}\ge 0, \quad \inf\Re(\cE|_{\rf}+\cF|_{\lf})>h+1\quad \mbox{and} \quad \Re(\cE|_{\rf_0}+\cF|_{\lf_0})>h+1,
$$
we have that 
\begin{equation}
  N_{\zf}(A\circ B)= N_{\zf}(A)\circ N_{\zf}(B)
\label{sm.4b}\end{equation}
with the composition on the right as $\phi$-operators.  
\label{sm.4}\end{proposition}
\begin{proof}
By Theorem~\ref{com.12}, the composition $A\circ B$ makes sense and its Schwartz kernel can be restricted to $\zf$.  This restriction comes in fact from the pushforward of the restriction of \eqref{com.10b} to $H^{k,\phi}_{1110}$.  In other words, $N_{\zf}(A\circ B)$ is given by the composition of $N_{\zf}(A)$ and $N_{\zf}(B)$ induced by $H^{k,\phi}_{1110}$ seen as triple space for $\zf$.  Since $H^{k,\phi}_{1110}$ is naturally identified with the $\phi$-triple space of \cite{Mazzeo-MelrosePhi}, the result follows.  
\end{proof}

Similarly, in terms of the vector bundle ${}^{k,\phi}N_{\tf}Y$ of \eqref{la.7}, the boundary hypersurface $\ff_0$ in $M^2_{k,\phi}$ is naturally the double space for ${}^{k,\phi}N_{\tf}Y$-suspended operators for the fiber bundle $\phi_{\tf}: \tf\to Y\times [0,\frac{\pi}2]_{\theta}$.  Hence, the normal operator $N_{\ff_0}(A)$ can be seen as a ${}^{k,\phi}N_{\tf}Y$-suspended operator.  In terms of the small calculus, this induces the short exact sequence 
\begin{equation}
\xymatrix{
0\ar[r] & x_{\ff_0}\Psi^{m}_{k,\phi}(M;E,F)\ar[r] & \Psi^m_{k,\phi}(M;E,F) \ar[r]^-{N_{\ff_0}} & \Psi^m_{\sus({}^{k,\phi}N_{\tf}Y)-\phi_{\tf}}(\tf;E,F) \ar[r] &0,
}
\label{sm.5}\end{equation}
where $x_{\ff_0}\in\CI(M^2_{k,\phi})$ is a boundary defining function for $\ff_0$.  

\begin{proposition}
For $A\in \Psi^{m,\cE}_{k,\phi}(M;F,G)$ and $B\in \Psi^{m',\cF}_{k,\phi}(M;E,F)$ with index families $\cE$ and $\cF$ such that 
\begin{equation}
\begin{gathered}
  \inf\Re\cE|_{\ff_0}\ge 0, \quad \inf\Re\cF|_{\ff_0}\ge 0, \quad \inf\Re(\cE|_{\fbf_0}+\cF|_{\fbf_0})>h+1, \\ 
  \inf\Re(\cE|_{\lf_0}+ \cF|_{\rf_0})>0\quad \mbox{and} \quad \inf\Re(\cE|_{\ff}+\cF|_{\lf})>h+1,
\end{gathered}
\label{sm.6b}\end{equation}
we have that
\begin{equation}
     N_{\ff_0}(A\circ B)= N_{\ff_0}(A)\circ N_{\ff_0}(B)
\label{sm.6c}\end{equation}
where the composition on the right is as ${}^{k,\phi}N_{\tf}Y$-suspended operators.
\label{sm.6}\end{proposition}
\begin{proof}
From Theorem~\ref{com.12}, we see that the composition $A\circ B$ makes sense as a $k,\phi$-pseudodifferential operators and the restriction of its Schwartz kernel to $\ff_0$ is well-defined.  Moreover, this restriction comes from the pushforward of the restriction of \eqref{com.10b} to $\ff^0_{T}$.  Thus the composition on the right of \eqref{sm.6c} is the one induced by $\ff^0_T$ seen as a triple space for $\ff_0$, which is precisely composition as ${}^{k,\phi}N_{\tf}Y$-suspended operators. 
\end{proof}

Finally, in terms of the vector bundle ${}^{k,\phi}N_{\sc}Y$ of \eqref{la.7}, the face $\ff$ does not quite correspond to the double space of ${}^{k,\phi}N_{\sc}Y$-suspended operators with respect to the fiber bundle $\phi_{\sc}: \sc\to Y\times [0,\infty)_k$.   Instead, because of the blow-up of $\Phi_0$ in \eqref{kfb.9d},  it is an adiabatic version of this suspended calculus, namely it is semi-classical in the suspension parameters with $k$ playing the role of the semi-classical parameter.  However, since suspended operators are already `classical' in the suspension parameter, insisting on having rapid decay at $\ff\cap \fbf_0$ and $\ff\cap \fbf$, the boundary hypersurface $\ff$ can be seen as a double space for $(k^{-1}){}^{k,\phi}N_{\sc}Y$-suspended operators.  That is, in terms of suspended operators, the effect of blowing up $\Phi_0$ in \eqref{kfb.9d} amounts to rescaling the suspension parameters by $k^{-1}$.  Notice that such an observation was implicitly used in \cite{Melrose-Rochon} to avoid introducing an extra blow-up.  In particular, the normal operator map $N_{\ff}$ induces the short exact sequence
\begin{equation}
\xymatrix{
0\ar[r] & x_{\ff}\Psi^m_{k,\phi}(M;E,F) \ar[r] & \Psi^m_{k,\phi}(M;E,F) \ar[r]^-{N_{\ff}} & \Psi^m_{\sus(V)-\phi_{\sc}}(\sc;E,F) \ar[r] & 0,
}
\label{sm.7}\end{equation}
where $x_{\ff}\in \CI(M^2_{k,\phi})$ is a boundary defining function for $\ff$ and $V:=(k^{-1}){}^{k,\phi}N_{\sc}Y$.

\begin{proposition}
For $A\in \Psi^{m,\cE}_{k,\phi}(M;F,G)$ and $B\in \Psi^{m',\cF}_{k,\phi}(M;E,F)$ with index families $\cE$ and $\cF$ such that
\begin{equation}
\begin{gathered}
\inf\Re \cE|_{\ff}\ge 0, \quad \inf\Re\cF|_{\ff}\ge 0, \quad \inf\Re(\cE|_{\fbf}+\cF|_{\fbf})>h+1, \\
 \inf\Re(\cE|_{\rf}+\cF|_{\lf})>h+1 \quad \mbox{and}\quad \inf\Re(\cE|_{\lf}+\inf\Re \cF|_{\rf})>0, 
\end{gathered}
\label{sm.8b}\end{equation}
we have that
\begin{equation}
  N_{\ff}(A\circ B)= N_{\ff}(A)\circ N_{\ff}(B)
\label{sm.8c}\end{equation}
with composition on the right induced by the boundary hypersurface $\ff^+_T$ seen as a triple space for $\ff$.  Furthermore, if $\cE|_{\fbf_0}=\cE|_{\fbf}=\emptyset$, then the composition on the right is as $(k^{-1}){}^{k,\phi}N_{\sc}Y$-suspended operators.
 \label{sm.8}\end{proposition}
 \begin{proof}
 By Theorem~\ref{com.12}, the composition $A\circ B$ is a $k,\phi$-pseudodifferential operator whose restriction to $\ff$ makes sense.  Furthermore, by the pushforward theorem, this restriction comes from the pushforward of the restriction of \eqref{com.10b} to $\ff^+_T$, hence \eqref{sm.8c} holds with the composition on the right induced by $\ff^+_T$ seen as a triple space for $\ff$.  By the discussion above, this corresponds to composition as $(k^{-1}){}^{k,\phi}N_{\sc}Y$-suspended operators when $\cE|_{\fbf_0}=\cE|_{\fbf}=\emptyset$. 
  \end{proof}

\section{Low energy limit of the resolvent of Dirac fibered boundary operators} \label{le.0}

Let $\eth_{\phi}\in \Diff^1_{\phi}(M;E)$ be the elliptic formally self-adjoint first order fibered boundary operator of \S~\ref{fdt.0}.  Suppose that Assumption~\ref{dt.4} holds and that $\eth_{\phi}$ is a Dirac operator associated to a fibered boundary metric $g_{\phi}$ and a structure of Clifford module on $E$ with respect to the Clifford bundle of the $\phi$-tangent bundle.  In particular, $\eth_h$ in \eqref{dt.3} is a family of Euclidean Dirac operators.  Let $\gamma\in \CI(M;\End(E))$ be self-adjoint as an operator in $\Psi^0_{\phi}(M;E)$ and suppose that 
\begin{equation}
  \gamma^2=\Id_E, \quad \eth_{\phi}\gamma+ \gamma\eth_{\phi}=0.
\label{le.1}\end{equation}
In terms of \eqref{dt.3} and \eqref{dt.8e}, suppose also that $\gamma$ anti-commutes with $D_v$, $\eth_h$, $c$ and $\eth_Y$.

In this section, we will consider the first order $k,\phi$-operator 
\begin{equation}
    \eth_{k,\phi}:= \eth_{\phi} + k\gamma.
\label{le.2}\end{equation}
By \eqref{le.1}, notice that 
\begin{equation}
    \eth_{k,\phi}^2= \eth^2_{\phi}+ k^2\Id_E.
\label{le.3}\end{equation}
In particular, for $k>0$, $\eth_{k,\phi}^2$ has positive spectrum and is invertible.  Essentially for the same reason, its normal operator is invertible, which means  by \cite{Mazzeo-MelrosePhi} that $\eth_{k,\phi}^2$ is invertible in the small $\phi$-calculus for $k>0$.    Since $\eth_{k,\phi}^{-1}= \eth_{k,\phi} (\eth_{k,\phi}^2)^{-1}$, we see that $\eth_{k,\phi}$ is  invertible as well in the small $\phi$-calculus for $k>0$.  On the other hand, when $k=0$, $\eth_{\phi}$ is typically not invertible in $\Psi^*_{\phi}(M;E)$, but as shown in \S~\ref{fdt.0}, it is at least Fredholm when acting on suitable Sobolev spaces with an inverse modulo compact operators in the large $\phi$-calculus.  This and the invertibility for $k>0$ can be combined to invert $\eth_{k,\phi}$ as a $k,\phi$-operator as we will now explain.  In order to do this, we need to make the following hypothesis.

\begin{assumption}
There exists $\epsilon>0$ such that the interval $(-1-\epsilon,\epsilon)$ contains no critical weight of the indicial family $I(D_b,\lambda)$ of Definition~\ref{dt.6}.  There is also $\epsilon\le \epsilon_1$ such that each element $\psi$ of the kernel of $D_{\phi}$ in $L^2_b(M;E)$ is such that $x^{-\epsilon_1}\psi$ is bounded.  \label{le.4}\end{assumption} 

\begin{remark}
By Corollary~\ref{dt.22}, we can take $\epsilon_1=\epsilon$.  However, there are situations where we can take $\epsilon_1>\epsilon$ as the next example shows, yielding better control on the inverse of \eqref{le.3}.
\end{remark}
\begin{example}
If $\eth_{\phi}$ is the Hodge-deRham operator acting on forms valued in a flat vector bundle, then by Lemma~\ref{do.8c},  provided the de Rham cohomology group
\begin{equation}
H^q(Y;\ker D_v)=\{0\} \quad \mbox{for} \quad q\in\{\frac{h-1}2,\frac{h}2,\frac{h+1}2\},  
\label{HdR.2}\end{equation}
Assumption~\ref{le.4} will be satisfied for $g_{\phi}$ with metric $g_Y$ sufficiently scaled (so that the positive spectrum of $\mathfrak{d}^2$ is sufficiently large).  Moreover, if also
\begin{equation}
  H^q(Y;\ker D_v)=\{0\} \quad \mbox{for} \quad q\in\{\frac{h-2}2,\frac{h+2}2\},
\label{HdR.3}\end{equation}
then we can also assume that $\epsilon_1+\epsilon\ge 2\epsilon >1$, again provided the metric $g_Y$ arising in the asymptotic behavior of $g_{\phi}$ is sufficiently small.  Finally, if \eqref{HdR.2} holds, but not \eqref{HdR.3}, in which case $h$ is necessarily even, then we can still ensure that $\epsilon_1+\epsilon>1$ by requiring that the $L^2$-kernel of $\eth_{\phi}$ is trivial, in fact requiring to be  trivial only in degree $q\in \{\frac{h-2}2, \frac{h}2, \frac{h+2}2, \frac{h+4}2\}$ in the scattering case (when $Y=\pa M$ and $\phi$ is the identity map).
Indeed, in this case, again assuming $g_Y$ is sufficiently small, we can take $\epsilon_1>1$, since by Lemma~\ref{do.8c} the indicial root $\lambda=\frac12$ coming from the non-triviality of $H^{\frac{h\pm 2}2}(Y;\ker D_v)$  does not show up in the polyhomogeneous expansion of elements of the $L^2$-kernel.  In general, with $h$ odd or even, we can ensure that $\epsilon_1>1$ by scaling $g_Y$ provided either
\begin{equation}
   H^q(Y;\ker D_v)=\{0\} \quad \mbox{for} \quad q=\frac{h\pm\ell}2, \quad \ell\in \{0,1,2,3\},
\label{HdR.4}\end{equation}
or  that 
 we know that \eqref{HdR.2} holds and that the $L^2$-kernel of $\eth_{\phi}$ is  trivial, in fact only in degree $q=\frac{h+ 1\pm \ell}2$ with $\ell\in\{1,2,3,4,5\}$ in  the scattering case. 
\label{HdR.1}\end{example}

As in \S~\ref{fdt.0}, it will be convenient, instead of $\eth_{k,\phi}$, to work with the conjugated operator
\begin{equation}
 D_{k,\phi}:= x^{-\frac{h+1}2}\eth_{k,\phi}x^{\frac{h+1}2}= D_{\phi} +k\gamma
\label{le.4b}\end{equation}
acting formally on $L^2_b(M;E)$.  In terms of this conjugated operator, we have the following characterization of the inverse.

\begin{theorem}
There exists $G_{k,\phi}\in \Psi^{-1,\cG}_{k,\phi}(M;E)$ such that 
$$
    D_{k,\phi}G_{k,\phi}=\Id, \quad G_{k,\phi}D_{k,\phi}=\Id,
$$
where $\cG$ is an index family given by the empty set at $\lf, \rf$ and $\fbf$, while 
\begin{equation}
\begin{gathered}
  \inf\Re \cG|_{\zf}\ge-1, \quad \inf \Re\cG|_{\fbf_0}\ge h, \quad   \cG|_{\ff}=\bbN_0, \quad \inf\Re \cG|_{\ff_0}\ge 0,  \\
   \mbox{and} \quad \inf\Re \cG|_{\lf_0}\ge \nu,  \quad \inf\Re \cG|_{\rf_0}\ge h+1+\nu \quad \mbox{with} \quad \nu:=\min\{\epsilon,\epsilon_1-1\}.
\end{gathered}
\label{le.46b}\end{equation}
Furthermore, if $\epsilon+\epsilon_1> 1$ for $\epsilon$ and $\epsilon_1$ in Assumption~\ref{le.4}, then in fact $\cG|_{\zf}=(\bbN_0-1)\cup \cN$ with $\cN$ an index set with $\inf \Re \cN>0$.  
\label{le.46}\end{theorem}  

\begin{remark}
Coming back to $\eth_{k,\phi}=\eth_{\phi}+ k\gamma$, we see that $x^{\frac{h+1}2}G_{k,\phi}x^{-\frac{h+1}2}$ is such that 
$$
      \eth_{k,\phi}(x^{\frac{h+1}2}G_{k,\phi}x^{-\frac{h+1}2})=(x^{\frac{h+1}2}G_{k,\phi}x^{-\frac{h+1}2})\eth_{k,\phi}=\Id.
$$
\label{le.46c}\end{remark}

To prove Theorem~\ref{le.46}, we will closely follow the approach of Guillarmou and Hassell \cite{GH1,GH2}, but relying on the $k,\phi$-calculus described in \S~\ref{kfb.0}.  Roughly, the strategy will consist in constructing a parametrix with error term vanishing to some positive order in $k$ as $k\searrow 0$ when it is described in terms of the lift from the right of a $b$-density.  Thus in terms of the density bundle ${}^{k,\phi}\Omega_R(M)$, this will correspond to decay of order strictly bigger than $h+1$  at $\rf$ and $\fbf$ and decay of positive order at the other boundary hypersurfaces where $k=0$ in $M^2_{k,\phi}$.   Once we get such a good parametrix, we can construct the actual inverse from the parametrix by using a Neumann series argument.  The construction of the parametrix and the proof of Theorem~\ref{le.46} will involve a few steps, namely, we will need to invert $D_{k,\phi}$ at $\zf$, $\ff_0$,  $\fbf_0$ and $\ff$ making sure along the way the error term decays suitably elsewhere.

\subsection*{Step 0: Inversion at $\zf$ and $\ff_0$}

Consider then the fibered cusp operator
\begin{equation}
   D_{\fc}= x^{-\frac12}D_{\phi} x^{-\frac12}.
\label{le.5}\end{equation}
This operator is formally self-adjoint on $L^2_b(M;E)$.  By Assumption~\ref{le.4}, we can take $\delta=\frac12$ and $\mu=\frac12+\epsilon$ in Corollary~\ref{dt.31} to obtain an inverse $G_{-\frac12}: L^2_b(M;E)\to \cD_{-\frac12}$ such that 
\begin{equation}
   G_{-\frac12}D_{\fc}= \Id-\Pi, \quad D_{\fc}G_{-\frac12}=\Id-\Pi,
\label{le.6}\end{equation}
where $\Pi$ is the orthogonal projection in $L^2_b(M;E)$ onto the finite dimensional kernel of $D_{\fc}$ and $\cD_{-\frac12}\subset L^2_b(M;E)$ is the minimal domain of $D_{\fc}$.  If $\Pi=0$, then as in \cite{GH1}, one can take $x^{-\frac12}G_{-\frac12}x^{-\frac12}$ to invert $D_{k,\phi}= D_{\phi}+k\gamma$ at $\zf$, for
\begin{equation}
D_{\phi}x^{-\frac12}G_{-\frac12}x^{-\frac12}= x^{\frac12}D_{\fc}G_{-\frac12}x^{-\frac12}= x^{\frac12}\Id x^{-\frac12}=\Id.
\label{le.7}\end{equation}
If instead $\Pi\ne 0$, then \eqref{le.7} becomes
\begin{equation}
   D_{\phi}x^{-\frac12}G_{-\frac12}x^{-\frac12}= \Id - x^{\frac12}\Pi x^{-\frac12}
\label{le.8}\end{equation}
and we can proceed as in \cite{GH2} to remove the error term.  More precisely,  let $\{\varphi_j\}_{j=1}^J$ be an orthonormal basis of the kernel of $D_{\fc}$ in $L^2_b(M;E)$, so that 
$$
       \Pi= \sum_{j=1}^J (\pr_L^*\varphi_j)\pr_R^*(\varphi_j\nu_b)
$$
for the $b$-density $\nu_b= x^{h+1}dg_{\phi}$ with respect to which $D_{\fc}$ is formally self-adjoint.  By 
Assumption~\ref{le.4}, $\varphi_j=\mathcal{O}(x^{\frac12+\epsilon_1})$ near $\pa M$.  In particular, $\{x^{-\frac12}\varphi_j\}$ is a basis of the kernel of $D_{\phi}$ in $L^2_b(M;E)$.  If $\{\psi_j\}_{j=1}^J$ is a choice of orthonormal basis of $\ker_{L^2_b} D_{\phi}$, then 
$$
  \psi_i= \sum_{j=1}^J \alpha_{ij} x^{-\frac12}\varphi_j
$$
for some $\alpha_{ij}$ and 
\begin{equation}
        \Pi_{\ker_{L^2_b}D_{\phi}}= \sum_{j=1}^J (\pr_L^*\psi_j)\pr_R^*(\psi_j\nu_b)
\label{le.9}\end{equation}
is the orthogonal projection onto $\ker_{L^2_b}D_{\phi}$.  If $\alpha^{ij}$ denotes the inverse of the matrix $\alpha_{ij}$, then
\begin{equation}
   x^{-\frac12}\varphi_i= \sum_{j=1}^J \alpha^{ij} \psi_j.
\label{le.10}\end{equation}
In terms of the projection \eqref{le.9}, we compute that 
\begin{equation}
 \Pi_{\ker_{L^2_b}D_{\phi}}(x^{\frac12}\varphi_j) = \sum_{k=1}^J \left( \int_M \psi_k x^{\frac12}\varphi_j \nu_b \right)\psi_k = \sum_{k=1}^{J}\left( \int_M \left( \sum_{i=1}^J \alpha_{ki}x^{-\frac12}\varphi_i \right)x^{\frac12}\varphi_j\nu_b \right)\psi_k= \sum_{k=1}^J \alpha_{kj}\psi_k.
\label{le.11}\end{equation}
This means that 
\begin{equation}
   \psi^{\perp}_j:= x^{\frac12}\varphi_j- \Pi_{\ker_{L^2_b}D_{\phi}}(x^{\frac12}\varphi_j)=\mathcal{O}(x^{\epsilon_1}) \quad \mbox{near} \; \pa M.
\label{le.12}\end{equation}
It also follows from \eqref{le.11} that
\begin{equation}
\begin{aligned}
\sum_{j=1}^J \pr_L^*(\Pi_{\ker_{L^2_b}D_{\phi}}(x^{\frac12}\varphi_j))\pr_R^*(x^{-\frac12}\varphi_j\nu_b) &=
\sum_{j,k,\ell} \alpha_{kj}\alpha^{j\ell}(\pr_L^*\psi_k)\pr_R^*(\psi_{\ell}\nu_b)= \sum_k (\pr_L^*\psi_k)\pr_R^*(\psi_{k}\nu_b) \\
&= \Pi_{\ker_{L^2_b}D_{\phi}}.
\end{aligned}
\label{le.13}\end{equation}
\begin{lemma}
There exists $\chi_k$ such that $D_{\fc}\chi_k=x^{-\frac12}\psi^{\perp}_k$.  Moreover, $\chi_k$ is smooth on $M\setminus \pa M$ with polyhomogeneous expansion at $\pa M$ having leading term of order at least $x^{\nu+\frac12}$ with $\nu=\min\{\epsilon, \epsilon_1-1\}$.  
\label{le.14}\end{lemma}
\begin{proof}
By \eqref{le.10}, $x^{-\frac12}\varphi_i$ is orthogonal to $\psi_k^{\perp}$, which implies that $\varphi_i$ is orthogonal to $x^{-\frac12}\psi_k^{\perp}$.  Thus, $x^{-\frac12}\psi_k^{\perp}$ is orthogonal to $\ker_{L^2_b}D_{\fc}$.  Hence, taking $\delta=\frac12$ in \eqref{dt.30b}, this means that $x^{-\frac12}\psi_k^{\perp}$ is in the range of $\Id-P_2$, so that 
$$
\chi_k:= G_{-\frac12}(x^{-\frac12}\psi_k^{\perp})
$$ 
is such that
\begin{equation}
\begin{aligned}
D_{\fc}\chi_k &= x^{-\frac12}D_{\phi}x^{-\frac12}\chi_k= x^{-\frac12}D_{\phi}x^{-\frac12}G_{-\frac12}(x^{-\frac12}\psi^{\perp}_k) \\
&= (\Id-P_2)(x^{-\frac12}\psi_k^{\perp})\\
&= x^{-\frac12}\psi^{\perp}_k
\end{aligned}
\end{equation}
as claimed.  Moreover, by Assumption~\ref{le.4}, Corollary~\ref{dt.31} and Proposition~\ref{phi.17b}, $\chi_k$ is polyhomogeneous at $\pa M$ with  $\chi_k=\mathcal{O}(x^{\nu+\frac12})$. \end{proof}

Using \eqref{le.13}, we see that 
\begin{equation}
\begin{aligned}
D_{\phi}x^{-\frac12} &\left(G_{-\frac12}+ \sum_{j=1}^J (\pr_L^*\chi_j \pr_R^*(\varphi_j\nu_b)+  \pr_L^*\varphi_j\pr_R^*(\chi_j\nu_b)) \right)x^{-\frac12} \\
&\hspace{5cm}= \Id -x^{\frac12}\Pi x^{-\frac12}+ \sum_{j=1}^J \pr_L^*\psi_j^{\perp}\pr_R^*(x^{-\frac12}\varphi_j\nu_b) \\
&\hspace{5cm}=\Id+ \sum_{j=1}^J\left( -\pr_L^*(x^{\frac12}\varphi_j)\pr_R^*(x^{-\frac12}\varphi_j \nu_b) + \pr_L^*\psi^{\perp}_j\pr_R^*(x^{-\frac12}\varphi_j\nu_b) \right) \\
&\hspace{5cm}=\Id - \sum_{j=1}^J \pr_L^*(\Pi_{\ker_{L^2_b}D_{\phi}}(x^{\frac12}\varphi_j))\pr_R^*(x^{-\frac12}\varphi_j\nu_b) \\
&\hspace{5cm}= \Id-\Pi_{\ker_{L^2_b}D_{\phi}}.
\end{aligned}
\label{le.15}\end{equation}

To construct the inverse of $D_{k,\phi}$, this suggests to consider the approximate inverse
\begin{equation}
Q_0:= k^{-1}\gamma G_{\zf}^{-1}+ G_{\zf}^0
\label{le.16}\end{equation}
with 
\begin{equation}
 \begin{gathered}
 G^{-1}_{\zf}:= \sum_{j=1}^J \pr_L^{*}\psi_j \pr_R^*(\psi_j\nu_b)= \Pi_{\ker_{L^2_b}D_{\phi}}, \\
 G^0_{\zf}= x^{-\frac12}\left( G_{-\frac12} + \sum_{j=1}^J \left( \pr_L^*\chi_j\pr_R^*(\varphi_j\nu_b)+ \pr_L^*\varphi_j \pr_R^*(\chi_j \nu_b) \right)  \right)x^{-\frac12}.
\end{gathered} 
\label{le.17}\end{equation}
On $M^2_{\phi}\times [0,\infty)_k$, it is such that
\begin{equation}
(D_{\phi}+\gamma k)Q_0= \Id + R_0 \quad \mbox{with} \; R_0= k\gamma G^0_{\zf}.
\label{le.18}\end{equation} 

When we lift this parametrix to $M^2_{k,\phi}$, we can regard $Q_0$ as an element of $\Psi^{-1,\cQ_0}_{k,\phi}(M;E)$ with index family $\cQ_0$ such that 
\begin{equation}
\begin{gathered}
   \cQ_0|_{\zf}= \bbN_0-1, \quad \inf \Re\cQ_0|_{\fbf}=\inf \Re\cQ_0|_{\fbf_0}\ge h, \quad  \inf\Re\cQ_0|_{\ff_0}\ge0, \quad  \inf\Re\cQ_0|_{\ff}\ge 0 \quad \mbox{and} \\
    \inf\Re \cQ_0|_{\lf}=\inf\Re \cQ_0|_{\lf_0}\ge \nu,  \quad \inf\Re \cQ_0|_{\rf}=\inf\Re \cQ_0|_{\rf_0}\ge h+1+\nu, \quad \mbox{with} \quad \nu:=\min\{\epsilon,\epsilon_1-1\},
\end{gathered}\label{le.18b}\end{equation} 
while the error term $R_0$ has leading order 1 at $\zf$ and $\ff_0$,  leading order 0 at $\ff$, leading order $h+1$ at $\fbf_0$, leading order $h$ at $\fbf$, leading order $1+\nu$ at $\lf_0$, leading order $\nu$ at $\lf$, leading order $h+2+\nu$ at $\rf_0$ and leading order $h+1+\nu$ at $\rf$.  This means that $Q_0$ also inverts $D_{k,\phi}$ at  $\ff_0$.  
At $\ff_0$, the lift of $Q_0$ gives the expected model for the inverse of $D_{k,\phi}$.  In fact, since $\psi_k$ and $\chi_k$ are of order $x^{\epsilon_1}$ and $x^{\frac12+\nu}$ at the boundary, notice that $k^{-1}G^{-1}_{\zf}$ is of order $h+2\epsilon_1$ in terms of $(\phi,k)$-densities at $\ff_0$, while 
$$
x^{-\frac12}\left(\sum_{j=1}^J \left( \pr_L^*\chi_j\pr_R^*(\varphi_j\nu_b)+ \pr_L^*\varphi_j \pr_R^*(\chi_j \nu_b)\right)\right)x^{-\frac12}
$$
is of order $h+1+\nu+\epsilon_1$.  Hence, the term of order $0$ at $\ff_0$ of $Q_0$ comes exclusively from the term $x^{-\frac12}G_{-\frac12}x^{-\frac12}$ in $G^0_{\zf}$.

\subsection*{Step 1: Cutting off to enforce rapdid decay at $\lf$, $\fbf$ and $\rf$}

At $\fbf_0$, the error term $R_0$ does not vanish at order $h+1$.   Moreover,  the error term $R_0$ does not vanish rapidly at $\fbf$, $\lf$ and $\rf$.  This forces us to seek a better model to invert the operator at $\fbf_0$.  Looking at the behavior of $Q_0$ near $\fbf_0$, notice that $k^{-1}G^{-1}_{\zf}$ is $\mathcal{O}(x_{\fbf_0}^{h+2\epsilon_1})$, while $G^0_{\zf}$ has main term of order $h$ coming from $x^{-\frac12}G_{-\frac12}x^{-\frac12}$ (in terms of right $k,\phi$-densities).

On the other hand, before performing the last blow-up in \eqref{kfb.9bb}, we can consider the coordinates
\begin{equation}
k, \kappa= \frac{k}{x}, \kappa'= \frac{k}{x'}, y,y', z,z'
\label{le.19}\end{equation}
in the interior of $\fbf_0$, where $y$ and $z$ denote coordinates on the base and the fibers of a local trivialization of $\phi: \pa M\to Y$.  Recalling \eqref{dt.8b}, we see that in terms of these coordinates, the restriction of the operator $D_{k,\phi}$ to $\fbf_0$ is given by   
\begin{equation}
 D_{k,\phi}= D_v + kD_{\cC} \quad \mbox{with} \quad D_{\cC}:= -c\frac{\pa}{\pa \kappa}+ \frac{1}{\kappa}D_Y+ \gamma.
\label{le.20}\end{equation}
Here, the operator $D_{\cC}$ can be seen as an operator on the cone $Y\times [0,\infty)_{\kappa}$ with cone metric 
$d\kappa^2 + \kappa^2 g_Y$ acting on the sections of $\ker D_v$.  Near the apex of the cone, that is, for $\kappa<1$, the operator $\kappa D_{\cC}$ can be treated as $b$-operator, while for $\kappa>1$, the operator $D_{\cC}$ can be seen as a scattering operator in the sense of \cite{MelroseGST}.  This is consistent with the fact that, when we forget about the fibers of $\phi: \pa M\to Y$, the boundary hypersurface $\fbf_0$, before performing the last blow up in \eqref{kfb.9bb}, is just the double space for $Y\times [0,\infty]$ corresponding to a $b$-double space near $\kappa=\kappa'=0$ an a scattering double space near $\kappa=\kappa'=\infty$. 

More precisely, this double space is given by 
\begin{equation}
  (Y\times[0,\infty])^2_{b,\sc}= [Y^2 \times [0,\infty]^2; Y^2\times \{0\}^2, Y^2\times\{\infty\}^2, B_{\sc}],
\label{le.21}\end{equation}   
where $B_{\sc}$ is the intersection of the lifted diagonal with the boundary hypersurface created by the blow-up of $Y^2\times \{\infty\}^2$.  
\begin{figure}[h]
\begin{tikzpicture}
\draw(1,0) arc [radius=1, start angle=0, end angle=90];
\draw(1,0)--(4,0);
\draw(0,1)--(0,4); 
\draw(4,0)--(4,3);
\draw(0,4)--(3,4);
\draw(3,4) arc [radius=1, start angle=180, end angle=210];
\draw(4,3) arc [radius=1, start angle=270, end angle=240];
\draw (3.1339746,3.5) arc[radius=0.258819, start angle=135, end angle=315];
\node at (2.5,-0.3) {$\rf_0(Y)$};
\node at (-0.6,2.5) {$\lf_0(Y)$};
\node at (0.3,0.3) {$\fb_0(Y)$};
\node at (4.6,2) {$\lf_{\infty}(Y)$};
\node at (2,4.2) {$\rf_{\infty}(Y)$};
\node at (2.7,2.9) {$\sc(Y)$};
\node at (4.3,3.2) {$\fb_{\infty}(Y)$};
\node at (3.7,3.9) {$\fb_{\infty}(Y)$};
\end{tikzpicture}
\caption{The double space $(Y\times [0,\infty])^2_{b,\sc}$}
\label{fig.7}\end{figure}
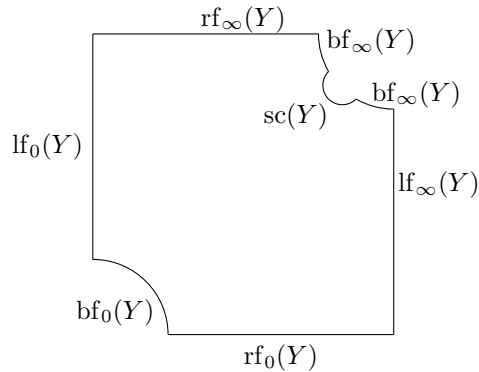

Denote by $\fb_0(Y)$, $\fb_{\infty}(Y)$ and $\sc(Y)$ the boundary hypersurfaces created by these three blow-ups and let $\lf_0(Y), \rf_0(Y), \lf_{\infty}(Y)$ and $\rf_{\infty}(Y)$ be the lifts of the boundary hypersurfaces $Y^2\times\{0\}\times [0,\infty]$, $Y^2\times[0,\infty]\times\{0\}$, $Y^2\times\{\infty\}\times [0,\infty]$ and $Y^2\times [0,\infty]\times\{\infty\}$. In terms of the boundary hypersurface $\fbf_0$, notice that $\fb_0(Y)$, $\fb_{\infty}(Y)$ and $\sc(Y)$ correspond to $\fbf_0\cap \zf$, $\fbf_0\cap \fbf$ and $\fbf_0\cap \ff$, while $\lf_{\infty}(Y)$ and $\rf_{\infty}(Y)$ correspond to $\fbf_0\cap \lf$ and $\fbf_0\cap \rf$.  However, for $\lf_0(Y)$ and $\rf_0(Y)$, there is a small twist since they correspond respectively to $\fbf_0\cap \rf$ and $\fbf_0\cap \lf$ instead of $\fbf_0\cap \lf$ and $\fbf_0\cap \rf$ as one could have naively expected.  This is consistent with the fact that, for instance, $Y^2\cap \{0\}\times[0,\infty]$ intersects $Y^2\times [0,\infty]\times \{\infty\}$, but not $Y^2\times \{\infty\}\times [0,\infty]$.

Now, by construction, the term $q_0$ of order $h$ of $Q_0$ at $\fbf_0$ is such that $\Pi_hq_0\Pi_h=q_0$, where $\Pi_h$ is the projection of \eqref{dt.5}.  To enforce rapid decay of the error term at $\fbf$, $\lf$ and $\rf$, we can take $Q_1$ to be $Q_0$ smoothly cut-off near $\fbf$, $\lf$ and $\rf$, but insisting that $Q_1=Q_0$ away from a small neighborhood of $\fbf$, $\lf$ and $\rf$, in particular near $\ff_0$ and $\zf$.  In this case, we will have that
\begin{equation}
  D_{k,\phi}Q_1= \Id - R_1
\label{le.22}\end{equation} 
with $R_1$ now having a term $r_1$ of order $h+1$ at $\fbf_0$ vanishing near $\fbf_0\cap \fbf$, $\fbf_0\cap\lf$, $\fbf_0\cap \rf$ and $\fbf_0\cap\ff_0$.  Choosing our cut-off function to be constant in the fibers of the lift of the fiber bundle 
$$
\phi\times \phi: \pa M\times \pa M\to Y\times Y 
$$ 
to $\fbf_0$,
we can also ensure that the term $q_1$ of order $h$ of $Q_1$ at $\fbf_0$ and the term $r_1$ of order $h+1$ of $R_1$ at $\fbf_0$ are such that 
$$
    \Pi_h q_1\Pi_h= q_1 \quad \mbox{and} \quad \Pi_h r_1\Pi_h= \Pi_h r_1.
$$  

We need also to choose the cut-off function near $\lf$ and $\rf$ in such a way that it does not introduce more singular terms in the expansion of the error terms at $\lf_0$ and $\rf_0$.  Near $\rf\cap \rf_0$, this can be done in terms of the right variable $\frac{x'}k$, ensuring that the error terms vanishes to order $h+2+\nu$ at $\rf_0$.  Indeed, the problematic term to cut-off is $k^{-1}G^{-1}_{\zf}$, but cutting off in this manner, that is using a cut-off function in $\frac{x'}k$, we see that $D_{\phi}$ applied to the cut-off of this term still gives zero, so yields no singular term, while $k\gamma$ yields a term of order $h+2+\nu$.  Cutting off the term $G^0_{\zf}$ in the same manner also yields at term of order $h+1+\epsilon_1$ by \eqref{le.15}. At $\lf$, we should instead cut off using a cut-off function constant in the fibers of the lift of the fiber bundle $\phi\times \Id: \pa M\times M\to Y\times M$ on $\lf_0$.  In terms of the decomposition \eqref{dt.8b}, we thus see that $D_{k,\phi}$ applied to the cut-off of $k^{-1}G_{\zf}^{-1}$ gives a term of order $\epsilon_1$ at $\lf_0$.  From \eqref{le.15}, we also see that the terms of order $\nu$ or less of $G^0_{\zf}$ at $\lf_0$ are in $\ker D_{v}$, where $\nu=\min\{\epsilon,\epsilon_1-1\}$.  Hence, we see from Lemma~\ref{dt.7} that cutting off $G^0_{\zf}$ near $\lf_0\cap\lf$ yields an error term of  order $\nu+1$ at $\lf_0$.  

To summarize, cutting off to get rapid decay at $\lf$, $\fbf$ and $\rf$, we get $Q_1$ in \eqref{le.22} where we can assume that the error term $R_1$ is in $\Psi^{-1,\cR_1}_{k,\phi}(M;E)$ with index family $\cR_1$ given by the empty set at $\lf$, $\fbf$ and $\rf$ and such that 
\begin{equation}
\begin{gathered}
\inf\Re\cR_1|_{\ff}\ge 0, \quad \inf \Re\cR_1|_{\fbf_0}\ge h+1, \quad \inf\Re\cR_1|_{\ff_0}\ge1, \\ 
\inf\Re\cR_1|_{\rf_0}\ge h+2+\nu, \quad 
 \inf\Re \cR_1|_{\lf_0}\ge 1+\nu, \quad \inf\Re \cR_1|_{\zf}\ge 1.
\end{gathered}
\label{le.22a}\end{equation}

\subsection*{Step 2: Inverting at $\fbf_0$}

To get rid of $r_1$ at $\fbf_0$, this means that in terms of \eqref{le.20}, we should try to find $q_2$ such that
\begin{equation}
      kD_{\cC}(q_2)=\Pi_h r_1.  
\label{le.22b}\end{equation}
To achieve this, we need to analyse the operator $D_{\cC}$ in terms of the double space \eqref{le.21}.  More precisely, we will invert $D_{\cC}$ using the pseudodifferential operators defined by the double space \eqref{le.21}.  To define this pseudodifferential calculus, let $\Delta_{b,\sc}$ be the lift of the diagonal in $(Y\times[0,\infty])^2$ to $(Y\times [0,\infty])^2_{b,\sc}$.  Let also ${}^{b,\sc}\Omega(Y\times[0,\infty])$ be the density bundle on $Y\times[0,\infty]$ corresponding to a $b$-density bundle near $\kappa=0$ and to a $\sc$-density bundle near $\kappa=\infty$.  If $\pr_L$ and $\pr_R$ are the projections $(Y\times [0,\infty])^2\to Y\times [0,\infty]$ on the left and right factors respectively and if $\beta_{b,\sc}: (Y\times [0,\infty])^2_{b,\sc}\to (Y\times[0,\infty])^2$ is the natural blow-down map, we can consider the lift from the right of the $b,\sc$-density bundle,
$$
     {}^{b,\sc}\Omega_R(Y\times [0,\infty]):= \beta^*_{b,\sc}\pr_R^* {}^{b,\sc}\Omega(Y\times[0,\infty]).
$$
If $F$ is a vector bundle on $Y\times [0,\infty]$, one can also consider the bundle 
$$
\Hom_{b,\sc}(F,F):=\beta_{b,\sc}^*(\pr_L^*\otimes \pr_R^*F^*).
$$
With this notation, the small calculus of $b,\sc$-operators acting on sections of $F$ can be defined as the union over all $m\in\bbR$ of the spaces
\begin{multline}
\Psi^m_{b,\sc}(Y;F) := \{ \kappa \in I^m((Y\times [0,\infty])^2_{b,\sc}, \Delta_{b,\sc}; \Hom_{b,\sc}(F,F)\otimes {}^{b,\sc}\Omega_R(Y\times [0,\infty])) \; |  \\
\kappa\equiv 0 \; \mbox{at}  \; \pa (Y\times[0,\infty])^2_{b,\sc}\setminus (\fb_0(Y)\cup \sc(Y))  \}.
\label{bsc.1}\end{multline}
If $\cE$ is an indicial family for $(Y\times [0,\infty])^2_{b,\sc}$, we can define more generally the spaces
\begin{gather}
\label{bsc.2} \Psi^{-\infty,\cE}_{b,\sc}(Y;F):= \cA^{\cE}_{\phg}(((Y\times [0,\infty])^2_{b,\sc}; \Hom_{b,\sc}(F,F)\otimes {}^{b,\sc}\Omega_R(Y\times [0,\infty])), \\
\label{bsc.3} \Psi^{m,\cE}_{b,\sc}(Y;F):= \Psi^m_{b,\sc}(Y;F) + \Psi^{-\infty,\cE}_{b,\sc}(Y;F), \quad m\in\bbR.
\end{gather}
By the discussion above, we could alternatively define the $b,\sc$-calculus by restriction of the $b$-$\sc$ transition calculus of $Y\times [0,\infty]$ to the boundary hypersurface $\fb_0$.  In particular, restriction to $\fb_0$ of the composition result of Theorem~\ref{com.14} yields the following composition result for the $b,\sc$-calculus.
\begin{theorem}
Let $A\in\Psi^{m,\cE}_{b,\sc}(Y;F)$ and $B\in\Psi^{m',\cF}(Y;F)$ be pseudodifferential $b,\sc$-operators with index families $\cE$ and $\cF$ given by the empty set at $\lf_{\infty}(Y)$, $\rf_{\infty}(Y)$ and $\fb_{\infty}(Y)$.  Assume furthermore that 
$$
   \inf\Re\cE|_{\sc(Y)}\ge 0, \quad \inf\Re\cF_{\sc(Y)}\ge 0 \quad \mbox{and} \quad \inf\Re(\cE|_{\rf_0(Y)}+\cF|_{\lf_0(Y)})>0.  
$$
In this case, $A\circ B\in \Psi^{m+m',\cG}_{b,\sc}(Y;F)$ with index family $\cG$ given by
\begin{equation}
\begin{array}{ll}
\cG|_{\sc(Y)}= \cE|_{\sc(Y)}+ \cF|_{\sc(Y)}, & \cG|_{\fb_0(Y)}= (\cE_{\fb_0(Y)}+ \cF_{\fb_0(Y)})\overline{\cup}(\cE|_{\lf_0(Y)}+\cF|_{\rf_0(Y)}), \\
\cG|_{\lf_0(Y)}=(\cE|_{\fb_0(Y)}+ \cF|_{\lf_0(Y)})\overline{\cup}(\cE|_{\lf_0(Y)}), & \cG|_{\rf_0(Y)}=(\cE|_{\rf_0(Y)}+\cF|_{\fb_0(Y)})\overline{\cup}(\cF|_{\rf_0(Y)}), \\
 \cG|_{\fb_{\infty}(Y)}=\cG|_{\lf_{\infty}(Y)}=\cG|_{\rf_{\infty}(Y)}=\emptyset. &   
\end{array}
\label{bsc.4b}\end{equation}
\label{bsc.4}\end{theorem}
\begin{proof}
Recall that in terms of the boundary hypersurface $\fb_0$ in Theorem~\ref{com.14},  the boundary hypersurfaces $\fb_0(Y)$, $\fb_{\infty}(Y)$ and $\sc(Y)$ correspond to $\fb_0\cap \zf$, $\fb_0\cap \fb$ and $\fb_0\cap \sc$, while $\lf_{\infty}(Y)$ and $\rf_{\infty}(Y)$ correspond to $\fb_0\cap \lf$ and $\fb_0\cap \rf$.  However, for $\lf_0(Y)$ and $\rf_0(Y)$, there is a small twist since they correspond respectively to $\fb_0\cap \rf$ and $\fb_0\cap \lf$.  With this understood, it suffices to look at what happens to the term of order zero at $\fb_0$ in \eqref{com.14b}.  In particular, the condition 
$$
\inf\Re(\cE|_{\rf_0(Y)}+\cF|_{\lf_0(Y)})>0
$$
is there to ensure that in \eqref{com.14b}, there are no terms of negative order at $\fb_0$ and that the term of order zero comes exclusively from the terms of order zero at $\fb_0$ of the operators that are composed.  Thus, it suffices to replace the index sets at $\fb_0$ by $0$ and restrict  \eqref{com.14b} to $\fb_0$ to obtain \eqref{bsc.4b}. 
\end{proof}

Now that we have properly defined the pseudodifferential $b,\sc$-operators, we can come back to the question of inverting the operator $D_{\cC}$ in this calculus.  First, as mentioned before, near  $\fb_0(Y)$, $\kappa D_{\cC}$ is $b$-operator.  From \eqref{dt.8b} and \eqref{le.20}, we see that its indicial family is given by
\begin{equation}
      I(\kappa D_{\cC},\lambda)= D_Y-c\lambda= I(D_b,-\lambda).
\label{le.23}\end{equation}
We also know from the parametrix construction of \S~\ref{fdt.0} that the term of order $h$ of $Q_1$ at $\fbf_0\cap\zf$ is precisely
$$
x^{-\frac12}\left(\frac{1}{2\pi} \int_{-\infty}^{\infty} e^{i\xi\log s} I(D_b,-\frac12+i\xi)^{-1}d\xi\right) (x')^{-\frac12} \frac{dx'}{x'}
$$
where $s=\frac{x}{x'}$.  In terms of the coordinates \eqref{le.19}, this becomes
\begin{equation}
   k^{-1} \kappa^{\frac12}\left(\frac{1}{2\pi} \int_{-\infty}^{\infty} e^{i\xi\log\left(\frac{\kappa}{\kappa'}\right)} I(\kappa D_{\cC},\frac12+i\xi)^{-1}d\xi\right) (\kappa')^{\frac12} \frac{d\kappa'}{\kappa'}
\label{le.24}\end{equation}
This suggests that we should consider 
$$
G_{b}^{\frac12}=\left(\frac{1}{2\pi} \int_{-\infty}^{\infty} e^{i\xi\log\left(\frac{\kappa}{\kappa'}\right)} I(\kappa D_{\cC},\frac12+i\xi)^{-1}d\xi\right) \frac{d\kappa'}{\kappa'},
$$
the inverse of $\kappa^{-\frac12}(\kappa D_{\cC})\kappa^{\frac12}$ at $\fb_0(Y)$ as a $b$-operator, since then $\kappa^{\frac12}G_{b}^{\frac12}\kappa^{\frac12}$ gives a corresponding inverse for $D_{\cC}$ at $\fb_0(Y)$,
$$
        D_{\cC}\kappa^{\frac12}G_{b}^{\frac12}\kappa^{\frac12}= \kappa^{-\frac12}(\kappa D_{\cC})\kappa^{\frac12}G_{b}^{\frac12}\kappa^{\frac12}= \kappa^{-\frac12}\Id \kappa^{\frac12}=\Id,
$$
and \eqref{le.24} is precisely $k^{-1}\kappa^{\frac12}G_b^{\frac12}\kappa^{\frac12}$ as expected.

  At $\sc(Y)$, $D_{\cC}$ can instead be inverted as a scattering operator.  From \eqref{le.20}, we see that its normal operator is 
\begin{equation}
   N_{\sc}(D_{\cC})= \eth_h +\gamma
\label{le.25}\end{equation}
where $\eth_h$ is the family of Euclidean Dirac operators of \eqref{dt.3} on  ${}^{\phi}NY= \left. {}^{\sc}T(Y\times[0,\infty])\right|_{Y\times\{\infty\}}$.  Since by assumption $\gamma\eth_h+ \eth_h\gamma=0$, we see that
\begin{equation}
     N_{\sc}(D_{\cC})^2=\eth_h^2+\Id
\label{le.26}\end{equation}
is clearly invertible as a family of suspended operators in the sense of \cite{Mazzeo-MelrosePhi}.  Thus, $N_{\sc}(D_{\cC})$ itself is invertible as a family of suspended operators with inverse 
\begin{equation}
 N_{\sc}(D_{\cC})^{-1}= N_{\sc}(D_{\cC}) (\eth_h^2+\Id)^{-1}.
\label{le.26b}\end{equation}
Hence, using this to invert $D_{\cC}$ at $\sc(Y)$, while at $\fb_0(Y)$ using instead $\kappa^{-\frac12}(\kappa D_{\cC})\kappa^{\frac12}$, we can construct a parametrix for $D_{\cC}$ as follows.  
\begin{lemma}
There exists $Q_{\cC}\in\Psi^{-1,\cQ}_{b,\sc}(Y;\ker D_v)$ and $R_{\cC}\in\Psi^{-\infty,\cR}_{b,\sc}(Y;\ker D_v)$ such that 
\begin{equation}
    D_{\cC}Q_{\cC}= \Id-R_{\cC}, \quad Q^*_{\cC}D_{\cC}=\Id-R^*_{\cC},
\label{le.27}\end{equation}
where $\cQ$ is an index family which is trivial at $\fb_{\infty}(Y), \lf_{\infty}(Y), \rf_{\infty}(Y)$, given by $\bbN_0$ at $\sc(Y)$ and such that
$$
    \inf\Re \cQ|_{\fb_0(Y)}\ge 1, \quad  \inf\Re\cQ_{\lf_0(Y)}\ge 1+\epsilon \quad\mbox{and} \quad \inf\Re\cQ_{\rf_0(Y)}\ge 1+\epsilon,
$$ 
while $\cR$ is an index family which is trivial at $\sc(Y)$, $\fb_{\infty}(Y)$, $\lf_{\infty}(Y)$, $\rf_{\infty}(Y)$, $\lf_0(Y)$, $\fb_0(Y)$ and such that $\inf \Re \cR|_{\rf_{0}(Y)}\ge 1+\epsilon$.  Moreover, at $\sc$, the restriction of $Q_{\cC}$ is given by $N_{\sc}(Q_{\cC})= N_{\sc}(D_{\cC})^{-1}$, while at $\fb_{0}(Y)$, the restriction of $\kappa^{-\frac12}Q_{\cC}\kappa^{-\frac12}$ is $G_{b}^{\frac12}$.
\label{le.27b}\end{lemma}
\begin{proof}
At the end $\kappa=\infty$, we can invert $D_{\cC}$ as a scattering operator as in \cite{Mazzeo-MelrosePhi} to obtain $Q_{\cC}$ near $\sc(Y),\fb_{\infty}(Y)$, $\lf_{\infty}(Y)$ and $\rf_{\infty}(Y)$.  At end $\kappa=0$, we can invert instead $\kappa^{-\frac12}(\kappa D_{\cC})\kappa^{\frac12}$as a $b$-operator as in \cite{MelroseAPS}, yielding a parametrix $\widetilde{Q}$ such that
$$
     \kappa^{-\frac12}(\kappa D_{\cC})\kappa^{\frac12}\widetilde{Q}=\Id-\widetilde{R},
$$
where $\widetilde{Q}$ is a $b$-operator of order one with polyhomogeneous expansion at $\lf_{0}(Y)$ and $\rf_0(Y)$ having leading term of order $\frac12+\epsilon$ and smooth at $\fb_0(Y)$ with restriction given by $G^{\frac12}_b$, while $\widetilde{R}$ is a $b$-operator of order $-\infty$ vanishing rapidly at $\lf_0(Y)$ and $\fb_0(Y)$ and with polyhomogeneous expansion at $\rf_{0}(Y)$ having leading term of order $\frac12+\epsilon$.  It suffices then to take $Q_{C}= \kappa^{\frac12}\widetilde{Q}\kappa^{\frac12}$ and $R_{\cC}= \kappa^{-\frac12}\widetilde{R}\kappa^{\frac12}$ near $\fb_0(Y)$.  This can be combined with the construction near $\sc(Y)$ to give the parametrix $Q_{\cC}$ globally as claimed.  
\end{proof}

In particular, if $u$ is a boundary defining function for $Y\times \{0\}$ in $Y\times [0,\infty]$, then since $$
\inf \Re \cR|_{\rf_0(Y)}\ge 1+\epsilon>\frac12,
$$ 
this parametrix shows that the operator $u^{\frac12}D_{\cC}u^{\frac12}$ has right and left parametrices $u^{-\frac12}Q_{\cC}u^{-\frac12}$ and $u^{-\frac12}Q_C^*u^{-\frac12}$ with compact error term,
$$
(u^{\frac12}D_{\cC}u^{\frac12})(u^{-\frac12}Q_{\cC}u^{-\frac12})= \Id-(u^{\frac12}R_{\cC}u^{-\frac12}), \quad (u^{-\frac12}Q^*_{\cC}u^{-\frac12})(u^{\frac12}D_{\cC}u^{\frac12})=\Id-(u^{-\frac12}R^*_{\cC}u^{\frac12}),
$$
implying that $D_{\cC}$ induces a Fredholm  operator
\begin{equation}
D_{\cC}: u^{\frac12}H^1_{b,\sc}(Y\times[0,\infty];\ker D_v)\to u^{-\frac12}L^2_b(Y\times[0,\infty];\ker D_v),
\label{le.28}\end{equation}
where  
\begin{multline}
H^1_{b,\sc}(Y\times [0,\infty];\ker D_v)= \{ f\in L^2_b(Y\times [0,\infty];\ker D_v) \; |  \\
  \xi f\in L^2_b(Y\times[0,\infty];\ker D_v) \; \forall \xi \in \cV_{b,\sc}(Y\times[0,\infty];\ker D_v)\}
\end{multline}
with $\cV_{b,\sc}(Y\times[0,\infty];\ker D_v)$ the Lie algebra of smooth vector fields on $Y\times [0,\infty]$ which are $b$-vector fields near $Y\times \{0\}$ and scattering vector fields near $Y\times \{\infty\}$.  The operator is also formally self-adjoint.  

\begin{lemma}
The Fredholm operator of \eqref{le.28} is a bijection.
\label{le.29}\end{lemma}
\begin{proof}
Since we assume that $\eth_{\phi}$ is a Dirac operator, we know from \eqref{dt.8b} and \eqref{dt.8e} that the operator $D_{\cC}$ takes the form
\begin{equation}
  D_{\cC}=-c\frac{\pa}{\pa\kappa}+ \frac{1}{\kappa}(\eth_Y+\frac{c}2)+\gamma.
\label{le.30}\end{equation}
Using that $\gamma$ anti-commutes with $-c\frac{\pa}{\pa\kappa}+ \frac{1}{\kappa}(\eth_Y+\frac{c}2),$ that $c^2=-\Id$ and that $c$ anti-commutes with $\eth_Y$, we compute that 
\begin{equation}
  \kappa^2D_{\cC}^2= -\left( \kappa\frac{\pa}{\pa \kappa} \right)^2+ 2\kappa\frac{\pa}{\pa \kappa}+ (\eth_Y^2+c\eth_Y-\frac{3}4)+ \kappa^2.
\label{le.31}\end{equation}
Hence, we see that 
\begin{equation}
\kappa^{-1}(\kappa^2D_{\cC}^2)\kappa= -\left( \kappa\frac{\pa}{\pa \kappa} \right)^2+ \left( \eth_Y^2+c\eth_Y+\frac14 \right)+ \kappa^2.
\label{le.32}\end{equation}
Setting $\widetilde{\eth}_Y:=c\eth_Y$, this becomes
\begin{equation}
   \kappa^{-1}(\kappa^2D_{\cC}^2)\kappa= -\left( \kappa\frac{\pa}{\pa \kappa} \right)^2+ \left(  \widetilde{\eth}_Y^2+ \widetilde{\eth}_Y+\frac14\right) + \kappa^2.
\label{le.33}\end{equation}
Now, $\widetilde{\eth}_Y^*=(c\eth_Y)^*= -\eth_Yc=c\eth_Y=\widetilde{\eth}_Y$, so $\widetilde{\eth}_Y$ is formally self-adjoint.  Moreover, 
\begin{equation}
   I(D_b,\lambda)= c(\lambda-\widetilde{\eth}_Y+\frac12),
\label{le.34}\end{equation}
so Assumption~\ref{le.4} implies that $\widetilde{\eth}_Y$ has no eigenvalue in the range $[-\frac12,\frac12]$.   Now, if $\lambda$ is an eigenvalue of $\widetilde{\eth}_Y$ with eigensection $\widetilde{\sigma}$, then $\widetilde{\sigma}$ is an eigensection of $\widetilde{\eth}_Y^2+ \widetilde{\eth}_Y+\frac14$ with eigenvalue $(\lambda+\frac12)^2$.  Since $\lambda\notin [-\frac12,\frac12]$, we have in particular that $(\lambda+\frac12)^2>0$.  Hence, decomposing \eqref{le.33} in terms of the eigenspaces of $\widetilde{\eth}_Y$, we obtain the modified Bessel equation
\begin{equation}
   \left( -\left(\kappa\frac{\pa}{\pa \kappa} \right)^2+ \alpha^2+\kappa^2  \right)f=0,  \quad \mbox{with}\; \alpha^2=(\lambda+\frac12)^2>0.
\label{le.35}\end{equation} 
A basis of solutions of this equation is given by the modified Bessel functions $K_{\alpha}$ and $I_{\alpha}$.  The function $I_{\alpha}$ grows exponentially as $\kappa\to \infty$ and tends to zero as $\kappa\searrow 0$, while $K_{\alpha}$ blows up like $\kappa^{-|\alpha|}$ as $\kappa\searrow 0$ and decays exponentially at infinity.  Thus, except for the trivial solution, no solution of \eqref{le.35} are in $L^2((0,\infty),\frac{d\kappa}{\kappa})$.  This means that the operator $\kappa^{-1}(\kappa^2D_{\cC}^2)\kappa$ has a trivial kernel in $L^2_b(Y\times [0,\infty];\ker D_v)$, hence that $D_{\cC}$ has a trivial kernel in $\kappa L^2_b(Y\times [0,\infty],\ker D_v)$.  A fortiori, $D_{\cC}$ thus has a trivial kernel in $uL^2_{b}(Y\times[0,\infty];\ker D_v)$.

Now, if $\sigma\in L^2_b(Y\times[0,\infty],\ker D_v)$ is such that $D_{\cC}\sigma=0$, then in fact $\sigma$ decays rapidly as $\kappa\to \infty$ by \cite{Mazzeo-MelrosePhi} and has a polyhomogeneous expansion at $\kappa=0$ by \cite{MelroseAPS}.  Furthermore, Assumption~\ref{le.4} and the fact that $I(\kappa D_{\cC},\lambda)=I(D_b,-\lambda)$ implies that $\sigma= \mathcal{O}(u^{1+\epsilon})$ near $Y\times \{0\}$.  This shows that the operator \eqref{le.28} is injective.  

To show that it is surjective, notice the operator $u^{\frac12}D_{\cC}u^{\frac12}$ is formally self-adjoint, which forces in particular the operator \eqref{le.28} to be surjective since $D_{\cC}$ has a trivial kernel in $u^{\frac12}L^2_b(Y\times[0,\infty];\ker D_v)\subset L^2_b(Y\times[0,\infty];\ker D_v)$.    

\end{proof}

Let us denote by $G_{\cC}$ the inverse of the bijective operator in \eqref{le.28}. 

\begin{lemma}
The inverse $G_{\cC}$ is an element of $\Psi^{-1,\cG}_{b,\sc}(Y;\ker D_v)$ with index family $\cG$ trivial at $\rf_{\infty}(Y)$, $\lf_{\infty}(Y)$ and $\fb_{\infty}(Y)$, given by $\bbN_0$ at $\sc(Y)$ and such that 
\begin{equation}
\inf\Re(\cG|_{\fb_0(Y)})\ge 1, \quad \inf \Re(\cG|_{\rf_0(Y)})\ge 1+\epsilon, \quad \inf\Re(\cG|_{\lf_0(Y)})\ge 1+\epsilon.
\label{le.38b}\end{equation}
Moreover, the restriction of $u^{-\frac12}G_{\cC}u^{-\frac12}$ at $\fb_0(Y)$ is precisely $G_{b}^{\frac12}$, while at $\sc(Y)$, we have instead $N_{\sc}(G_{\cC})=N_{\sc}(D_{\cC})^{-1}$.
\label{le.38}\end{lemma}
\begin{proof}
Using the parametrix of \eqref{le.27} and proceeding as in the proof of Corollary~\ref{dt.31}, we have that 
\begin{equation}
\begin{aligned}
G_{\cC}&= G_{\cC}\Id= G_{\cC}(D_{\cC}Q_{\cC}+R_{\cC})= Q_{\cC}+ G_{\cC}R_{\cC}, \\
G_{\cC}&= \Id G_{\cC}= (Q^*_{\cC}D_{\cC}+R_{\cC}^*)G_{\cC}= Q^*_{\cC}+R^*_{\cC}G_{\cC}.
\end{aligned}
\label{le.36}\end{equation}
Inserting the second equation in the first one thus yields
\begin{equation}
G_{\cC}= Q_{\cC}+ G_{\cC}R_{\cC}= Q_{\cC}+ Q^*_{\cC}R_{\cC}+ R^{*}_{\cC}G_{\cC}R_{\cC}.
\label{le.37}\end{equation}
Since $R_{\cC}$ and $R^*_{\cC}$ are very residual operators in the sense of \cite{MazzeoEdge}, we see by the semi-ideal property of such operators that $R^{*}_{\cC}G_{\cC}R_{\cC}$ is also semi-residual.  Hence, the result follows from \eqref{le.37}, Lemma~\ref{le.27b} and the composition formula of Theorem~\ref{bsc.4}.  
\end{proof}

In fact, near $\fb_0(Y)$, we can compare $G_{\cC}$ with $\kappa^{\frac12}G^{\frac12}_b\kappa^{\frac12}$ as follows,
\begin{equation}
  \kappa^{\frac12}G_{b}^{\frac12}\kappa^{\frac12}= \kappa^{\frac12}G_{b}^{\frac12}\kappa^{\frac12}D_{\cC}G_{\cC}= G_{\cC}+ \kappa^{\frac12}G^{\frac12}_b\kappa^{\frac12}\gamma G_{\cC},
  \label{le.38}\end{equation}
where we have used in the last step that $G^{\frac12}_{b}$ is the inverse of $\kappa^{-\frac12}\left( \kappa\left( -c\frac{\pa}{\pa\kappa}+ \frac{1}{\kappa}D_Y \right)   \right)\kappa^{\frac12}$.  Hence, from the composition formula of Theorem~\ref{bsc.4}, we see that the leading order term of $G_{\cC}$ at $\fb_0(Y)$ is at order $1$ with next term at least at order $2$.  That is,
\begin{equation}
G_{\cC}= \kappa^{\frac12}G_{b}^{\frac12}\kappa^{\frac12}+ \mathcal{O}(x^2_{\fb_0(Y)}).
\label{le.39}\end{equation}
This can be used to improve the parametrix $Q_1$ in \eqref{le.22} by removing the term of order $h+1$ at $\fbf_0$ of the error term.  
\begin{proposition}
There exists $Q_2\in \Psi^{-1,\cQ_2}_{k,\phi}(M;E)$ and $R_2\in\Psi^{-1,\cR_2}_{k,\phi}(M;E)$ such that 
\begin{equation}
   D_{k,\phi} Q_2=\Id -R_2,
\label{le.40b}\end{equation}
where $\cQ_2$ and $\cR_2$ are the empty set at $\fbf, \lf$ and $\rf$, while 
\begin{equation}
\begin{gathered}
   \cQ_2|_{\zf}= (\bbN_0-1)\cup \cN_2 \;\mbox{with}\; \inf\Re\cN_2>0, \quad \inf \Re\cQ_2|_{\fbf_0}\ge h, \quad  \inf\Re\cQ_2|_{\ff_0}\ge 0, \quad \inf\Re\cQ_2|_{\ff}\ge 0 \\
    \inf\Re \cQ_2|_{\lf_0}\ge \nu,  \quad \inf\Re \cQ_2|_{\rf_0}\ge h+1+\nu \quad \mbox{with} \quad 
   \nu:=\min\{\epsilon, \epsilon_1-1\}
\end{gathered}
\label{le.40c}\end{equation}
and 
\begin{equation}
\begin{gathered}
\inf\Re\cR_2|_{\zf}\ge 1, \quad \inf\Re\cR_2|_{\ff_0}>0, \quad \inf\Re\cR_2|_{\ff}\ge 0,  \quad \inf\Re \cR_2|_{\lf_0}\ge 1+\nu, \quad \inf\Re \cR_2|_{\rf_0}\ge h+1+\epsilon \\
\mbox{and} \quad  \inf\Re\cR_2|_{\fbf_0}>h+1. 
\end{gathered}
\label{le.40d}\end{equation}
\label{le.40}\end{proposition}
\begin{proof}
To solve \eqref{le.22b}, we can take $q_2 = k^{-1}G_{\cC}\Pi_h r_1$ seen as term of order $h$ at $\fbf_0$.   Letting $Q_2'$ be a smooth extension of $q_2$ off $\fbf_0$ corresponding to a term of order $h$ there, we can consider
$$
    Q_2= Q_1+ Q_2'.
$$
This ensures in particular that $R_2$ in \eqref{le.40b} is such that its term $r_2$ or order $h+1$ at $\fbf_0$ is such that $\Pi_h r_2=0$.  But extending $D_{v}^{-1}r_2$, seen as at term of order $h+1$, smoothly off $\fbf_0$ and adding it to $Q_2$, we can suppose that $R_2$ has no term order $h+1$ at $\fbf_0$, that is, $\inf\Re\cR_2|_{\fbf_0}>h+1$.    

Clearly then,  the term of order $h$ at $\fbf_0$ of $Q_2$ must be  the inverse of $D_{\cC}$, namely it is precisely $G_{\cC}$.  Moreover, the property \eqref{le.39} ensures that the new error term $R_2$ still vanishes to order $1$ at $\zf$, 
$$
      \inf\Re R_2|_{\zf}\ge 1.
$$ 
Finally, the extension of $q_2$ off $\fbf_0$ can be done using the `right' boundary defining function $\frac{x'}k$ near $\lf_0$.  Since $\Pi_h q_2=q_2$, this means that the part of the error term $R_2$ coming from the extension of $q_2$ will have leading order $1+\epsilon$ at $\lf_0$, so that $\inf\Re \cR_2|_{\lf_0}\ge 1+\nu$ as claimed.  On the other hand, $q_2$ has in principle a term of order $h+1+\epsilon\le h+2+\nu$ at $\rf_0$, hence the slight lost of decay at $\rf_0$. 
\end{proof}

\subsection*{Step 3: Inversion at $\ff$}

The parametrix $Q_2$ inverts $D_{k,\phi}$ at all boundary hypersurfaces of $M^2_{k,\phi}$ except at $\ff$.  There, the model to invert is
\begin{equation}
  N_{\ff}(D_{k,\phi})= D_v+ \eth_h+ \gamma k.
\label{le.41}\end{equation}
Using that $\gamma$ anti-commutes with $D_v+\eth_h$ and that $\eth_h$ anti-commute with $D_v$, we compute that 
$$
   N_{\ff}(D_{k,\phi})^2= D_v^2+ \eth_h^2+k^2.
$$
This is clearly invertible as a suspended operator for $k>0$ with inverse given by 
\begin{equation}
   N_{\ff}(D_{k,\phi})^{-1}= (D_v+\eth_h+\gamma k)(D_v^2+\eth_h^2+k^2)^{-1}.
\label{le.42}\end{equation}
To see that this matches our model as $k\searrow 0$, we should decompose the normal operator in terms of $\ker D_v$ and its orthogonal complement $(\ker D_v)^{\perp}$.  First, on $(\ker D_v)^{\perp}$, $N_{\ff}(D_{k,\phi})$ is still invertible as a suspended operator for $k=0$.  Lifting this inverse from $\ff\times \{0\}\subset M^2_{\phi}\times [0,\infty)$ to $M^2_{k,\phi}$ through the blow-down map $M^2_{k,\phi}\to M^2_{\phi}\times[0,\infty)$, this clearly corresponds to the part of $Q_2|_{\ff_0}$ acting on $(\ker D_v)^{\perp}$ on $\ff_0$, while it vanishes rapidly at $\lf, \rf, \fbf_0, \lf_0$ and $\rf_0$.  Hence, when acting on $(\ker D_v)^{\perp}$, the operator $Q_2$ can be naturally extended on $\ff$ by 
$$
\left((N_{\ff}(D_{k,\phi})|_{\ker D_v^{\perp}}\right)^{-1}.
$$

On the other hand, on $\ker D_v$, the matching of $N_{\ff}(D_{k,\phi})^{-1}$ with $Q_2$ is more in the spirit of \cite{GH1}, so we shall take the point of view offered by Lemma~\ref{kfb.9b} and work initially with $[M^2_{k,b};\Phi_+]$.  On this space, the face $\ff_{b,+}$ created by the blow-up of $\Phi_+$ corresponds to a blow-down version of $\ff$ in $M^2_{k,\phi}=[M^2_{k,b};\Phi_+,\Phi_0]$.  Because of this missing final blow-up, the model operator $N_{\ff_{b,+}}(D_{k,\phi})$ acting on $\ker D_v$ is not $\eth_h+ \gamma k$ in the limit $k\to 0+$, but instead 
$$
     k(\eth_h+\gamma).
$$
This is because near $k=0$, it is $\frac{x'}{k}$, not $x'$, which can be used as a boundary defining function for $\ff_{b,+}$.  The inverse is thus given by 
\begin{equation}
(\eth_h+\gamma)^{-1}k^{-1}= (\eth_h+\gamma)(\eth_h^2+\Id)^{-1}k^{-1}.
\label{le.43}\end{equation}
By \eqref{le.26b}, this is precisely matched by $Q_2|_{\fbf_0}$ acting on $\ker D_v$, the factor $k^{-1}$ in \eqref{le.43} indicating that $k^{-1}(\eth_h+\gamma)^{-1}$ yields a term of order $h$ instead of $h+1$ at $\fbf_0$.  

Now,  the face $\ff_0$ created by the blow-up of $\Phi_0$ is not really needed to invert the part of $D_{k,\phi}$ asymptotically acting on $\ker D_v$.  Indeed, when we are considering the action on $\ker D_v$, the operator $D_{\phi}$ becomes a scattering operator, and we can simply use the $b$-$\sc$ transition double space.  The inverse can then be lifted to $M^2_{k,\phi}$ via the blow-down map
$$
         M^2_{k,\phi}\to [M^2_{k,b};\Phi_+].
$$ 
This means that after we blow up $\Phi_0$ on $[M^2_{k,b};\Phi_+]$, we still have that the limit of $N_{\ff}(D_{k,\phi})^{-1}$ acting on $\ker D_v$ matches the term of order $h$ of $Q_2$ at $\fbf_0$, but also that it matches the part of $Q_2|_{\ff_0}$ acting on $\ker D_v$.  This yields the following improved parametrix.

\begin{proposition}
There exists $Q_3\in \Psi^{-1,\cQ_3}_{k,\phi}(M;E)$ and $R_3\in \Psi^{-1,\cR_3}_{k,\phi}(M;E)$ such that 
\begin{equation}
             D_{k,\phi}Q_3=\Id-R_3
\label{le.44b}\end{equation}
with $\cQ_3=\cQ_2$ and $\cR_3$ the same index family as $\cR_2$, except at $\ff$ where we have instead that 
$$
\inf\Re\cR_3|_{\ff}>0.
$$ 
\label{le.44}\end{proposition}
\begin{proof}
This follows from the previous discussion.
\end{proof}

\subsection*{Step 4: Inversion up to an error term of order $-\infty$ decaying rapidly as $k\searrow 0$ }

By the composition rules of Theorem~\ref{com.12} and thanks to the decay rate of $R_3$ at all boundary hypersurfaces, notice that there exists $\delta>0$ such that for $H$ a boundary hypersurface of $M^2_{k,\phi}$ distinct from $\zf$, 
\begin{equation}
    R_3= \cO(x_H^{\mu}) \; \mbox{at} \; H \; \Longrightarrow \; R_3^k= \cO(x_H^{\mu+k\delta}) \; \mbox{at} \; H\quad \forall k\in \bbN_0.  
\label{le.44c}\end{equation}
When $H=\zf$, then \eqref{le.44c} holds provided $\epsilon+\epsilon_1> 1$.  However, for $\epsilon+\epsilon_1\le 1$,  it is not quite true at $\zf$, since $R_3=\cO(x_{\zf})$ there, but the lack of decay at $\lf_0$ and $\rf_0$ only ensures that $R^2_3=\cO(x_{\zf}^{\epsilon+\epsilon_1})$ at $\zf$.  Still, \eqref{le.44c} still holds for $H=\zf$ provided we take $\mu=\epsilon$ instead of $\mu=1$.    Since $R_3$ is a pseudodifferential operator of order $-1$, this means we can make sense of the formal sum
$$
      \sum_{j=1}^{\infty} R_3^{j}
$$
as an asymptotic sum, both symbolically and in terms of polyhomogeneous expansions at the various boundary hypersurfaces.  If $S$ is such an asymptotic sum, 
$$
   S\sim \sum_{j=1}^{\infty} R_3^j,
$$
then $S\in \Psi^{-1,\cS}_{k,\phi}(M;E)$ with $\cS$ satisfying the same lower bounds as $\cR_3$, except at $\zf$ when $\epsilon+\epsilon_1\le 1$, where we have instead $0<\inf\Re \cS|_{\zf}\le 1$ in that case.  

Then, by construction, 
$$
  R_4:= \Id-(\Id-R_3)(\Id+S) \in\Psi^{-\infty}_{k,\phi}(M;E) 
$$
has Schwartz kernel decaying rapidly at all boundary hypersurfaces of $M^2_{k,\phi}$.  Hence, setting $Q_4= Q_3(\Id+S)$, we have that
\begin{equation}
   D_{k,\phi}Q_4= \Id-R_4
\label{le.45}\end{equation}
with $Q_4\in \Psi^{-1,\cQ_4}_{k,\phi}(M;E)$, where $\cQ_4$ is an index family having the same lower bound as $\cQ_3$, except at $\rf_0$ where we have instead $\inf\Re\cQ_4|_{\rf_0}\ge h+\epsilon$.  Furthermore, by Theorem~\ref{com.12}, if $\epsilon+\epsilon_1>1$, then  $\inf\Re\cS|_{\zf}\ge 1$, $\inf\Re \cS|_{\lf_0}\ge 1+\nu$ and $1+2\nu>0$, so $\cQ_4|_{\zf}= (\bbN_0-1)\cup\cN_4$ with $\inf\Re \cN_4>0$.   

The error term $R_4$ can be seen as a smooth family of operators  $R_4(k)\in\dot{\Psi}^{-\infty}(M;E)$ parametrized by $k\in[0,\infty)$ and approaching rapidly $0$ as $k\searrow 0$.  In particular, the operator $R_4(k)$ 
has a small operator norm for $k$ small.  

\subsection*{Step 5:  Completion of the proof of Theorem~\ref{le.46}}
\begin{proof}
Since the operator norm of $R_4(k)$ tends to zero when $k\searrow 0$, there exists $\delta>0$ such that $\Id-R_4(k)$ is invertible with inverse given by $\Id+S_4(k)$, where 
$$
    S_4(k)= \sum_{j=1}^{\infty} R_4(k)^j \quad \mbox{for}\; k\in [0,\delta)
$$
is a smooth family of operators in $\dot{\Psi}^{-\infty}(M;E)$ decaying rapidly to zero when $k\searrow 0$.  Hence, for $k\in [0,\delta)$, we can finally find a right inverse
\begin{equation}
  G_{k,\phi}:= Q_4(\Id+S_4)  \quad \Longrightarrow \quad D_{k,\phi}G_{k,\phi}=\Id.
\label{le.45}\end{equation}
 For $k\ge \delta$, we can invert $D_{k,\phi}$ simply in the small $\phi$-calculus as in \cite{Mazzeo-MelrosePhi}, that is, with inverse in $\Psi^{-1}_{\phi}(M;E)$.  Hence, $G_{k,\phi}$ can be extended to $k\in [\delta,\infty)$ to give a right inverse for all $k\ge 0$.

By the composition rules of Theorem~\ref{com.14},  $G_{k,\phi}$ is a $k,\phi$-operator and  \eqref{le.46b} holds, except possibly at $\rf_0$ and $\ff$ where we can only conclude for the moment that $\inf\Re \cG|_{\rf_0}\ge h+\epsilon$ and $\inf\Re \cG|_{\ff}\ge 0$.  Hence, it remains to prove that $G_{k,\phi}D_{k,\phi}=\Id$ and that in fact $\inf\Re \cG|_{\rf_0}\ge h+1+\nu$ and $\cG|_{\ff}=\bbN_0$.  To see this, take the adjoint of $D_{k,\phi}G_{k,\phi}=\Id$,
$$
     G_{k,\phi}^*D_{k,\phi}=\Id.
$$
It suffices then to notice that 
$$
    G_{k,\phi}^*= G_{k,\phi}^*(D_{k,\phi}G_{k,\phi})= G_{k,\phi}.
$$
In particular, $G_{k,\phi}$ is self-adjoint as expected and $\inf\Re\cG|_{\lf_0}\ge \nu\; \Longrightarrow \;\inf\Re\cG|_{\rf_0}\ge h+1+\nu$.  On the other hand, for $k>0$ we know from the parametrix construction of \cite{Mazzeo-MelrosePhi} in the small $\phi$-calculus that the expansion at $\ff$ of $G_{k,\phi}$ must be smooth.  By continuity, this means that this is still the case in the limit $k\searrow 0$, so that $\cG|_{\ff}=\bbN_0$ as claimed.  
\end{proof}

Composing $G_{k,\phi}$ with itself  also gives a description of the inverse of $D_{k,\phi}^2=D_{\phi}^2+k^2$.
\begin{corollary}
Let $\eth_{\phi}\in\Diff^1_{\phi}(M;E)$ be a Dirac operator satisfying Assumption~\ref{le.4}.  Then there exists an operator $G_{k,\phi}^2\in \Psi^{-2,\cG_2}_{k,\phi}(M;E)$ such that 
$$
          (D_{\phi}^2+k^2)G_{k,\phi}^2= G_{k,\phi}^2 (D_{\phi}^2+k^2)= \Id,
$$
where $\cG_2$ is an index family given by the empty set at $\lf, \rf$ and $\fbf$, and such that 
\begin{equation}
\begin{gathered}
  \inf\Re \cG_2|_{\zf}\ge-2,  \quad \inf \Re\cG_2|_{\fbf_0}\ge h-1, \quad   \cG_2|_{\ff}=\bbN_0, \quad \inf\Re \cG_2|_{\ff_0}\ge \left\{ \begin{array}{ll} 0, & h>1,  \\ (0,1), & h=1,  \end{array} \right.  \\
   \mbox{and} \quad \inf\Re \cG_2|_{\lf_0}\ge (\nu-1,1)\in \bbR\times \bbN_0,  \quad \inf\Re \cG_2|_{\rf_0}\ge (h-1+\nu,1)\in \bbR\times \bbN_0.
\end{gathered}
\label{le.47b}\end{equation}\label{le.47}\end{corollary}
\begin{proof}
To apply Theorem~\ref{le.46}, we need to find a self-adjoint operator $\gamma\in\CI(M;\End(E))$ as in \eqref{le.1}.  One way to proceed is to consider, instead of $\eth_{\phi}$ the operator $\widetilde{\eth}_{\phi}\in\Diff^1_{\phi}(M;E\oplus E)$ given by
\begin{equation}
  \widetilde{\eth}_{\phi}= \left( \begin{array}{cc} \eth_{\phi} & 0 \\ 0 & -\eth_{\phi} \end{array} \right),
\label{le.47c}\end{equation}    
since then one can consider the self-adjoint operator $\widetilde{\gamma}\in \CI(M;\End(E\oplus E))$ given by
\begin{equation}
   \widetilde{\gamma}= \left( \begin{array}{cc} 0  &  -\sqrt{-1}  \\ \sqrt{-1} & 0 \end{array} \right).
\label{le.47d}\end{equation}
The operator  $\widetilde{\gamma}$ is such that $\widetilde{\gamma}^2=\Id_{E\oplus E}$ and $\widetilde{\eth}_{\phi}\widetilde{\gamma}+ \widetilde{\gamma}\widetilde{\eth}_{\phi}=0$.  Setting $\widetilde{D}_{\phi}=x^{-\frac{h+1}2}\widetilde{\eth}_{\phi}x^{\frac{h+1}2}$, we can thus apply Theorem~\ref{com.12} to find $\widetilde{G}_{k,\phi}\in \Psi^{-1,\cG}_{k,\phi}(M;E\oplus E)$ such that
$$
           (\widetilde{D}_{\phi}+ k\widetilde{\gamma})\widetilde{G}_{k,\phi}= \widetilde{G}_{k,\phi}(\widetilde{D}_{\phi}+ k\widetilde{\gamma})= \Id_{E\oplus E}
$$
with index family $\cG$ as in \eqref{le.46b}.  Composing $\widetilde{G}_{k,\phi}$ with itself and applying Theorem~\ref{com.12} gives us an operator $\widetilde{G}^2_{k,\phi}\in\Psi^{-2,\cG_2}_{k,\phi}(M;E\oplus E)$ with index family $\cG_2$ as in \eqref{le.47b} such that 
$$
   (\widetilde{D}_{\phi}^2+ k^2)\widetilde{G}^2_{k,\phi}= \widetilde{G}^2_{k,\phi}(\widetilde{D}_{\phi}^2+ k^2)= \Id_{E\oplus E}
$$
If $P_1:E\oplus E\to E$ is the bundle projection on the first factor, it suffices then to take $G^2_{k,\phi}=P_1\widetilde{G}^2_{k,\phi}P_1$. 
\end{proof}
\begin{remark}
In terms of $\eth_{\phi}^2$, this means that $x^{\frac{h+1}2}G^2_{k,\phi}x^{-\frac{h+1}2}$ is such that
$$
    (\eth_{\phi}^2+k^2)(x^{\frac{h+1}2}G^2_{k,\phi}x^{-\frac{h+1}2})=(x^{\frac{h+1}2}G^2_{k,\phi}x^{-\frac{h+1}2})(\eth_{\phi}^2+k^2)=\Id.
$$
\end{remark}

Thanks to Example~\ref{HdR.1}, this can be applied in particular to Hodge Laplacian of a fibered boundary metric.

\begin{corollary}
Let $\eth_{\phi}$ be the Hodge-deRham operator associated to a fibered boundary metric $g_{\phi}$ product-type up to order $2$.  Suppose that the exterior differential $d^{\ker D_v}$ and its formal adjoint $\delta^{\ker D_v}$ acting on sections of the flat vector bundle $\ker D_v\to Y$ in Lemma~\ref{do.8c} are such that the de Rham cohomology groups
\begin{gather}
  \label{le.49a}     H^q(Y;\ker D_v)=\{0\} \quad \mbox{for} \quad q\in\{ \frac{h-1}2, \frac{h}2, \frac{h+1}2\}, \\
   \label{le.49b}  \Spec(d^{\ker D_v}\delta^{\ker D_v}+  \delta^{\ker D_v} d^{\ker D_v})_{\frac{h}2}>\frac{3}4, \\
   \label{le.49c}  \Spec (d^{\ker D_v}\delta^{\ker D_v})_{\frac{h+1}2}>1.
\end{gather} 
Then there exists  an operator $G_{k,\phi}^2\in \Psi^{-2,\cG_2}_{k,\phi}(M;\Lambda^*({}^{\phi}T^*M))$ such that 
$$
          (D_{\phi}^2+k^2)G_{k,\phi}^2= G_{k,\phi}^2 (D_{\phi}^2+k^2)= \Id,
$$
where $\cG_2$ is an index family given by the empty set at $\lf, \rf$ and $\fbf$, and such that,  
\begin{equation}
\begin{gathered}
  \inf\Re \cG_2|_{\zf}\ge-2,  \quad \inf \Re\cG_2|_{\fbf_0}\ge h-1, \quad   \cG_2|_{\ff}=\bbN_0, \quad \inf\Re \cG_2|_{\ff_0}\ge  \left\{ \begin{array}{ll} 0, & h>1,  \\ (0,1), & h=1,  \end{array} \right.  \\
   \mbox{and} \quad \inf\Re \cG_2|_{\lf_0}\ge (\nu-1,1)\in \bbR\times \bbN_0,  \quad \inf\Re \cG_2|_{\rf_0}\ge (h-1+\nu,1)\in \bbR\times \bbN_0,
\end{gathered}
\label{le.47e}\end{equation}
where $\nu=\min\{\epsilon,\epsilon_1-1\}$.
\label{le.48}\end{corollary}
\begin{proof}
We need to show that the indicial family $I(D_b,\lambda)$ has no indicial root in the interval $[-1,0]$.  By Lemma~\ref{do.8c}, this will be the case provided the de Rham cohomology groups
\begin{gather}
  \label{le.50a}     H^q(Y;\ker D_v)=\{0\} \quad\mbox{for} \quad q\in\{ \frac{h-1}2, \frac{h}2, \frac{h+1}2\}, \\
  \label{le.50b}  \left( \Spec(d^{\ker D_v}\delta^{\ker D_v}+  \delta^{\ker D_v} d^{\ker D_v})_{\frac{h}2}\setminus\{0\}\right)>\frac{3}4, \\
 \label{le.50c}    \left(\Spec (d^{\ker D_v}\delta^{\ker D_v})_{\frac{h+1}2}\setminus\{0\}\right)>1, \\
  \label{le.50d}   \left(\Spec (\delta^{\ker D_v} d^{\ker D_v})_{\frac{h-1}2}\setminus\{0\}\right)>1, \\
 \label{le.50e}    \left(\Spec (d^{\ker D_v}\delta^{\ker D_v})_{\frac{h+2}2}\setminus\{0\}\right)>\frac{3}4, \\
  \label{le.50f}    \left(\Spec (\delta^{\ker D_v} d^{\ker D_v})_{\frac{h-2}2}\setminus\{0\}\right)>\frac34.
\end{gather}

Clearly, \eqref{le.49a} and \eqref{le.49b} corresponds to \eqref{le.50a} and \eqref{le.50b}.  On the other hand, by the symmetry of the positive spectrum of the Hodge Laplacian, \eqref{le.49b} also implies  \eqref{le.50e} and \eqref{le.50f}, while \eqref{le.49c} implies \eqref{le.50c} and \eqref{le.50d}.

\end{proof}

When $Y=\pa M$ and $\phi: \pa M\to Y$ is the identity, we see, taking into account the different conventions for densities to define pseudodifferential operators, that Corollary~\ref{le.48} gives back \cite[Theorem~1]{GS}, but on a double-space with with one extra face, namely $\ff_0$.  In our parametrix construction however, the face $\ff_0$ is not required when $\phi$ is the identity map, so our parametrix does indeed descend to the $b$-$\sc$ transition double space of \cite{GH1,Kottke} as in \cite{GS}.  

On the other hand, with respect to \cite[Theorem~1]{GS}, our hypothesis is slightly less restrictive.
 Indeed, first, in the terminology of \cite{GS}, we are allowing an asymptotically conic metric to order 2 instead of $3$.  Second, the assumption \cite[(2)]{GS}, namely
\begin{equation}
     \ker_{x^{-1}L^2_b}(D^2_{\phi})=\ker_{L^2_b}(D^2_{\phi})
\label{le.51}\end{equation}
in our notation, implies in particular that 
\begin{equation}
  \ker_{x^{-1}L^2_b}D_{\phi}= \ker_{L^2_b}D_{\phi}.
\label{le.52}\end{equation}
By the relative index theorem of \cite[Theorem~6.5]{MelroseAPS}  and  the symmetry of the critical weights of $I(D_b,\lambda)$ around $-\frac12$, we can infer from \eqref{le.52} that 
\begin{equation}
  (-1,0)\cap \Crit(D_b)=\emptyset.
\label{le.53}\end{equation}
By Lemma~\ref{do.8c}, the condition \eqref{le.53} implies \eqref{le.49a}, \eqref{le.49b} and \eqref{le.49c}, but the last two with only non-strict inequalities.  However, using the symmetries of the positive spectrum of the Hodge Laplacian, the first condition of \cite[(4)]{GS}, namely
\begin{equation}
  \left| q-\frac{h+1}2 \right|\le \frac12\; \Longrightarrow \; 1-\left(\frac{h+1}2-q  \right)^2\notin \Spec(d^{\ker D_v}\delta^{\ker D_v})_q
\label{le.54}\end{equation}
in our notation, precisely rules out the equality case in the \eqref{le.49b} and \eqref{le.49c} with non-strict inequalities.  Thus, conditions \cite[(2),(4)]{GS} implies our conditions \eqref{le.49a}, \eqref{le.49b} and \eqref{le.49c}.  Conversely, if $h$ is even, notice that \eqref{le.49a} and \eqref{le.49b} implies \cite[(2),(4)]{GS} by \cite[Lemma~27]{GS}.  If instead $h$ is odd, then at least we see that \eqref{le.49a} and the stronger version of \eqref{le.49c}
$$
\Spec (d^{\ker D_v}\delta^{\ker D_v}+ \delta^{\ker D_v}d^{\ker D_v} )_{\frac{h+1}2}>1
$$
imply \cite[(2),(4)]{GS} by \cite[Remark~28]{GS}.

\section{The inverse of a non-fully elliptic supended Dirac $\phi$-operator} \label{sus.0}

In this final section, let us come back to our original motivation for studying the low energy limit of the resolvent of a Dirac $\phi$-operator.  Thus, on $M\times \bbR^q$, let 
\begin{equation}
    \eth_{\sus}= \eth_{\phi} + \eth_{\bbR^q}
\label{sus.1}\end{equation}
be a $\bbR^q$-suspended Dirac $\phi$-operator, where $\eth_{\phi}$ is a Dirac $\phi$-operator associated to a fibered boundary metric $g_{\phi}$ and a Clifford module $E\to M$ as in \S~\ref{le.0}, and where $\eth_{\bbR^q}$ is a family of Euclidean Dirac operators on $\bbR^q$ parametrized by $M$ and anti-commuting with $\eth_{\phi}$.  If $\{e_1,\ldots,e_q\}$ is the canonical basis of $\bbR^q$, then 
$$
    \eth_{\bbR^q}= \sum_{j=1}^q \cl(e_j) \nabla_{e_j}
$$  
with $\nabla$ the pull-back of the Clifford connection of $E$ to its pull-back on $M\times \bbR^q$ and $\cl(e_j)$ denotes Clifford multiplication by $e_j$.  Thus, we suppose that the Clifford module structure of $E$ lifts to a Clifford module on its pull-back on $M\times \bbR^q$ for the Clifford bundle associated to the product metric
$$
     g_{\phi}+ g_{\bbR^q}
$$
on $M\times \bbR^q$, where $g_{\bbR^q}$ is the canonical Euclidean metric on $\bbR^q$.  Taking the Fourier transform in $\bbR^q$, we obtain a family of operators
\begin{equation}
  \widehat{\eth}_{\sus}(\xi)= \eth_{\phi}+ i\cl(\xi),  \quad \xi\in\bbR^q.
\label{sus.2}\end{equation}
As noted in the introduction, for $\xi\ne 0$, this can be rewritten 
\begin{equation}
\widehat{\eth}_{\sus}= \eth_{\phi}+ k\gamma \quad \mbox{with} \; k=|\xi|,  \quad \gamma=\frac{i}{|\xi|} \cl(\xi). 
\label{sus.3}\end{equation}
Conjugating by $x^{\frac{h+1}2}$, we get the corresponding operators
\begin{equation}
  D_{\sus}= x^{-\frac{h+1}2}\eth_{\sus}x^{\frac{h+1}2}= D_{\phi}+ \eth_{\bbR^q} \quad \mbox{and} \quad \widehat{D}_{\sus}(\xi)= D_{\phi}+ i\cl(\xi)
\label{sus.3}\end{equation}
with $\xi\in\bbR^q$ and $D_{\phi}= x^{-\frac{h+1}2}\eth_{\phi}x^{\frac{h+1}2}$ as in \eqref{fdt.1}.  

We will suppose that Assumption~\ref{le.4} holds for the operator $\eth_{\phi}$.  In this case, using \eqref{sus.3}, we know by Theorem~\ref{le.46} that the inverse $G_{\xi,\phi}$ of the operator $\widehat{D}_{\sus}(\xi)$ admits a pseudodifferential description all the way down to $\xi=0$.  Hence, taking the inverse Fourier transform of the inverse $G_{\xi,\phi}$ will give a corresponding inverse for $D_{\sus}$.  The detailed description of $G_{\xi,\phi}$ in the limit $\xi\to 0$ will allow us to give a pseudifferential characterization of the inverse of $D_{\sus}$.  First, recall for instance from \cite[Lemma~6.2]{DLR} that the small calculus of $\bbR^q$-suspended $\phi$-operators acting on sections of $E$ is the union over $m\in\bbR$ of the spaces 
\begin{multline}
\Psi^m_{\phi-\sus(\bbR^q)}(M;E):= \{ \kappa\in I^m(M^2_{\phi}\times \overline{\bbR^q}, \Delta_{\phi}\times \{0\};
  \pr_1^*(\Hom_{\phi}(E,E)\otimes {}^{\phi}\Omega_R(M)))\cdot \pr_2^*(\Omega_{\bbR^q})\; | \\ \kappa\equiv 0 \; \mbox{at}\; \pa(M^2_{\phi}\times \overline{\bbR^q})\setminus (\ff\times \overline{\bbR^q}) \},
\label{sus.4}\end{multline}
where $\overline{\bbR^q}$ is the radial compactification of $\bbR^q$,  $\pr_1: M^2_{\phi}\times \overline{\bbR^q}\to M^2_{\phi}$ and $\pr_2: M^2_{\phi}\times \overline{\bbR^q}\to \overline{\bbR^q} $ are the projections on the first and second factors and $\Omega_{\bbR^q}$ is the density of the Euclidean metric on $\bbR^q$.  However, because of the lack of decay and the lack of smoothness, the inverse Fourier transform of $G_{\xi,\phi}$ will not quite be an element of $\Psi^{-2}_{\phi-\sus(\bbR^q)}(M;E)$.  We need in fact to slightly modify this space of operators if we want to include the inverse Fourier transform of $G_{\xi,\phi}$.  To describe this space, let $\rho$ be a total boundary defining function for the $b$-double space $M^2_b$.  Let $V_{\rho}=M^2_b\times \rho\bbR^q$ denote  the vector bundle of rank $q$ over $M^2_b$ trivialized by the sections $\rho e_1,\ldots, \rho e_q$.  As sections of $V_\rho$, these sections are not vanishing on $\pa M^2_b$, though of course they do vanish as sections of $M^2_b\times \bbR^q\to M^2_b$.  Let $\overline{V_\rho}= M^2_b\times \overline{\rho\bbR^q}$ denote the fiberwise radial compactification of the fiber bundle $V_\rho$.  The double space needed to describe the Schwartz kernels of our class of   
operators is obtained by blowing up the $p$-submanifold $\Phi\times \{0\}\subset M^2_b\times \overline{\rho\bbR^q}$, that is, the zero section of $V_\rho|_{\Phi}$, where $\Phi\subset M^2_b$ is the $p$-submanifold of \eqref{phi.11},
\begin{equation}
  \widetilde{M}^2_{\phi-\sus(V_\rho)}= [M^2_b\times \overline{\rho\bbR^q}; \Phi\times \{0\}].
\label{sus.4b}\end{equation} 
Denote by $\ff$ the new boundary hypersurface created by this blow-up.  Let us also denote by $\bbS(V_\rho)$, $\fbf$, $\lf$ and $\rf$ the boundary hypersurfaces of $\widetilde{M}^2_{\phi-\sus(V_{\rho})}$ corresponding to the lifts of $M^2_b\times \pa(\overline{\rho\bbR^q})$, $\fb\times \overline{\rho\bbR^q}$, $\lf\times \overline{\rho\bbR^q}$ and $\rf\times \overline{\rho\bbR^q}$.  Because of the blow-up of $\Phi\times\{0\}$, notice that the space of suspended operators \eqref{sus.4} can alternatively be defined by 
\begin{multline}
\Psi^m_{\phi-\sus(\bbR^q)}(M;E)= \{ \kappa\in I^m(\widetilde{M}^2_{\phi-\sus(V_\rho)}, \Delta_{\phi,\sus};
  \widetilde{\pr}_1^*(\Hom_{b}(E,E)\otimes \beta_b^*\pr_R^*{}^{\phi}\Omega(M)))\cdot\widetilde{\pr}_2^*(\rho^{-q}\Omega_{\rho\bbR^q}) \; | \\ \kappa\equiv 0 \; \mbox{at}\; \pa(\widetilde{M}^2_{\phi-\sus(V_\rho)})\setminus \ff \},
\label{sus.4c}\end{multline}
where $\Delta_{\phi,\sus}$ is the lift of $\Delta_b\times\{0\} \subset M^2_b\times \overline{\rho\bbR^q}$ to $\widetilde{M}^2_{\phi-\sus(V_\rho)}$ with $\Delta_b$ the $b$-diagonal in $M^2_b$, 
$$
\widetilde{\pr}_1: \widetilde{M}^2_{\phi-\sus(V_{\rho})}\to M^2_b \quad \mbox{and} \quad \widetilde{\pr}_2: \widetilde{M}^2_{\phi-\sus(V_{\rho})}\to \overline{\rho\bbR^q}
$$ 
are the natural map induced by the blow-down map and the natural projections $\overline{V_\rho}\to M^2_b$ and $\overline{V_{\rho}}\to \overline{\rho\bbR^q}$, while  $\Omega_{\rho\bbR^q}=\rho^{q}\Omega_{\bbR^q}$ is the natural Euclidean density on $\rho\bbR^q$ and
$$
\Hom_b(E,E)=\beta_b^*(\pr_L^*E\otimes \pr_R^*E^*)
$$ 
with $\pr_L:M^2\to M$ and $\pr_R:M^2\to M$ the projections on the left and right factors.

If $\cE$ is an index family associated to the manifold with corners $\widetilde{M}^2_{\phi-\sus(V_\rho)}$, one can more generally consider the spaces
\begin{equation}
\begin{aligned}
\Psi^{-\infty,\cE}_{\phi-\sus(\bbR^q)}(M;E)&:=\cA^{\cE}_{\phg}(\widetilde{M}^2_{\phi-\sus(V_{\rho})}; \widetilde{\pr}_1^*(\Hom_{b}(E,E)\otimes \beta_b^*\pr_R^*{}^{\phi}\Omega(M))\cdot\widetilde{\pr}_2^*(\rho^{-q}\Omega_{\rho\bbR^q})),  \\
\Psi^{m,\cE}_{\phi-\sus(\bbR^q)}(M;E)&:= \Psi^m_{\phi-\sus(\bbR^q)}(M;E)+\Psi^{-\infty,\cE}_{\phi-\sus(\bbR^q)}(M;E), \quad m\in\bbR.
\end{aligned}
\label{sus.5}\end{equation}

\begin{theorem}
Suppose that Assumption~\ref{le.4} holds for $\eth_{\phi}$ and that $h=\dim Y>1$.  Then the inverse $D^{-1}_{\sus}$ of $D_{\sus}$, for instance seen as acting from its minimal domain onto the $L^2$-space of sections of $E$ with respect to the metric $g_{b}+g_{\bbR^q}$ with $g_b$ a $b$-metric on $M$, is an element of $\Psi^{-1,\check{\cG}}_{\phi-\sus(\bbR^q)}(M;E)$ for an index family $\check{\cG}$ such that 
\begin{equation}
\begin{gathered}
\inf\Re\check{\cG}|_{\bbS(V_{\rho})}\ge q-1, \quad \inf\Re\check{\cG}|_{\fbf}\ge h+q, \quad \inf\Re\check{\cG}|_{\ff}\ge 0,  \\
\inf\Re\check{\cG}|_{\lf}\ge \nu+q, \quad \inf\Re\check{\cG}|_{\rf}\ge h+q+1+\nu, \quad \mbox{with} \; \nu=\min\{ \epsilon,\epsilon_1-1\}.
\end{gathered}
\label{sus.6a}\end{equation}
Furthermore, if $\epsilon+\epsilon_1>1$ for $\epsilon$ and $\epsilon_1$ as in Assumption~\ref{le.4}, then 
$$
    \check{\cG}|_{\bbS(V_\rho)}= (q-1+\bbN_0)\cup (\cN+q)
$$
with $\cN$ an index set such that $\Re\cN>0$.
\label{sus.6}\end{theorem}
\begin{proof}
Notice first that performing a standard symbolic inversion as in the proof of Proposition~\ref{dt.11}, there exists 
$Q\in \Psi^{-1}_{\phi-\sus(\bbR^q)}(M;E)$ such that 
$$
         D_{\sus}Q=\Id +R,  \quad R\in\Psi^{-\infty}_{\phi-\sus(\bbR^q)}(M;E).  
$$
Hence, taking its Fourier transform $\widehat{Q}(\xi)$ in the factor $\bbR^q$ gives for each $\eta\in\bbS^{q-1}$ an operator
$\widehat{Q}(k\eta)$ in $\Psi^{-1}_{k,\phi}(M;E)$ such that for $\gamma=i\cl(\eta)$,  
$$
     (D_{\phi}+k\gamma)\widehat{Q}(k\eta)= \Id + \widehat{R}(k\eta), \quad \widehat{R}(k\eta)\in\Psi^{-\infty}_{k,\phi}(M;E).
$$
In particular, this shows that
$$
\begin{aligned}
(D_{\phi}+k\gamma)^{-1}&= (D_{\phi}+k\gamma)^{-1}((D_{\phi}+k\gamma)\widehat{Q}(k\eta)-\widehat{R}(k\eta) ) \\
              &= \widehat{Q}(k\eta)+ (D_{\phi}+k\gamma)^{-1}\widehat{R}(k\eta).
\end{aligned}
$$
Since the inverse Fourier transform of the first term on the right hand side is already in the desired space, it suffices to concentrate on the second term.  By Theorem~\ref{le.46} and Theorem~\ref{com.12}, notice that
$$
   \widehat{R}_2(k\eta):= (D_{\phi}+k\gamma)^{-1}\widehat{R}(k\eta)\in\Psi^{-\infty,\cG}_{k,\phi}(M;E).  
$$
Now, near $\zf$, but away from the other boundary hypersurfaces, the inverse Fourier transform converts the polyhomogeneous expansion at $\zf$ into a polyhomogeneous expansion at $\bbS(V_\rho)$ with term of order $k^\ell=|\xi|^{\ell}$ at $\zf$ being converted into a term of order $\rho_{\infty}^{q+\ell}$ at $\bbS(V_\rho)$, where $\rho_{\infty}$ denotes a boundary defining function for $\bbS(V_\rho)$.  In particular, the term of order $-1$ at $\zf$ corresponds to a term of order $\rho^{q-1}$ at $\bbS(V_\rho)$ given by the pull-back of 
$$
G^{-1}_{\zf}=\Pi_{\ker_{L^2_b}D_{\phi}} 
$$
on $M^2_{b}$ to $\bbS(V_\rho)$.  Near $\fbf_0$, $\lf_0$ and $\rf_0$, but away from $\ff_0$,  we can take advantage of  the rapid decay of $\widehat{R}_2(k\eta)$ at $\fbf$, $\lf$ and $\rf$ to make the change of variable  
\begin{equation}
               \widetilde{\xi}= \frac{\xi}{\rho},  \quad \widetilde{x}= \rho x
\label{sus.6c}\end{equation}
in the inverse Fourier transform of $\widehat{R}_2(\xi)$, so that 
$$
      \left(\frac{1}{(2\pi)^q}\int_{\bbR^q}  e^{ix\cdot\xi} \widehat{R}_2(\xi) d\xi\right)dx=  \left(\frac{1}{(2\pi)^q} \int_{\rho^{-1}\bbR^q} e^{i\widetilde{x}\cdot \widetilde{\xi}} \widehat{R}_2(\rho\widetilde{\xi})  d\widetilde{\xi}\right)d\widetilde{x}
$$ 
with $\widetilde{x}$ the natural variable on the fibers of $V_\rho= M^2_b\times \rho\bbR^q$ (so that $d\widetilde{x}=\Omega_{\rho\bbR^q}$).  In particular, the inverse Fourier transform will have the claimed behavior away from the lift of $V_\rho|_{\Phi}\subset V_\rho$ on $\widetilde{M}^2_{\phi-\sus(V_{\rho})}$.  

Hence, the only problematic region left to consider is near $\ff_0$ and $\ff$ in $M^2_{\phi,k}$.  To describe the inverse Fourier transform near this region, we will first provide more details on the expansion of $(D_{\phi}+k\gamma)^{-1}$ at $\ff$ and $\ff_0$.  Let $\rho_{\ff}$ and $\rho_{\ff_0}$ be boundary defining functions for the boundary hypersurfaces $\ff$ and $\ff_0$ in $M^2_{k,\phi}$.  Then the expansion of $(D_{\phi}+k\gamma)^{-1}$ at $\ff$ in powers of $\rho_{\ff}$ makes in principle the Fourier transform hard to compute, since in local coordinates, $\rho_{\ff}=\frac{x'}{k}$, yielding a singular expansion in $k$ as $k\searrow 0$.  However, as we will now show, the expansion at $\ff$ of $\widehat{R}_2(k\eta)$ is in powers of $\rho_{\ff}\rho_{\ff_0}$, that in powers of $x'$.  Indeed, since
\begin{equation}
      \widehat{R}_2(k\eta)=(D_{\phi}+k\gamma)^{-1}-\widehat{Q}(k\eta)
\label{sus.6b}\end{equation}
and since $\widehat{Q}(k\eta)$ is already a conormal distribution with smooth expansion at $M^2_{\phi}\times [0,\infty)_k$, it clearly suffices to show that the expansion of $(D_{\phi}+k\gamma)^{-1}$ at $\ff$ is in powers of $\rho_{\ff}\rho_{\ff_0}$ instead of just $\rho_{\ff}$, a result established in Lemma~\ref{sus.7} below. 

Knowing this, we can thus take the inverse Fourier transform in $\xi$ of each term in the expansion of $\widehat{R}_2(\xi)$ at $\ff$.  Doing this, we are left with an error term with rapid decay at $\ff$.   To take the inverse Fourier transform near $\ff_0$, we can thus make the change of variable \eqref{sus.6c} again and invoke Lemma~\ref{sus.11} below to show it is of the desired form.

Still, there could be a problem while taking the inverse Fourier transform of each term in the expansion of $\widehat{R}_2(\xi)$ at $\ff$.  Indeed, in principle the expansion in $|\xi|$ would yield an expansion at the boundary hypersurface created by the blow-up of the lift of $\Phi\times \overline{\rho\bbR^q}$ in $\widetilde{M}^2_{\phi-\sus(V_\rho)}$.  The fact that we do not need to perform this blow-up to have a polyhomogeneous conormal distribution comes from the fact that the expansion in $|\xi|$ is in fact smooth in $\xi$, so when we take the inverse Fourier transform, this ensures rapid decay at this extra-blown-up face.  
To see this smoothness in the expansion at $\xi=0$, notice that by Lemma~\ref{sus.7b} below, each term in the expansion of $(D_{\phi}+k\gamma)^{-1}$ at $\ff$ has a smooth expansion in $\frac{\xi}{\rho_{\fbf_0}}$, not just $\frac{|\xi|}{\rho_{\fbf_0}}$, at $\ff_0\cap \ff$, so that by \eqref{sus.6b}, the same holds for the terms in the expansion of $\widehat{R}(\xi)$ at $\ff$.

\end{proof}

\begin{lemma}
The expansion of $G_{k,\phi}=(D_{\phi}+k\gamma)^{-1}$ at $\ff$ in Theorem~\ref{le.46} can be described in terms of powers of $\rho_{\ff}\rho_{\ff_0}\rho_{\fbf_0}$ for $\rho_{\ff}$, $\rho_{\ff_0}$ and $\rho_{\fbf_0}$ boundary defining functions for $\ff$, $\ff_0$ and $\fbf_0$.
\label{sus.7}\end{lemma}
\begin{proof}
According to \eqref{le.42}, the top order term in the expansion of $G_{k,\phi}$ at $\ff$ is given by 
\begin{equation}
   N_{\ff}(D_{k,\phi})^{-1}= (D_v+\eth_h+ \gamma k)(D_v^2+ \eth_h^2+k^2)^{-1}.
\label{sus.8}\end{equation}
It has a term of order $h$ at $\ff\cap \fbf_0$   involving only the part \eqref{le.43} of the normal operator acting on sections of $\ker D_v$. Hence, the model \eqref{sus.8} can be extended smoothly off $\ff$ to an operator $Q_0\in\Psi^{-1,\cQ_0}_{k,\phi}(M;E)$, where for $j\in\bbN_0$, $\cQ_j$ corresponds to the index family such that
$$
      \cQ_j|_{\ff}=\cQ_j|_{\ff_0}=\cQ_{j}|_{\fbf_0}-h =   \bbN_0+j
$$
with $\cQ_j$ given by the empty set elsewhere. This extension can be made in such a way that its expansion at $\ff$ is in powers of $\rho_{\ff}\rho_{\fbf_0}\rho_{\ff_0}$.  Then we have that   
\begin{equation}
    D_{k,\phi}Q_0= \Id +R_1,
\label{sus.9}\end{equation} 
and with $R_1\in \Psi^{0,\cR_1}_{k,\phi}(M;E)$ having also expansion at $\ff$ in powers of $\rho_{\ff}\rho_{\fbf_0}\rho_{\ff_0}$, where for $j\in \bbN$,  $\cR_j$ corresponds to the index family such that 
$$
    \cR_j|_{\ff}= \bbN_0+j, \quad \cR_j|_{\ff_0}-1=\cR_j|_{\fbf_0}-h-1=\bbN_0,
$$
and which is the empty set at all other boundary hypersurfaces of $M^2_{k,\phi}$.  Indeed, by Theorem~\ref{com.12}, the error term $R_1$ is of the claimed form. Since the expansion of $R_1$ at $\ff$ is in power of $\rho_{\ff}\rho_{\ff_0}\rho_{\fbf_0}$, notice that       $N_{\ff}(R_1\rho_{\ff}^{-1})$ has index sets $\bbN_0+1$ and $\bbN_0+h+2$ at $\ff_0$ and $\fbf_0$.   
In fact, adding successively terms of order $(\rho^1_{\ff}\rho^1_{\ff_0}\rho_{\fbf_0}^{h+j})$ for $j\in \bbN$ in the expansion of $Q_0$ at the corner $\ff\cap \fbf_0$ and taking a Borel sum of those, we can require as well that $N_{\ff}(R_1\rho_{\ff}^{-1})$ decays rapidly at this corner.  

Now, replacing $\cE_{\ff}$ and $\cF|_{\ff}$ by $0$ in Theorem~\ref{com.12} yields a composition result for Schwartz kernels on $\ff$.  This suggests to consider a term $Q_1\in\Psi^{-1,\cQ_1}_{k,\phi}(M;E)$ such that
$$
   N_{\ff}(Q_1\rho_{\ff}^{-1})= -N_{\ff}(Q_0)N_{\ff}(R_1\rho_{\ff}^{-1})
$$
and with expansion at $\ff$ in powers of $\rho_{\ff}\rho_{\fbf_0}\rho_{\ff_0}$.    With this understood, we have that
$$
      D_{k,\phi}(Q_0+Q_1)=\Id+R_2
$$
with $R_2\in\Psi^{0,\cR_2}_{k,\phi}(M;E)$ having expansion at $\ff$ in powers of $\rho_{\ff}\rho_{\ff_0}\rho_{\ff_0}$.  In particular, $N_{\ff}(R_2\rho_{\ff}^{-2})$ has index set $\bbN_0+2$ and $\bbN_{0}+h+3$ at $\ff_0\cap\ff$ and $\fbf_0\cap \ff$.  Adding successively terms of  order $(\rho^2_{\ff}\rho^2_{\ff_0}\rho_{\fbf_0}^{h+1+j})$ for $j\in\bbN$ in the expansion of $Q_1$ at the corner $\fbf_0\cap\ff$ and taking a Borel sum, we can also ensure that $N_{\ff}(R_2\rho_{\ff}^{-2})$ vanishes rapidly there.  Clearly, this construction can be iterated, so that more generally, we can define recursively $Q_\ell\in\Psi^{-1,\cQ_\ell}(M;E)$ having expansion in powers of $\rho_{\ff}\rho_{\ff_0}\rho_{\fbf_0}$ at $\ff$ such that $N_{\ff}(Q_\ell\rho^{-\ell}_{\ff})= -N_{\ff}(Q_0)N_{\ff}(R_\ell\rho^{-\ell}_{\ff})$  and 
$$
        D_{k,\phi}(\sum_{j=0}^\ell Q_j)= \Id + R_{\ell+1}
$$    
with $R_{\ell+1}\in \Psi^{0,\cR_{\ell+1}}_{k,\phi}(M;E)$ having expansion in powers of $\rho_{\ff}\rho_{\ff_0}\rho_{\fbf_0}$ at $\ff$ with $N_{\ff}(R_{\ell+1}\rho_{\ff}^{-\ell-1})$ vanishing rapidly at $\fbf_0\cap\ff$.     If $Q\in\Psi^{-1,\cQ_0}_{k,\phi}(M;E)$ is a Borel sum of the $Q_j$, then its expansion at $\ff$ is in powers of $\rho_{\ff}\rho_{\ff_0}\rho_{\fbf_0}$ and 
$$
      D_{k,\phi}Q=\Id+R
$$
for some $R\in\Psi^{0,\cR}_{k,\phi}(M;E)$ with $\cR$ the index family such that
$$
     \quad \cR|_{\ff_0}-1=\cR|_{\fbf_0}-h-1=\bbN_0
$$
and with $\cR$  given by the empty set elsewhere, in particular at $\ff$.  Since 
$$
    D_{k,\phi}^{-1}= D_{k,\phi}^{-1}(D_{k,\phi}Q-R)= Q- D_{k,\phi}^{-1}R,
$$  
we see from Theorem~\ref{com.12} that $D_{k,\phi}^{-1}$ has the same expansion as the one of $Q$ at $\ff$, from which the result follows.
   
\end{proof}

\begin{lemma}
The terms in the expansion of $G_{\xi,\phi}$ at $\ff_0$ can be decomposed into terms coming from $M^2_{\phi}\times \bbR^q$ and terms coming from $[M^2_{k,b};\Phi_+]\times \bbS^{q-1}$.
\label{sus.11}\end{lemma}
\begin{proof}
Using \eqref{le.39}, we know how to invert $D_{\xi,\phi}$ at $\fbf_0$.  This inverse makes sense on $[M^2_{k,b};\Phi_+]\times \bbS^{q-1}$, that is, before we blow up $\Phi_0$ in \eqref{kfb.9bb} to obtain $M^2_{k,\phi}\times\bbS^{q-1}$.  When lifted to $M^2_{k,\phi}\times \bbS^{q-1}$, it induces on $\ff_0$ the part of the inverse of $N_{\ff_0}(D_{\xi,\phi})$ in the range of $\Pi_h$.  The part of the inverse of $N_{\ff_0}(D_{\xi,\phi})$ off this range is a family of suspended operators in the usual sense, so decaying rapidly on $\fbf_0$.  Moreover, the full inverse of $N_{\ff_0}(D_{\xi,\phi})$ does not depend on $\frac{\xi}{x}$ and descends to $M^2_{\phi}\times \bbR^q$.  Hence, let $Q_0\in\Psi^{-1,\cQ_0}_{k,\phi}(M;E)$ be a parametrix of $D_{\xi,\phi}$ obtained by extending the inverses at $\fbf_0$ and $\ff_0$ smoothly and by inverting symbolically, so that 
\begin{equation}
D_{\xi,\phi}Q_0= \Id- R_0'-R_0'',
\label{sus.12}\end{equation} 
where $R_0'\in \Psi^{-\infty,\cR_0'}_{k,\phi}(M;E)$ comes from a polyhomogeneous section on $[M^2_{k,b};\Phi_+]\times\bbS^{q-1}$, $R_0''\in\Psi^{-\infty}_{k,\phi}(M;E)$ vanishes to order one at $\ff_0$ and comes from a polyhomogeneous section on $M^2_{\phi}\times \bbR^q$, and where $\cQ_0$ is an index family such that
$$
\begin{gathered}
  \inf\Re \cQ_0|_{\zf}\ge 0, \quad \inf\Re\cQ_0|_{\ff_0}\ge 0, \quad \inf\Re \cQ_0|_{\fbf_0}\ge h, \quad \inf\Re\cQ_0|_{\lf_0}>0, \quad \inf\Re\cQ_0|_{\rf_0}>h+1, \\
  \cQ_0|_{\ff}=\bbN_0, \quad \cQ_0|_{\lf}=\cQ_0|_{\rf}=\cQ_0|_{\fbf}=\emptyset,
\end{gathered}  
$$
while $\cR_0'$ is an index family such that
$$
\begin{gathered}
  \inf\Re \cR_0'|_{\zf}\ge 0, \quad \inf\Re\cR_0'|_{\ff_0}>0, \quad \inf\Re \cR_0'|_{\fbf_0}>h+1, \quad \inf\Re\cR_0'|_{\lf_0}>0, \quad \inf\Re\cR_0'|_{\rf_0}>h+1, \\
  \cR_0|_{\ff}=\bbN_0, \quad \cR_0'|_{\lf}=\cR_0'|_{\rf}=\cR_0'|_{\fbf}=\emptyset.
\end{gathered}  
$$

  Essentially, the term $R_0'$ is the error term created by the inversion at $\fbf_0$ and the inversion on $\ff_0$ in the range of $\Pi_h$, while $R_0''$ is the error term created by the inversion at $\ff_0$ off the range of $\Pi_h$ and the symbolic inversion.  Extending smoothly 
$$
   N_{\ff_0}(Q_1):= N_{\ff_0}(Q_0)N_{\ff_0}(R_0'')
$$   
off $\ff\times \{0\}$ in $M^2_{\phi}\times \bbR^q$, we obtain an operator $Q_1\in\Psi^{-\infty}_{k,\phi}(M;E)$ coming from a smooth section on $M^2\times [0,\infty)$ such that 
$$
      D_{\xi,\phi}(Q_0+Q_1)=\Id-R_1' -R_1''
$$
with $R_1'$ and $R''$ satisfying respectively the same properties as those of $R_0'$ and $R_0''$, but with $R_0''$ vanishing to order $2$ at $\ff_0$.  Proceeding recursively, we can more generally construct $Q_i\in\Psi^{-\infty}_{k,\phi}(M;E)$ coming from a smooth section on $M^2_{\phi}\times \bbR^q$ such that
$$
    D_{\xi,\phi}\left(\sum_{j=0}^{i}Q_j \right)=\Id-R_i'-R_i''
$$
with $R_i'$ and $R_i''$ satisfying the same properties as $R_0'$ and $R_0''$, but with $R_i''$ vanishing to order $j$ at $\ff_0$.  Taking a Borel sum 
$$
   Q_{\infty}\sim \sum_{j=0}^{\infty} Q_j
$$ 
at $\ff_0$ gives a a parametrix $Q_{\infty}\in \Psi^{-1,\cQ_0}_{k,\phi}(M;E)$ such that
$$
   D_{\xi,\phi}Q_{\infty}=\Id-R_{\infty}
$$
with $R_{\infty}\in\Psi^{-\infty,\cR_0}_{k\phi}(M;E)$ satisfying the same properties as $R_0$.  Proceeding as in the proof of Proposition~\ref{le.40}, we can also remove the expansion of $R_{\infty}$ at the boundary hypersurface of $[M^2_{k,b};\Phi_+]$ that lifts to $\fbf_0$ on $M^2_{k,\phi}$ to get a new parametrix $Q\in\Psi^{-1,\cQ}_{k,\phi}(M;E)$ with $\cQ$ satisfying the same properties as $\cQ_0$ and such that
\begin{equation}
  D_{\xi,\phi}Q=\Id-R
\label{sus.13}\end{equation}  
for some $R\in\Psi^{-\infty,\cR}_{k,\phi}(M;E)$ with $\cR$ satisfying the same properties as $\cR_0'$, but with $\cR|_{\ff_0}=\cR|_{\fbf_0}=\emptyset$.  In this construction, notice that we can write
$$
  Q=Q'+Q''
$$
with $Q'$ coming from a conormal distribution on $[M^2_{k,b};\Phi_+]\times \bbS^{q-1}$ and $Q''$ coming from a conormal distribution on $M^2_{\phi}\times \bbR^q$.  On the other hand, by \eqref{sus.13}, 
$$
    G_{\xi,\phi}=G_{\xi,\phi}\Id= G_{\xi,\phi}(D_{\xi,\phi}Q+R)= Q+ G_{\xi,\phi}R.
$$
Since $R$ decays rapidly on $\fbf_0$ and $\ff_0$, using the fact that in \eqref{com.12b},  the terms in the expansion at $\ff_0$ coming from $(\cE|_{\lf_0}+\cF|_{\rf_0})$ correspond to terms polyhomogeneous  on $[M^2_{k,b};\Phi_+]$, we see that $G_{\xi,\phi}$ has the same expansion as $Q$ at $\ff_0$ modulo terms coming from $[M^2_{k,b};\Phi_+]\times\bbS^{q-1}$, from which the result follows.

\end{proof}

\begin{lemma}
Suppose that $h>1$.  Then each term of the expansion of $G_{\xi,\phi}=(D_{\phi}+i\cl(\xi))^{-1}$ at $\ff$ has a smooth expansion in powers of $\frac{\xi}{\rho_{\fbf_0}}$ at  $\ff\cap\ff_0$, not just in powers of $\frac{k}{\rho_{\fbf_0}}=\frac{|\xi|}{\rho_{\fbf_0}}$.
\label{sus.7b}\end{lemma}
\begin{proof}
Let us start by showing that $N_{\ff}(D_{k,\phi})^{-1}$ has a smooth expansion in $\frac{\xi}{\rho_{\fbf_0}}$ at $\ff\cap \ff_0$.  For this purpose, we can decompose the action of $N_{\ff}(D_{k,\phi})^{-1}$ with respect to the decomposition $\ker D_v\oplus \ker D_v^{\perp}$ in the fibers of $\phi: \pa M\to Y$.  Clearly, the part acting on $(\ker D_v)^{\perp}$ is smooth in $\xi$ and it even descends to $\ff\times \bbR^q_{\xi}$ in $M^2_{\phi}\times \bbR^q_{\xi}$.  

For the part acting on $\ker D_v$, it is given by 
\begin{equation}
  (\eth_h+i\cl(\xi))^{-1}= (\eth_h+i\cl(\xi))(\eth_h^2+|\xi|^2)^{-1}.
\label{sus.10}\end{equation} 
Now, $\eth^2_{h}= |\xi|^2\Delta_h$, where $\Delta_h$ can be seen as a family of Euclidean Laplacian on the fibers of the vector bundle ${}^{k,\phi}N_{\sc}Y$ introduced in \eqref{la.5}.  Taking the Fourier transform in the fibers of this vector bundles, the operator $(\eth_h^2)+ |\xi|^2)$ thus becomes
$$
    |\xi|^2(|\zeta|^2_{\sigma_2(\Delta_h)})+1),
$$
where $\zeta$ denotes linear coordinates in the fibers of ${}^{k,\phi}N_{\sc}^*Y$ with norm $|\cdot|_{\sigma_2(\Delta_h)}$ induced by the principal symbol of $\Delta_h$.  The inverse is clearly given by
$$
      |\xi|^{-2}(|\zeta|^2_{\sigma_2(\Delta_h)}+1)^{-1}.
$$
Taking the inverse Fourier transform, we see that
$$
  (\eth_h^2+|\xi|^2)^{-1}=\left(|\xi|^{-2} \frac{1}{(2\pi)^{h+1}} \int e^{i\zeta\cdot kY}\frac{d\zeta}{1+|\zeta|^2_{\sigma_2(\Delta_h)}}\right) d(kY),
$$
where $Y$ denotes linear coordinates in the fibers of $V=(k^{-1}){}^{k,\phi}N_{\sc}Y$ in \eqref{sm.7}, so that $kY$ corresponds to linear coordinates in ${}^{k,\phi}N_{\sc}Y$.  Since $(1+|\zeta|^2_{\sigma_2(\Delta_h)})^{-1}$ has an expansion in even powers of $|\zeta|^{-1}_{\sigma_2(\Delta_h)}$, its inverse Fourier transform has an expansion of the form 
$$
     \frac{1}{(2\pi)^{h+1}}\int \frac{e^{i\zeta\cdot kY}}{(1+|\zeta|^2_{\sigma_2(\Delta_h)})} d\zeta\sim \frac{1}{|kY|^{h-1}_{\sigma_2(\Delta_h)}} \sum_{j=0}^{\infty} a_j|kY|^{2j}_{\sigma_2(\Delta_h)}
$$  
at $|kY|_{\sigma_2(\Delta_h)}=0$, where we have used the fact that $h>1$ to rule out the presence of a logarithmic term in the expansion. Hence, taking into account the change of density $d(kY)= k^{h+1}dY$, we see that at $|kY|_{\sigma_2(\Delta_h)}=0$, the inverse $(\eth_h^2+|\xi|^2)^{-1}$ has the expansion
$$
 (\eth_h^2+|\xi|^2)^{-1}\sim \left( \frac{1}{|Y|^{h-1}_{\sigma_2(\Delta_h)}} \sum_{j=0}^{\infty} a_j |\xi|^{2j}|Y|^{2j}_{\sigma_2(\Delta_h)} \right) dY \quad \mbox{at}  \quad k|Y|_{\sigma_2(\Delta_h)}=0.
$$
Since $|Y|^{-1}_{\sigma_2(\Delta_h)}$ is a boundary defining function for $\fbf_0$, this expansion corresponds to a smooth expansion in $\frac{\xi}{\rho_{\fbf_0}}$ at $\ff\cap \ff_0$.  Since $(\eth_h+i \cl(\xi))$ is already smooth in $\xi$, we thus see that the composite \eqref{sus.10} has a smooth expansion in powers of $\frac{\xi}{\rho_{\fbf_0}}$ at $\ff\cap \ff_0$.  

Hence, letting $Q_0\in\Psi^{-1,\cQ_0}_{k,\phi}(M;E)$ and $R_1\in\Psi^{0,\cR_1}_{k,\phi}(M;E)$ be as in the proof of Lemma~\ref{sus.7}, notice that $Q_0$ and $R_1$ can be chosen to have a smooth expansion in $\frac{\xi}{\rho_{\fbf_0}}$ at $\ff_0$.  In fact, adding terms of order $\rho_{\ff}\left(\frac{|\xi|}{\rho_{\fbf_0}} \right)^j$ for $j\in \bbN$ to $Q_0$ that are smooth in $\frac{\xi}{\rho_{\fbf_0}}$ and taking a Borel sum, we can further assume that $N_{\ff}(R_1\rho^{-1}_{\ff})$ decays rapidly at $\ff_0$ as well.  

Letting $\widetilde{\cQ_j}$ denote index family given by $\widetilde{\cQ}_j|_{\ff}=\bbN_0+j$, $\widetilde{\cQ}_j|_{\fbf_0}=\bbN_0+h+1$ and by the empty set elsewhere, we can thus take $Q_1\in \Psi^{-1,\widetilde{\cQ}_1}_{k,\phi}(M;E)$ with smooth expansion in $\frac{\xi}{\rho_{\fbf_0}}$ at $\ff_0$ such that
$$
    N_{\ff}(Q_1\rho_{\ff}^{-1})= -N_{\ff}(Q_0)N_{\ff}(R_1\rho_{\ff}^{-1})
$$  
and 
$$
     D_{k,\phi}(Q_0+Q_1)=\Id+R_2
$$
for $R_2\in\Psi^{0,\widetilde{\cR}_2}_{k,\phi}(M;E)$ having smooth expansion in powers of $\frac{\xi}{\rho_{\fbf_0}}$ at $\ff_0$, where $\widetilde{\cR}_j$ denotes the index family with
$$
     \widetilde{\cR}_j|_{\ff}=\bbN_0+j, \quad \widetilde{\cR}_j|_{\ff_0}= \widetilde{\cR}_j|_{\fbf_0}-h= \bbN_0+1
$$
and with $\widetilde{\cR}_j$ elsewhere given by the empty set. In fact, since we are not insisting on $Q_1$ having rapid decay at $\ff_0$ and $\fbf_0$,  we can assume that $N_{\ff}(R_2\rho_{\ff}^{-2})$ decays rapidly at $\ff_0\cap \ff$ by considering appropriate terms of order $\rho^2_{\ff}\left( \frac{|\xi|}{\rho_{\fbf_0}}\right)^j$ smooth in $\frac{\xi}{\rho_{\fbf_0}}$ for $j\in\bbN_0$ in the expansion of $Q_0$ at $\ff_0\cap \ff$ and taking a Borel sum of those.  Similarly, adding terms of order $\rho^2_{\ff}(\rho_{\fbf_0})^{h+j}$ for $j\in\bbN$ through a Borel sum in the expansion of $Q_1$ at $\ff\cap \fbf_0$, we can assume as well that $N_{\ff}(R_2\rho^{-2}_{\ff})$ vanishes rapidly at $\ff\cap \fbf_0$.  

Clearly, this construction can be iterated, so that one can more generally construct $Q_j\in \Psi^{-1,\widetilde{\cQ}_j}_{k,\phi}(M;E)$ with smooth expansion in powers of $\frac{\xi}{\rho_{\fbf_0}}$ at $\ff_0$ such that 
$$
    N_{\ff}(Q_j\rho^{-j})=-N_{\ff}(Q_0)N_{\ff}(R_j\rho_{\ff}^{-j})
$$
and 
$$
    D_{k,\phi}(Q_0+\cdots + Q_j)= \Id+ R_{j+1}
$$
with $R_{j+1}\in \Psi^{0,\widetilde{\cR}_{j+1}}_{k,\phi}(M;E)$ having a smooth expansion in powers of  $\frac{\xi}{\rho_{\fbf_0}}$ at $\ff_0$ and  such that $N_{\ff}(R_{j+1}\rho_{\ff}^{-j-1})$ vanishes rapidly at $\ff_0\cap \ff$ and $\fbf_0\cap \ff$.  Taking Borel sum of the $Q_j$, we thus obtain an operator $Q\in \Psi^{-1,\cQ_0}_{k,\phi}(M;E)$ such that $Q$ has a smooth expansion in powers of $\frac{\xi}{\rho_{\fbf_0}}$ at $\ff_0$ and 
$$
       D_{k,\phi} Q=\Id +R
$$
with $R\in \Psi^{0,\widetilde{\cR}}_{k,\phi}(M;E)$, where $\widetilde{\cR}$ is the index family with
$$
    \widetilde{\cR}|_{\ff_0}=\widetilde{\cR}|_{\fbf_0}-h= \bbN_0+1
$$
and with $\widetilde{\cR}$ given by the empty set elsewhere.  Since
$$
     D_{k,\phi}^{-1}= D_{k,\phi}^{-1}(D_{k,\phi}Q-R)= Q- D_{k,\phi}R,
$$
we see from Theorem~\ref{com.12} that $Q$ and $D_{k,\phi}^{-1}$ have the same expansion at $\ff$, from which the result follows.
\end{proof}

\appendix

\section{Blow-ups in manifolds with corners} \label{cbu.0}

In this appendix, we will establish the commutativity of blow-ups of two $p$-submanifolds used in Lemma~\ref{kfb.9b} to show that our two ways of constructing the $k,\phi$-double space are equivalent.  Indeed, this result definitely requires a proof, especially since it does not seem to follow from standard results like the commutativity of nested blow-ups or the commutativity of blow-ups of transversal $p$-submanifolds.

\begin{lemma}
Let $W$ be a manifold with corners.  Suppose that $X$ and $Y$ are two $p$-submanifolds such that their intersection 
$Z= X\cap Y$ is also a $p$-submanifold with the property that for every $w\in Z$, there is a coordinate chart 
\begin{equation}
   \varphi: \cU\to \bbR^{n_1}_{k_1}\times \bbR^{n_2}_{k_2}\times \bbR^{n_3}_{k_3}\times \bbR^{n_4}_{k_4}
\label{kqfb.7a}\end{equation} 
sending $w$ to the origin such that 
\begin{equation}
\begin{aligned}
\varphi(\cU\cap X)&= \{0\}\times \{0\}\times \bbR^{n_3}_{k_3}\times \bbR^{n_4}_{k_4}, \\
\varphi(\cU\cap Y)&= \{0\}\times \bbR^{n_2}_{k_2}\times \{0\}\times \bbR^{n_4}_{k_4}, \\
\varphi(\cU\cap Z)&= \{0\}\times\{0\}\times \{0\}\times \bbR^{n_4}_{k_4}. 
\end{aligned}
\label{kqfb.7b}\end{equation}
Then the identity map in the interior extends to a diffeomorphism 
\begin{equation}
     [W;X,Y,Z]\to [W;Y,X,Z].
\label{kqfb.7c}\end{equation}
\label{kqfb.7}\end{lemma}
\begin{proof}
Since we blow up $Z$ last, notice first that this result does not quite follows from the commutativity of nested blow-ups.  Now, clearly, away from $Z$, the blow-ups of $X$ and $Y$ commute since they do not intersect.  Thus, to establish \eqref{kqfb.7c}, it suffices to establish it near $Z$.  Let $w\in Z$ be given and consider a coordinate chart $(\cU,\varphi)$ as in \eqref{kqfb.7b}.  Let $x=(x_1,\ldots, x_{n_1})$, $y=(y_1,\ldots,y_{n_2})$, $z=(z_1,\ldots,z_{n_3})$ and $w=(w_1,\ldots,w_{n_4})$ be the canonical coordinates for the factors $\bbR^{n_1}_{k_1}$, $\bbR^{n_2}_{k_2}$, $\bbR^{n_3}_{k_3}$ and $\bbR^{n_4}_{k_4}$ respectively.

When we blow up $X$, this coordinate chart is replaced by the one induced by the coordinates
$$
 (\omega_{x,y}=(\frac{x}r,\frac{y}r), r=\sqrt{|x|^2+|y^2|},z,w)\in \bbS^{n_1+n_2-1}_{k_1+k_2}\times [0,\infty)_r\times \bbR^{n_3}_{k_3}\times \bbR^{n_4}_{k_4}.
$$
In this coordinate chart, the lifts of $Y$ and $Z$ corresponds to 
$$
    (\{0\}\times \bbS^{n_2-1}_{k_2})\times [0,\infty)_r\times \{0\}\times \bbR^{n_4}_{k_4}
$$
and
$$
   \bbS^{n_1+n_2-1}_{k_1+k_2}\times \{0\}\times\{0\} \times \bbR^{n_4}_{k_4},
$$
where $\{0\}\times \bbS^{n_2-1}_{k_2}\subset \bbR^{n_1}_{k_1}\times \bbR^{n_2}_{k_2}$ is seen as a $p$-submanifold of 
$\bbS^{n_1+n_2-1}_{k_1+k_2}$ seen as the unit sphere in $\bbR^{n_1}_{k_1}\times \bbR^{n_2}_{k_2}$.  To blow up $Y$, this suggests to consider the smaller coordinate chart induced by the coordinates
$$
  (x,\omega_y=\frac{y}{|y|},r,z,w)\in \bbR^{n_1}_{k_1}\times \bbS^{n_2-1}_{k_2}\times [0,\infty)_{r}\times \bbR^{n_3}_{k_3}\times \bbR^{n_4}_{k_4}
$$  
in which the lift of $Y$ corresponds to 
$$
\{0\}\times \bbS^{n_2-1}_{k_2}\times [0,\infty)_r\times \{0\}\times \bbR^{n_4}_{k_4}
$$
and the lift of $Z$ to 
$$
\bbR^{n_1}_{k_1}\times \bbS^{n_2-1}_{k_2}\times \{0\}\times \{0\}\times \bbR^{n_4}_{k_4}.
$$  
Hence, blowing up $Y$, we obtain a coordinate chart on $[W;X,Y]$ by considering the one induced by the coordinates
$$
   (\omega_{x,z}=(\frac{x}{\rho},\frac{z}{\rho}), \rho=\sqrt{|x|^2+|z|^2}, \omega_y=\frac{y}{|y|},r=\sqrt{|x|^2+|y|^2},w)\in
   \bbS^{n_1+n_3-1}_{k_1+k_3}\times [0,\infty)_{\rho} \times \bbS^{n_2-1}_{k_2}\times [0,\infty)_r\times \bbR^{n_4}_{k_4}
$$
in which the lift of $Z$ corresponds to 
$$
   (\bbS^{n_1-1}_{k_1}\times \{0\})\times [0,\infty)_{\rho}\times \bbS^{n_2-1}_{k_2}\times\{0\}\times \bbR^{n_4}_{k_4}
$$
with $\bbS^{n_1-1}_{k_1}\times \{0\}\subset \bbR^{n_1}_{k_1}\times \bbR^{n_3}_{k_3}$ seen as a $p$-submanifold of $\bbS^{n_1+n_3-1}_{k_1+k_3}$.  To blow up $Z$, this suggests to consider the coordinate charts induced by the coordinates 
$$
(z,\omega_x=\frac{x}{|x|}, \rho=\sqrt{|x|^2+|z|^2}, \omega_y=\frac{y}{|y|}, r=\sqrt{|x|^2+|y|^2},w)\in\bbR^{n_3}_{k_{3}}\times \bbS^{n_1-1}_{k_1}\times [0,\infty)_{\rho}\times \bbS^{n_2-1}_{k_2}\times [0,\infty)_{r}\times \bbR^{n_4}_{k_4}
$$ 
in which the lift of $Z$ corresponds to 
$$
  \{0\}\times \bbS^{n_1-1}_{k_1}\times [0,\infty)_{\rho}\times \bbS^{n_2-1}_{k_2}\times \{0\}\times \bbR^{n_4}_{k_4}.
$$
Hence, blowing up the lift of $Z$, we see that $[W;X,Y,Z]$ admits a coordinate chart induced by the coordinates
\begin{multline}
(\omega_{z,r}=(\frac{z}{s},\frac{r}{s}), s= \sqrt{|z|^2+r^2}, \omega_x= \frac{x}{|x|}, \omega_y=\frac{y}{|y|}, \rho=\sqrt{|x|^2+ |y|^2},w) \\\in \bbS^{n_3}_{k_3+1}\times [0,\infty)_{s}\times \bbS^{n_1-1}_{k_1}\times \bbS^{n_2-1}_{k_2}\times [0,\infty)_{\rho}\times \bbR^{n_4}_{k_4}.
\end{multline}
In this chart, $Z$ lifts to 
$$
    \bbS^{n_3}_{k_3+1}\times \{0\}\times \bbS^{n_1-1}_{k_1}\times \bbS^{n_2-1}_{k_2}\times [0,\infty)_{\rho}\times \bbR^{n_4}_{k_4},
$$
$Y$ lifts to
$$
\bbS^{n_3}_{k_3+1}\times [0,\infty)_s\times \bbS^{n_1-1}_{k_1}\times \bbS^{n_2-1}_{k_2}\times \{0\}\times \bbR^{n_4}_{k_4}
$$
and $X$ lifts to
$$
 (\bbS^{n_3-1}_{k_3}\times \{0\})\times [0,\infty)_s\times \bbS^{n_1-1}_{k_1}\times \bbS^{n_2-1}_{k_2}\times[0,\infty)_{\rho}\times \bbR^{n_4}_{k_4},
$$
where $\bbS^{n_3-1}_{k_3}\times \{0\}\subset \bbR^{n_3}_{k_3}\times [0,\infty)_r$ is seen as a $p$-submanifold of the unit sphere $\bbS^{n_3}_{k_3+1}\subset \bbR^{n_3}_{k_3}\times [0,\infty)_{r}$.  

Since we are interested in the commutativity of the blow-ups of $X$ and $Y$ when the blow-up of $Z$ is subsequently performed, we can consider instead a smaller coordinate chart on $[W;X,Y,Z]$ in a neighborhood of the lift of $X$ which is induced by the coordinates
\begin{multline}
  (\omega_x=\frac{x}{|x|}, \omega_y=\frac{y}{|y|}, \omega_z=\frac{z}{|z|}, r=\sqrt{|x|^2+|y|^2}, \rho=\sqrt{|x|^2+|z|^2},s=\sqrt{|x|^2+|y|^2+|z|^2},w) \\
  \in \bbS^{n_1-1}_{k_1}\times \bbS^{n_2-1}_{k_2}\times\bbS^{n_3-1}_{k_3}\times [0,\infty)_{r}\times [0,\infty)_{\rho}\times [0,\infty)_{s}\times \bbR^{n_4}_{k_4}.
\label{kqfb.8}\end{multline}
This chart is defined near the intersection of the lifts of $X$, $Y$ and $Z$, which corresponds to $\bbS^{n_1-1}_{k_1}\times \bbS^{n_2-1}_{k_2}\times\bbS^{n_3-1}_{k_3}\times \{0\}\times \{0\}\times \{0\}\times \bbR^{n_4}_{k_4}$.
 
The definition of this system of coordinates is symmetric with respect to $X$ and $Y$, namely, considering instead $[W;Y,X,Z]$, we would have obtain the same coordinate system valid near the intersection of the lifts of $X,Y$ and $Z$.  This indicates that in this region, the identity map in the interior naturally extends to a diffeomorphism.  Since this clear elsewhere, the result follows. 
 
\end{proof}

\bibliography{QFBop}
\bibliographystyle{amsplain}

\end{document}